\documentclass[11pt, reqno]{amsart}    	% use "amsart" instead of "article" for AMSLaTeX format
\usepackage[margin=1in]{geometry}
\usepackage{setspace}
\usepackage{graphicx}	
\usepackage{subcaption}
\usepackage{float}
\usepackage[toc]{appendix}

\usepackage[backref=page]{hyperref}
\hypersetup{%
	colorlinks,
	linkcolor={red!50!black},
	citecolor={green!50!black},
	urlcolor={blue!50!black}
}
\usepackage{amsmath}%
\usepackage{}
\usepackage{amsthm}
\usepackage{amsfonts}%
\usepackage{amssymb}%
\usepackage{graphicx}
\usepackage{mathrsfs}
\usepackage{lineno}
\usepackage{bm}
\usepackage[draft]{todonotes}   % notes showed

\usepackage{enumerate}

\usepackage{paralist}	
\usepackage{amssymb}
\usepackage{amsmath, amsthm, amssymb}

\newtheorem{thm}{Theorem}[section]
\newtheorem{conj}[thm]{Conjecture}
\newtheorem{prop}[thm]{Proposition}
\newtheorem{claim}[thm]{Claim}
\newtheorem{lemma}[thm]{Lemma}

\newtheorem{fact}{Fact}

\theoremstyle{definition}

\def\eps{\varepsilon}
\def\cT{\mathcal T}
\def\ex{\text{ex}}
\usepackage{tikz}
\usepackage{tikz} \usetikzlibrary{shapes,arrows}

\newcommand*{\rom}[1]{\expandafter{\romannumeral #1\relax}}

\usepackage{amsmath}
\numberwithin{equation}{section}
\begin{document}
%\linenumbers

	\title[$Y$-tilings and Hamilton $ \ell $-cycles]{Large $ Y_{k,b} $-tilings and  Hamilton $ \ell $-cycles in $k$-uniform hypergraphs }
	\author{Luyining Gan}
	\address{Department of Mathematics and Statistics, University of Nevada, Reno, Reno, NV, USA}
	\email{lgan@unr.edu}
	
	\author{Jie Han}
	\address{School of Mathematics and Statistics, Beijing Institute of Technology, Beijing, China}
	\email{\tt hanjie@bit.edu.cn}
	%	\address{Department of Mathematics, University of Rhode Island, Kingston, RI, USA.}
	%	%\thanks{JH is partially supported by Simons Collaboration Grant \#630884.}
	%	\email{jie\_han@uri.edu}
	\author{Lin Sun}
	\address{School of Mathematics, Shandong University,
		Jinan, China.}
	\email{linsun@mail.sdu.edu.cn, ghwang@sdu.edu.cn}
	
	\author{Guanghui Wang}
	%	\address{School of Mathematics, Shandong University, Jinan, China.}
	%\thanks{GHW is partially supported by the Natural Science Foundation of China  (11871311, 11631014) and seed fund program for international research cooperation of Shandong University.}
	%	\email{ghwang@sdu.edu.cn}
	
	\begin{abstract}
		%Given $1\le b<k$, let $Y_{k,b}$ be the $k$-uniform hypergraph with two edges intersecting in $b$ vertices.
		%In this paper, given a $k$-graph $H$ with at least $\binom{n}{k} - \binom{n-m+1}{k} + o(n^k)$ edges, we prove that $H$ contains a $Y_{k,b}$-tiling of size $m$ for $k>b\ge 0$ and $n \ge (2(2k-b)^2+1)(k-1)m+m$.
		%%For $k>b\ge 0$ such that  $ n \ge (2(2k-b)^2+1)(k-1)m+m$, we show that any $ k $-graph $H$  with at least $  \binom{n}{k} - \binom{n-m+1}{k} + o(n^k) $ edges contains a $Y_{k,b}$-tiling of size $m$. %Note that $m\le\frac{n}{67}$.
		%Moreover, if $m\le\frac{n}{7}$, the edge condition of $Y_{3,2}$-tiling still holds.
		%%What is more, we show  the edge condition of  $Y_{3,2}$-tiling holds for any $ m\le\frac{n}{7} $. 
		%Indeed, these lead to some new results on $d$-degree Hamilton $\ell$-cycle problem. 
		%
		%We study the Dirac-type conditions for Hamilton $\ell$-cycles in $k$-uniform hypergraphs.
		%For all  $k\ge 3$, $1\leq \ell<k/2$ and $\max\{k-\ell,\ell+1\}\le d \le k-1  $, we show that the problem can be reduced to determining the edge-density condition forcing a large $Y_{k-d,2\ell-d}$-tiling, where $Y_{a,b}$ be the $a$-uniform hypergraph with two edges intersecting in $b$ vertices.
		
		Let $Y_{3,2}$ be the $3$-uniform hypergraph with two edges intersecting in two vertices.
		Our main result is that any $n$-vertex 3-uniform hypergraph with at least $\binom{n}{3} - \binom{n-m+1}{3} + o(n^3)$ edges contains a collection of $m$ vertex-disjoint copies of $Y_{3,2}$, for $m\le n/7$.
		The bound on the number of edges is asymptotically best possible. This problem generalizes the Matching Conjecture of Erd\H{o}s.
        %to $F$-tilings. 
		% an asymptotic bound on the minimum $d$-degree for the existence of Hamilton $ \ell $-cycles in $k$-graphs, where $ k-\ell $ divides $ n $. 
		We then use this result combined with the absorbing method to determine the asymptotically best possible minimum $(k-3)$-degree threshold for $\ell$-Hamiltonicity in $k$-graphs, where $k\ge 7$ is odd and $\ell=(k-1)/2$. 
		%When $ k $ is odd, $k\ge7,\ell=(k-1)/2 $ and $d=k-3$.% What is more, we show that for all $0<\alpha \le  1/7$, any  3-graph $H$  with at least $\binom{n}{3}-\binom{n-\alpha n}{3}+o(n^3)$ edges contains a $Y_{3,2}$-tiling covering more than $ 4\alpha n-o(n)$ vertices. 
		Moreover, we give related results on  $ Y_{k,b} $-tilings and  Hamilton $ \ell $-cycles with $ d $-degree for some other values of $ k,\ell,d $.
		
	\end{abstract}
	\maketitle
	\noindent {\bf Keywords:} Hypergraph; Hamilton cycle; Hypergraph regularity method; Absorbing method.
	
	\section{Introduction}
	%{Hamilton $\ell$-cycles in hypergraph}
	%Let us recall first some basic definitions of hypergraph theory. 
	%First, let us recall some basic definitions.
	Given $ k\ge2 $, a \emph{$ k $-uniform hypergraph} $H$ (in short, \emph{$ k $-graph}) is a pair $(V,E) $, where $V$ is a vertex set and $E$ is a family of $k$-element subsets of $V$. We denote the numbers of edges in $H$ by $ e(H) $.
	Given two $k$-graphs $F$ and $H$, an \emph{$ F $-tiling} in $ H $ is a subgraph of $H$ consisting of vertex-disjoint copies of $F$. 
	The number of copies of $F$ is called the \emph{size} of the $F$-tiling.
	%By the \emph{size} of an $ F $-tiling we mean the number of copies of $F$ contained in it. 
	When $F$ is a single edge, an $F$-tiling is known as a \emph{matching}. 
	%Naturally, we consider the minimum number of edges in a $ k $-graph and the condition for the existence of an $ F $-tiling with a fixed size.
	The following conjecture was proposed by Erd\H os in 1960s.
	
	%Here we prefer the notations from Extremal Set Theory. 
	%Here we prefer the notations as follows. Given a $ k $-graph $H$,
	%A $ k $-uniform family $ \mathcal{F} \subseteq\binom{[n]}{k}$ is a collection of $ k $-subsets of $ [n] $ (so it is a $ k $-graph).
	%Given a family $ \mathcal{F} $, the size of the maximum matching in $ \mathcal{F} $ is denoted by $\nu( \mathcal{F})$. 
	%the size of the maximum matching in $H$ is denoted by $\nu( H)$.
	\begin{conj} \cite{MR260599}
		\label{conj}
		Let $ n,s,k $ be three positive integers such that $ k\ge2 $ and $ n \ge k(s+1)-1 $. If $H$ is a $ k $-graph on $ n $ vertices which does not have a matching of size $s+1$, then
		\[ e(H)\le \max \left \{ \binom{k(s+1)-1}{k}, \binom{n}{k}-\binom{n-s}{k} \right \}. \]

		%	If $ \mathcal{F} \subseteq\binom{[n]}{k}$ and $ \nu( \mathcal{F})=s$ such that $ n \ge k(s+1)-1 $, then 
		%	\[ |\mathcal{F}| \le \max \left \{ \binom{k(s+1)-1}{k}, \binom{n}{k}-\binom{n-s}{k} \right \} \]holds.
	\end{conj} 
	
	The bounds in the conjecture come from two extremal constructions: a complete $k$-graph on $k(s+1)-1$ vertices and a  complete $k$-graph on $n$ vertices with a complete $k$-graph on $n-s$ vertices removed.
	The case $ s=1 $ is the classical Erd\H os--Ko--Rado Theorem \cite{Erd1974Intersection}. For $ k = 1 $ the conjecture is trivial and for $ k = 2 $ it was proved by Erd\H os and Gallai \cite{1959On}. For $ k = 3 $, it was settled by Frankl,  R\"odl and Ruci\'nski in the case $ n \ge4s $ in \cite{2012On},  and by  \L uczak and Mieczkowska for all $ n $ and $ s \ge s_0 $ in \cite{Tomasz2014On}, and finally, completely resolved by Frankl in \cite{Frankl2017On}. 
	For general $ k$, Erd\H os \cite{MR260599} himself proved the conjecture for $ n>n_{0}(k,s) $.
	Subsequent improvements on $n_0$ have been done by various authors~\cite{BOLLOB1976SETS,2012The,2013Improved,0THE} and the current state of art is $ n_0\le \frac{5}{3}sk-\frac{2}{3}s $ for sufficiently large $ s$ by Frankl and Kupavskii \cite{0THE}. Also, Frankl \cite{Frankl2017Proof} showed the conjecture for all $ n\le(s+1)(k+\varepsilon) $, where $ \varepsilon $ depends on $ k $ only.
	
	%Bollob\'as, Daykin and Erd\H os \cite{BOLLOB1976SETS} improved $ n_{0}(k,s) $ to $ 2sk^3 $, which was subsequently lowered to $ 3sk^2 $ by Huang, Loh and Sudakov \cite{2012The}. The best known bound on $ n_0 $ is $ (2k-1)s+k $ by Frankl \cite{2013Improved}.  Recently, Frankl and Kupavskii \cite{0THE} showed the conjecture for $ n\ge\frac{5}{3}sk-\frac{2}{3}s $ and $ s\ge s_{0} $, and Han \cite{Han2016PERFECT} gave an upper bound on the number of edges in Erd\H os Matching Conjecture. 
	
	\subsection{$Y_{k,b}$-tilings}
	We consider a generalization of this problem to $F$-tilings, and as the simplest case, we consider the $k$-graphs $ F $ with two (intersecting) edges.
	For $k>b\ge 0$, let $Y_{k,b}$ be the $k$-graph consisting of two edges that intersect in exactly $b$ vertices. 
	%In this paper, we consider the question for $Y_{k,b}$-tiling. 
	Note that a $Y_{k,0}$-tiling is just a matching, so we shall only consider $b\ge 1$.
	The extremal examples for Conjecture~\ref{conj} suggest the following bounds for $k$-graphs without $Y_{k,b}$-tilings of certain size, and we formulate it as a conjecture.
	\begin{conj}
		\label{conj1}
		Let $ n,s,k,b $ be positive integers such that $ k>b>0$ and $ n \ge (2k-b)(s+1)-1$. 
		If $H$ is a $ k $-graph on $ n $ vertices which does not have a $Y_{k,b}$-tiling of size $s+1$, then
		\[ 
		e(H)\le \max \left \{ \binom{(2k-b)(s+1)-1}{k}, \binom{n}{k}-\binom{n-s}{k} \right \}+o(n^k). 
		\]
	\end{conj} 
	The case $ s=0$ is an old conjecture of Erd\H{o}s \cite{MR0409246} and was resolved by Frankl and F\"uredi \cite{Peter1985Forbidding} (see Theorem~\ref{FrFu}), which  says that a $Y_{k,b}$-free $k$-graph (in which, no pair of edges intersect in exactly $b$ vertices) has $o(n^k)$ edges.  For a $ k $-graph $ F $ and $ n\in \mathbb N $, the Tur\'an number of $ F $ is defined as $ \text{ex}(n, F)=\max\{e(H):|V(H)|=n, F\nsubseteq H\} $. So the result of \cite{Peter1985Forbidding} says $ \text{ex}(n,Y_{k,b})=o(n^k) $.

	The major terms in the conjecture above come from a complete $k$-graph on $(2k-b)(s+1)-1$ vertices and a  complete $k$-graph on $n$ vertices with a complete $k$-graph on $n-s$ vertices removed.
	Moreover, in each example, one can replace the (large) independent sets by a $Y_{k,b}$-free $k$-graph.
	Thus, in the range where the second term above is significantly larger, it is reasonable to conjecture a tight bound $\binom{n}{k}-\binom{n-s}{k}+\text{ex}(n-s, Y_{k,b})$.
	
	Overall, the study of $Y_{k,b}$-tilings is natural and important, because
	\begin{itemize}
		\item the problem generalizes the classical conjecture of Erd\H{o}s on matchings;
		\item $Y_{k,b}$-free $k$-graphs were studied by Frankl and F\"uredi, which resolved a conjecture of Erd\H{o}s;
		\item as we shall see later, $Y_{k,b}$-tilings have important applications on the Hamilton cycle problems in hypergraphs.
	\end{itemize}

	Now we present our results on $Y_{k,b}$-tilings.
	First, we use induction on $s$ and prove the following result for small $s$.
	
	\begin{thm} \label{ineg}
		For $k\ge 3$ and $1\le b \le k -1$ there exists $s_0>0$ such that the following holds whenever $s \ge s_0$.
		Let $H$ be an $n$-vertex $k$-graph and $n \ge (2(2k-b)^2+1)(k-1)s+s$.
		%If $e(H) \ge \binom{n}{k} - \binom{n-t+1}{k} + d_k (n-(2k-b)(k-1))^{k-1} + \deg_{k-1} n^{k-2}(2k-b)(t-1)$,
		%If $e(H) \ge \binom{n}{k} - \binom{n-t+1}{k} + d_k n^{k-1} + \deg_{k-1} n^{k-2}(2k-b)(t-1)$,
		%If $e(H) \ge \binom{n}{k} - \binom{n-m+1}{k} + \binom{n}{k-1}(2k-b)m$,
		If 
		\[
		e(H) \ge \binom{n}{k} - \binom{n-s+1}{k} + \binom{n-1}{k-1}+ \binom{n-1}{k-2}(2k-b)s,
		\]
		then $H$ contains a $Y_{k,b}$-tiling of size $s$.
	\end{thm}
	%Note that in Theorem~\ref{ineg}, the size $s$ of $Y_{3,2}$-tiling is at most 
	Note that the largest size $s$ of a $Y_{3,2}$-tiling when Theorem~\ref{ineg} applies is 
	${n}/{67}$. 
	Our main result below verifies Conjecture~\ref{conj1} for $Y_{3,2}$-tilings of size $ \alpha {n}$ for any $0<\alpha \le  1/7$. Note that $Y_{3,2}$ was also denoted as $C_4^3$ (or $\mathcal C_4^3$) and $C_2^3$ by other authors\footnote{Under those notation they mean \emph{3-uniform loose cycles on four vertices}, or \emph{3-uniform loose cycles with two edges.}}.
	
	\begin{thm}\label{edge th2}
		For every $\alpha,\varepsilon \in (0, 1/7]$ there exists $n_0$ such that the following holds for integer $n\ge n_0$. 
		Let $H$ be a $3$-graph of order $n$ such that
		\[
		e(H)\ge\binom{n}{3}-\binom{n-\alpha n}{3}+\varepsilon n^3.
		\]
		Then $H$ contains a $Y_{3,2}$-tiling covering $ 4\alpha n$ vertices.\footnote{Theorem~\ref{edge th2} can  not be implied by Conjecture~\ref{conj} easily. Indeed, The Erd\H{o}s-Matching-Conjecture  with $ e(H)\ge \binom{n}{3} - \binom{n-s}{3} $ gives a matching of size $ s $, which after regularity, gives only a $Y_{3,2}$-tiling of size $ 3s/4 $, or equivalently, a $Y_{3,2}$-tiling covering $ 3s $ vertices, rather than $ 4s $.}
	\end{thm}	

%Clearly $Y_{3,0}$-tilings are equivalent to matchings and note that the 3-graph $Y_{3,1}$ consists of two edges sharing one vertex. 
%Theorem~\ref{ineg} is proved by induction.
%The majority of work of this paper is the proof of Theorem~\ref{edge th2} on $Y_{3,2}$-tilings.

The minimum 2-degree threshold forcing perfect $Y_{3,2}$-tilings has been studied in~\cite{Daniela2006Loose,2013Tiling} and the corresponding minimum vertex-degree threshold was determined in~\cite{2015Minimum}.
The study of $Y_{3,2}$-tilings is also motivated by its connection to Hamilton $\ell$-cycles in $k$-graphs (see e.g.~\cite{Han2015Minimum}), which we describe in the next subsection.
At last, we remark that optimal minimum-degree-type tiling results are rather rare, not to mention edge-density-type results (as far as we know, known results are essentially on matchings).

	\subsection{Hamilton cycles in $k$-graphs}
	As applications of Theorems~\ref{ineg} and~\ref{edge th2}, we prove new results on the Hamilton $\ell$-cycle problem. 
	The existence of Hamilton cycles in graphs is a fundamental problem of graph theory with a rich history. 
	The classical theorem of Dirac \cite{1952Some} states that every graph $G$ on $n\ge3$ vertices with minimum degree $\delta(G)\ge n/2$ contains a Hamilton cycle. In recent years, researchers have worked on extending this result to hypergraphs (see recent surveys, \cite{MR3727617,Vojtech2010Dirac}).  
	Given a $ k $-graph $H=(V,E)$ with a vertex set $S\in \binom{V}{d} $, 
	we denote by $ N(S) $ the family of $T\in\binom{V}{k-d} $ such that $ T\cup S\in E $ and $ \deg_H(S):=|N(S)| $. 
	The \emph{minimum $ d $-degree} $ \delta_{d}(H) $ of $H$ is the minimum of $ \deg_{H}(S) $ over all $ d $-element vertex sets $S$ in $H$.
	
	For $ 1\le \ell <k $, a $ k $-graph is called an $ \ell $-cycle if there exists a cyclic ordering of its vertices such that every edge is composed of $ k $ consecutive vertices, two (vertex wise) consecutive edges share exactly $ \ell $ vertices. If the ordering is not cyclic, we call it an $ \ell $-path and we say the first and last $ \ell $ vertices are the ends of the path. A $ k $-graph on $ n $ vertices contains a Hamilton $\ell$-cycle  if it contains an $ \ell $-cycle as a spanning subhypergraph, so $(k-\ell)\mid n$. Note that an $ \ell $-cycle on $ n $ vertices contains exactly $ n/(k-\ell) $ edges.
	
	We define the \emph{Dirac threshold} $ h_{d}^{\ell}(k,n) $ to be the smallest integer $ h $ such that every	$ n $-vertex $ k $-graph $ H $ satisfying $ \delta_{d}(H)\ge h $ contains a Hamilton $ \ell $-cycle. Let $ h_{d}^{\ell}(k):=\limsup_{n\to \infty}h_{d}^{\ell}(k,n)/\binom{n}{k-d} $. 
	Confirming a conjecture of Katona and Kierstead \cite{2006Hamiltonian},  R\"odl, Ruci\'nski, and Szemer\'edi \cite{2008An,VOJTECH2006A} determined $ h_{k-1}^{k-1}(k) =1/2$ for any fixed $ k $, that is, $\delta_{k-1}(H)\ge (1/2+o(1))n$ guarantees $(k-1)$-Hamiltonicity. 
	The asymptotical Dirac threshold $ h_{k-1}^{\ell}(k)$ for any $ 1 \le \ell < k $ such that $(k-\ell)\mid k$ follows as a consequence of this result and a construction of Markstr\"om and Ruci\'nski. 
	When $(k-\ell)\nmid k$, the threshold $ h_{k-1}^{\ell}(k) $ was determined through a series of works  \cite{Daniela2006Loose, 2010Loose, Hi2010Dirac, Daniela2010Hamilton}.
	%by K\"uhn and	 Osthus \cite{Daniela2006Loose}, Keevash, K\"uhn, Mycroft and Osthus  \cite{2010Loose}, H\`an and Schacht \cite{Hi2010Dirac} and K\"uhn, Mycroft and Osthus  \cite{Daniela2010Hamilton}. 
	We collect these results in the following theorem.
	
	%\begin{thm}\cite{Hi2010Dirac,2010Loose,Daniela2010Hamilton,VOJTECH2006A,2008An} 
	%For any $ k\ge3, 1 \le \ell < k, \eta> 0,$  there exists $n_{0} $ such that $n\ge n_{0}$. If $(k-\ell) \mid n $ and $ H $ is a $n$-vertex $ k $-graph satisfying		
	%			\[\delta_{k-1}(H)\ge 
	%			\begin{cases}
	%			(\frac{1}{2}+\eta )n & {(k-\ell) \mid k}\\
	%			\Big(\frac{1}{\lceil \frac{k}{k-\ell}\rceil (k-\ell)}+\eta \Big) n & {(k-\ell) \nmid k},\\
	%			\end{cases}
	%			\]
	%			then $H$ contains a Hamilton $\ell$-cycle.
	%		\end{thm}
	\begin{thm}\cite{Hi2010Dirac,2010Loose,Daniela2010Hamilton,VOJTECH2006A,2008An} 
		For any $k> \ell \ge 1$, we have	
		\[h_{k-1}^\ell(k)= 
		\begin{cases}
			1/2 & {(k-\ell) \mid k}\\
			\frac{1}{\lceil \frac{k}{k-\ell}\rceil (k-\ell)} & {(k-\ell) \nmid k}.
		\end{cases}
		\]
	\end{thm}
	
	Some exact thresholds $ h_{k-1}^{\ell}(k,n) $ are known: for $ k =3$ and $\ell= 2$~\cite{2011Dirac} and for $ k \ge 3  $ and $ 1\le \ell < k/2 $~\cite{CZYGRINOW2014TIGHT,Jie2015Minimum}. 
	For $d=k-2$, the exact thresholds are known for $1\le \ell<k/2$~\cite{Bu2013Minimum, Han2015Minimum,Bastos2016Loose,J2017Loose}, but let us only state the asymptotical threshold $h_{k-2}^\ell(k)$ here. 
	%Bu\ss, H\`an and Schacht \cite{Bu2013Minimum} showed that $h_{2}^{1}(3,n)=(\frac{7}{16}+\eta) n$. 
	%Han and Zhao \cite{Han2015Minimum} showed the exact result for this.
	%Bastos, Mota, Schacht, Schnitzer and Schulenburg \cite{Bastos2016Loose,J2017Loose} determined the exact Dirac threshold of Hamilton $ \ell $-cycle in $ k $-graph for $ d=k-2 $,  which generalizes the previous results for 3-graph. We give the approximate version as follows.
	
	\begin{thm} \cite{Bu2013Minimum, Bastos2016Loose}
		For integers $k\ge 3$ and $ 1\le\ell < k/2$, we have 
		\[
		h_{k-2}^\ell(k) = 1-\left(1-\frac{1}{2(k-\ell)} \right)^{2}.
		\]
	\end{thm}
	
	%When $\ell \ge k/2$, the problem for $d=k-2$ is open. 
	Recently, Polcyn, Reiher, R\"odl, Ruci\'nski, Schacht and Szemer\'edi \cite{Reiher2019Minimum,Polcyn2020Minimum}  showed the asymptotical Dirac threshold of tight Hamilton cycle in 3-graph and 4-graph, that is, $h_{1}^{2}(3)=h_{2}^{3}(4)=5/9$. Lang and Sanhueza-Matamala	 \cite{2020Minimum} proved that $ h_{k-2}^{k-1}(k)=5/9$ for all $ k \ge 3 $ (the same	result was also proved independently by Polcyn,  Reiher, R\"odl, and Sch\"ulke \cite{2020On}) and also provided a general upper bound of $ 1-1/(2(k-d)) $ for $ h_{d}^{k-1}(k)$, narrowing the gap to the lower bound of $ 1-1/\sqrt{k-d}$ due to Han and Zhao \cite{2016Forbidding}.
	For $ d\le k-3 $, H\`an, Han and Zhao \cite{han2021minimum} determined the exact $ h_{d}^{k/2}(k,n) $ for any even integer $ k \ge 6 $, integer $ d $ such that  $ k/2 \le d \le k-1 $. 
	
	In this paper we determine more thresholds for $d\le k-3$ and $\ell<k/2$. 
	\begin{thm} \label{thm1}
		Suppose that $k\ge3$, $k-\ell\le d < 2\ell \le k-1$ such that $2k-2\ell \ge (2(2k-2\ell-d)^2+1)(k-d-1)+1$ or suppose that $ k $ is odd, $ k\ge7,\ell=(k-1)/2 $  and  $d=k-3$, then 
		\[
		h_{d}^{\ell}(k) = 1-\left(1-\frac{1}{2(k-\ell)}\right)^{k-d}.
		\] 
	\end{thm}
	
	The proof of Theorem~\ref{thm1} relies on the results in Theorems~\ref{ineg} and~\ref{edge th2}.
	Indeed, under the popular framework of ``absorbing method'', one key step is to find a constant number of vertex-disjoint $\ell$-paths whose union covers almost all the vertices.
	To obtain a large path cover one usually uses the regularity method and reduces it to finding an almost perfect  $Y_{k,2\ell}$-tiling\footnote{It is possible to use $F$-tilings for other appropriate $k$-graph $F$ to produce the path cover.} in the reduced $k$-graph.
	In our proof, we further reduce the problem by fractional matching-covering argument to finding a large (fractional) $Y_{k-d,b-d}$-tiling of given size, so that we could apply Theorems~\ref{ineg} and~\ref{edge th2}.
	
	The other two important components of the absorbing method are the absorbing lemma and the connecting lemma.
	It has been observed in~\cite{Han2021Non} that the regularity approach used by K\"uhn, Mycroft and Osthus~\cite{Daniela2010Hamilton}, where they derived an absorbing lemma and a connecting lemma for $d=k-1$, can be adapted to prove a connecting lemma for $d\ge\ell+1$.
	Here we further note that the approach in~\cite{Daniela2010Hamilton} can actually establish the absorbing lemma for $d\ge \max\{k-\ell,\ell+1\}$.
	Therefore, when $\ell<k/2$ and $d\ge \max\{k-\ell,\ell+1\}=k-\ell$, it suffices to study the path cover problem via an almost perfect $Y_{k,2\ell}$-tiling.
	
	To formulate this, let us introduce the following notation.
	For any $ 0<\varepsilon<\eta $, $k\ge3$, $1\leq \ell<k/2$ and $ 1\le d \le k-1 $, let $ t_{d}^{\ell}(k,n,\varepsilon) $ denote
	the minimum $ t $ such that every $k$-graph $H$ of order $n$ with $\delta_d(H)\ge t$ contains a $Y_{k,2\ell}$-tiling covering all but at most $ \varepsilon n $ vertices of $ H $.  Let \[ t(k,d,\ell):=\limsup_{\varepsilon\to 0} \limsup_{n\to \infty} \frac{t_{d}^{\ell}(k,n,\varepsilon)}{\binom{n}{k-d}}.\] 
	%The following is one of our main results, which gives an asymptotic bound on the minimum $ d $-degree for the existence of Hamilton $ \ell $-cycles in $k$-graphs. 
	We have the following result.
	
	\begin{thm}\label{thm} 
		For all $k\ge3$, $1\leq \ell<k/2$,  $k-\ell\le d \le k-1$, we have $h_{d}^{\ell}(k)\le t(k,d,\ell)$.
	\end{thm}

	Note that Theorem~\ref{thm} implies that $ h_{d}^{\ell}(k,n)\le(t(k,d,\ell)+\eta) \binom{n}{k-d}$.
	The following proposition, usually called space barriers, implies that $h_{d}^{\ell}(k)\ge1-(1-\frac{1}{2(k-\ell)})^{k-d}$.
	%These provide a range of $h_{d}^{\ell}(k)$, that is, $ 1-(1-\frac{1}{2(k-\ell)})^{k-d}\le h_{d}^{\ell}(k)\le t(k,d,\ell)$.
	
	\begin{prop} \label{pr}
		For all $k\ge3$, $1\leq \ell<k/2$, $ 1\leq d\le k-1 $ and every $ n$ with $(k-\ell)\mid n$, there exists an $n$-vertex $ k $-graph $H_0$ with $\delta_d(H_0)\ge(1-(1-\frac{1}{2(k-\ell)})^{k-d}-o(1)) \binom{n}{k-d}$ containing no Hamilton $\ell$-cycle. In other words, $h_{d}^{\ell}(k)\ge1-(1-\frac{1}{2(k-\ell)})^{k-d}$.
	\end{prop}
	
	\begin{proof}
		Let $ t=n/(k-\ell) $.
		%and $ s=\lceil\frac{k}{k-\ell}\rceil=2 $. 
		Let $H_0=(V,E)$ be an $ n $-vertex $ k $-graph in which $ V $ is partitioned into two sets $ A $ and $ B $ such that $ |A|=\lceil t/2\rceil-1 $ and $ E $ consists of all $ k $-sets that intersect $ A $. Thus 
		\[\delta_d(H_0)\ge\binom{n-d}{k-d}-\binom{n-|A|-d}{k-d}=(1-(1-\frac{1}{2(k-\ell)})^{k-d}-o(1)) \binom{n}{k-d}.\] 
		If $H_0$ contains a Hamilton $\ell$-cycle $ C $, then each vertex is contained in at most two edges of $ C $. 
		Since $ A $ is a vertex cover of $ C $,  we have $ |C|\le 2|A|<t $, which is a contradiction.
	\end{proof}
	
	In the proof above, if we decrease the size of $A$ by $o(n)$, then the size of a maximum $Y_{k,2\ell}$-tiling (which equals $|A|$) also decreases by the same amount, giving that $t(k,d,\ell)\ge1-(1-\frac{1}{2(k-\ell)})^{k-d}$.
	Known evidences help us conjecture that equalities should hold.
	\begin{conj}\label{conj2} 
		For all $k\ge3$, $1\leq \ell<k/2$,  $k-\ell\le d \le k-1$, we have $h_{d}^{\ell}(k)= t(k,d,\ell)=1-(1-\frac{1}{2(k-\ell)})^{k-d}$.
	\end{conj}
	
	By Theorem~\ref{thm}, to verify Conjecture~\ref{conj2} it suffices to prove that $t(k,d,\ell)\le 1-(1-\frac{1}{2(k-\ell)})^{k-d}$, which, by our approach (the fractional matching-covering reduction presented later), would follow from the validation of a range in Conjecture~\ref{conj1} (see Theorem~\ref{tran}).

	%	\begin{prop}For all $k\ge3$, $1\leq \ell,d\leq k-1$ and every $ n $	with $(k-\ell)\mid n$, there exists a $ k $-graph on $ n $ vertices with $\delta_d(H)\ge(1-(1-\frac{1}{\lceil\frac{k}{k-\ell}\rceil(k-\ell)})^{k-d}-o(1)) \binom{n}{k-d}$ does not contain a Hamilton $\ell$-cycle.\end{prop}
	%\begin{proof}
	%	Let $ t=n/(k-\ell) $ and $ s=\lceil\frac{k}{k-\ell}\rceil $. Let $H_0=(V,E)$ be an $ n $-vertex $ k $-graph in which $ V $ is partitioned into sets $ A $ and $ B $ such that $ |A|=\lceil t/s\rceil-1 $. $ E $ consists of all $ k $-sets that intersect $ A $. Then $\delta_d(H)\ge\binom{n-d}{k-d}-\binom{n-|A|-d}{k-d}=(1-(1-\frac{1}{\lceil\frac{k}{k-\ell}\rceil(k-\ell)})^{k-d}-o(1)) \binom{n}{k-d}$. If $H_0$ contains a Hamilton $\ell$-cycle $ C $, then each vertex is contained in at most $s$ edges of $ C $. Since $ A $ is a vertex cover of $ C $, $ |C|\le |A|s<t $, a contradiction.
	%\end{proof}
	%{Notations}
	
	The rest of this paper is organized as follows. 
	We use the absorbing method to prove Theorem~\ref{thm} in Section $2$. 
	The hypergraph regularity lemma is introduced in Section $3$, which is used to prove the absorbing path lemma and the path cover lemma in Appendix A and Section $4$ respectively. 
	In Sections $5$ and $6$, we prove Theorems~\ref{ineg} and~\ref{edge th2}, and use them to derive Theorem~\ref{thm1}.
	
	\section{Proof of Theorem~\ref{thm}}
	
	\subsection{Auxiliary lemmas}
	We build the Hamilton $\ell$-cycle by the absorbing method, popularized by R\"odl, Ruci\'nski, and Szemer\'edi in \cite{VOJTECH2006A}. More precisely, we divide the proof of Theorem~\ref{thm} into the following lemmas: the connecting lemma (Lemma~\ref{con lem}), the absorbing path lemma (Lemma~\ref{abs lem}), the path cover lemma (Lemma~\ref{pat lem}), and the reservoir lemma (Lemma~\ref{res lem}).
	
	The connecting lemma states that in any sufficiently large $ k $-graph with large minimum $ d $-degree, we can connect  any two disjoint ordered  $ \ell $-sets of vertices by a short $ \ell $-path.
	%any other ordered  $ \ell $-set of vertices.
	We remark that the proof of Lemma~\ref{con lem} in~\cite{Han2021Non} follows closely that of a connecting lemma in~\cite{Daniela2010Hamilton} with minor modifications. We write $x\ll y\ll z$ to mean that we can choose constants from right to left, that is, for any $z>0$, there exist functions $f$ and $g$ such that, whenever $y\leq f(z)$ and $x\leq g(y)$, the subsequent statement holds. Statements with more variables are defined similarly.
	
	\begin{lemma}[Connecting lemma \cite{Han2021Non}, Lemma 4.1]
		\label{con lem}
		Suppose that $k\ge3$ and $1\leq \ell<d \leq k-1$ such that $(k-\ell)\nmid k$, and that $1/n \ll \beta\ll \mu,1/k$.
		Let $H$ be a $k$-graph of order $n$ satisfying $\delta_d(H)\ge \mu \binom{n}{k-d}$. Then for any two disjoint ordered $\ell$-sets $S$ and $T$ of vertices of $H$, there exists an $\ell$-path $P$ in $H$	from $S$ to $T$ such that $P$ contains at most  $8k^{5}$ vertices.	
	\end{lemma}
	
	Let $H$ be an $n$-vertex $k$-graph, and let $S$ be a set of $ k-\ell $ vertices of $H$. 
	An $ \ell $-path $P$ in $H$ with ordered ends $ P^{beg} $ and $ P^{end} $ \emph{absorbs $S$} if $P$ does not contain any vertex of $S$, and $H$ contains an $ \ell $-path $Q$ with the same ordered ends $ P^{beg} $ and $ P^{end} $, where $V(Q)=V(P)\cup S $.
	%Note that each time we absorb $( k-\ell) $-sets because the order of an  $ \ell $-path is always in the form of $\ell+r(k-\ell)$, where $r\in \mathbb{N}$ is the length (i.e. the number of edges) of the path. 
	Note that the number of vertices we absorb each time must be a multiple of $k-\ell$, because the order of an  $ \ell $-path is always in the form of $\ell+r(k-\ell)$, where $r\in \mathbb{N}$ is the length (i.e. the number of edges) of the path. 
	We follow \cite{Daniela2010Hamilton} and use the absorbing path of order $b(k,\ell)$ with $b(k,\ell)\le k^{4}$, whose definition can be found in  Appendix A.  
	We say that a $( k-\ell) $-set $S$ of $V(H)$ is \emph{$c$-good} (otherwise \emph{c-bad}) if $H$ contains at least $cn^{b(k,\ell)}$ absorbing paths for $S$, each on  $b(k,\ell)$ vertices.
	
	\begin{lemma} [Absorbing path lemma]
		\label{abs lem} 
		Suppose $k\ge3$, $1\leq \ell,d \leq k-1$ such that $d \ge \max\{k-\ell,\ell+1\}$ and $(k-\ell)\nmid k$, and that $1/n \ll \alpha \ll c \ll \gamma  \ll\mu'\ll \mu,1/k$. 
		Let $H$ be a $k$-graph of order $n$ with $\delta_d(H)\ge \mu \binom{n}{k-d}$. Then $H$ contains an $\ell$-path $P$ on at most $\mu' n$ vertices such that the following properties hold:
		\begin{enumerate}
			\item \label{AbsLem:1} Every vertex of $H - V(P)$ lies in at most  $\gamma n^{k-\ell-1}$ $c$-bad $(k-\ell)$-sets, 
			\item \label{AbsLem:2} $P$ can absorb any collection of at most $\alpha n$ disjoint $c$-good $(k-\ell)$-sets of vertices of $H-V (P)$.
		\end{enumerate}
	\end{lemma}
	
	The following lemma states that any sufficiently large $ k $-graph satisfying the minimum degree condition can be almost covered by a constant number of vertex-disjoint $ \ell $-paths.
	
	\begin{lemma} [Path cover lemma]
		\label{pat lem}
		Suppose that $k\ge3$,   $1\leq \ell<k/2$,  $1\leq d\leq k-1$, and  $1/n \ll 1/D\ll \varepsilon \ll\mu,1/k$. Let $H$ be a $k$-graph of order $n$ with $\delta_d(H)\ge (t(k,d,\ell)+\mu) \binom{n}{k-d}$.
		Then $H$ contains a set of at most $D$ disjoint $\ell$-paths covering all but at most $4\varepsilon n$ vertices of $H$.
	\end{lemma}
	
	We also need the following two results in the proof of Theorem~\ref{thm}. The first result  asserts that for $ 1\le d\le k-1 $, if $ H $ is a large $ k $-graph with minimum  $ d $-degree and  we choose $ R\subseteq V(H) $ uniformly at random, then with high probability all sets of $d$ vertices have a large degree in $R$, ensuring that we can use vertices of $ R $ to connect every two disjoint $ \ell $-sets.
	
	\begin{lemma} [Reservoir lemma \cite{Daniela2010Hamilton}, Lemma 8.1]
		\label{res lem}  
		Suppose that $k\ge2$,  $1 \le d \le k-1$, and  $1/n \ll \alpha ,\mu,1/k$. Let $H$ be a $k$-graph of order $n$ with $\delta_{d}(H) \ge \mu\binom{n}{k-d}$, and let $R$ be a subset of $V(H)$ of size $\alpha n$ chosen uniformly at random. Then the probability that
		$|N_{H}(S)\cap \binom{R}{k-d}|\ge \mu \binom{\alpha n}{k-d}-n^{k-d-1/3}$
		for every $S \in  \binom{V(H)}{d}$ is $1-o(1)$.
	\end{lemma}
	
		The second result (Daykin and H\"aggkvist \cite{Daykin1981Degrees}) gives a minimum vertex degree condition that guarantees the existence of a perfect matching in a uniform hypergraph.
	
	\begin{lemma}\cite{Daykin1981Degrees} \label{mat lem}
		Suppose that $k\ge2$ and $k\mid n$. Let $H$ be a $k$-graph of order $n$ with $\delta_1(H)\ge \frac{k-1}{k} \binom{n-1}{k-1}$. Then $H$ contains a perfect matching.	
	\end{lemma}
	
	%The rest of this section is devoted to the proof of Theorem~\ref{thm}.
	Assuming these lemmas, there is a routine scheme of the absorbing method to prove Theorem~\ref{thm}.
	In our specific case, the proof follows closely the argument in~\cite{Daniela2010Hamilton}.
	
	\subsection{Proof of Theorem~\ref{thm}}
	%\noindent\textbf{Proof of Theorem 1.1} 
	Suppose we have the constants satisfying the following hierarchy
	\[1/n\ll 1/D \ll \varepsilon \ll \alpha \ll c  \ll \gamma \ll \gamma'\ll \eta \ll 1/k\]
	and in particular assume that $n$ is a multiple of $k-\ell$.
	 Let $H$ be a $k$-graph on $n$ vertices such that $\delta_d(H)\ge (t(k,d,\ell)+\eta) \binom{n}{k-d}$. 
	 A Hamilton $\ell$-cycle of $H$ will be constructed via the following steps.
	
	\textbf{Build an absorbing path.} 
	Let $P_{0}$ be an absorbing $\ell$-path returned by Lemma~\ref{abs lem}. 
	Then in particular, $|V(P_0)|\le\gamma'n$ and $P_{0}$ can absorb any set of at most $\alpha n $ $c$-good $(k-\ell)$-sets. 
	Define the auxiliary $(k-\ell)$-graph $G$ on $V(H)$ such that $E(G)$ consists of all $c$-good $(k-\ell)$-sets of $V(H)$. 
By \eqref{AbsLem:1} of Lemma~\ref{abs lem}, we infer that $\deg_{G}(v) \ge \binom{n-1}{k-\ell-1} -\gamma n^{k-\ell-1} \ge (1-\gamma')\binom{n}{k-\ell-1}$ for every $v \in V(G) \setminus V (P_{0})$.
	
	\textbf{Choose a reservoir set.} 
	We choose a set of $\alpha n$ vertices uniformly at random from $H$, denoted by $R$.
Applying Lemma~\ref{res lem} to $H$ and $G$, we obtain that with probability $1-o(1)$, 
	
	\[\left |N_{G}(v)\cap  \binom{R}{k-\ell-1}\right| \ge (1-2\gamma')\binom{\alpha n}{k-\ell-1}\]
	for every $v \in V(G) \setminus V (P_{0})$ and 
	\[\left |N_{H}(S)\cap  \binom{R}{k-d}\right| \ge (t(k,d,\ell)+\eta/2)\binom{\alpha n}{k-d}\]
	for every $S \in \binom{V(H)}{d}$. 
	Since $\mathbb{E}[|R \cap V(P_{0})|] = \alpha|V(P_{0})|$, by Markov's inequality, with probability at least $1/2$,  we have $|R \cap V(P_{0})| \le 2\gamma'\alpha n$. 
	Then we fix a choice of $R$ which has all three properties above. 
	Let $R' = R \setminus V(P_{0})$, and thus $|R'| \ge (1 - 2\gamma')\alpha n$. 
	Then for every $S \in \binom{V(H)}{d}$, we have
	\[ \left | N_{H}(S)\cap  \binom{R'}{k-d} \right |\ge (t(k,d,\ell)+\eta/2)\binom{\alpha n}{k-d}-2\gamma'\alpha n\binom{n}{k-d-1}\ge t(k,d,\ell)\binom{|R'|}{k-d}.\]  
	Similarly, $|N_{G}(v)\cap  \binom{R'}{k-\ell-1}| \ge (1-\eta)\binom{|R'|}{k-\ell-1}$ for every vertex $v$ in $V(G)\setminus V(P_0)$.
	
	\textbf{Cover the majority of vertices by vertex-disjoint $\ell$-paths.} 
	Let $V':=V(H)\setminus (R \cup V(P_{0}))$, and let $H':=H[V']$ be the restriction of $H$ on $V'$. 
	As $|R \cup V(P_{0})|\le2\gamma'n$, we have $\delta_d(H')\ge (t(k,d,\ell)+\eta/2) \binom{|H'|}{k-d}$. Applying Lemma~\ref{pat lem} to $H'$, we obtain a collection of at most $D$ vertex-disjoint $\ell$-paths $P_{1},\dots,P_{q}$ 
	covering all but at most $\varepsilon n$ vertices of $H'$.
	Let $X$ denote the set of uncovered vertices. Thus $|X| \le \varepsilon n$.
	
	\textbf{Connect up all $\ell$-paths.}   
	Denote by $P^{beg}_{i}$ and $P^{end}_{i}$ the ordered ends of $P_{i}$, $0 \le i \le q$. 
	For $0 \le i \le q$, we now find disjoint $\ell$-paths $P'_{i}$ by Lemma~\ref{con lem} to connect $P^{end}_{i}$
	and  $P^{beg}_{i+1}$ (where subindices are taken modulo $q+1$), which will connect $P_{i}$ and $P_{i+1}$. Note that  $V(P'_{i})\subseteq R' \cup P^{end}_{i} \cup P^{beg}_{i+1}$ and $|V(P'_{i})|\le 8k^{5}$.
	More precisely, suppose that we have chosen such $\ell$-paths  $P'_{0},\dots,P'_{i-1}$. Let $R_{i}=(R' \cup P^{end}_{i} \cup P^{beg}_{i+1})\setminus \bigcup_{j=0}^{i-1}V(P'_{j})$.  
	Thus 
	\[\delta_d(H[R_{i}])\ge ( t(k,d,\ell) +\eta/2)\binom{\alpha n}{k-d}-8k^{5}(D+\gamma' n ) \binom{\alpha n+2\ell}{k-d-1}\ge \frac{t(k,d,\ell)}{2}\binom{\alpha n}{k-d}\]
	and thus we may apply Lemma~\ref{con lem} to find a desired $\ell$-path  $P'_{i}$.
	
	\textbf{Absorb the leftover.} Let $C = P_{0}P'_{0}P_{1}P'_{1}\cdots P_{q}P'_{q}$ be the $\ell$-cycle we have obtained so far.  Let $R'' := V(H)\setminus V(C)$. Then indeed $ R''=X\cup(R'\setminus\bigcup_{1\le i\le q}V(P_i'))  $ and in particular, $(1- 3\gamma')\alpha n \le|R''| \le (\alpha+\varepsilon)n$. Because $k-\ell$ divides both $ n $ and $|V(C)|$, we get $(k-\ell) \mid|R''|$. 
	Moreover,  for each $v \in R''$, we have $|N_{G[R'']}(v)| \ge (1-2\eta)\binom{| R''|}{k-\ell-1}$.
	Since $(k-\ell)\nmid k$, we get $k-\ell \ge 2$. 
	By Lemma~\ref{mat lem}, the graph $G[R'']$ contains a perfect matching which allows us to partition the set $R''$ into (at most $\alpha n$) $c$-good $(k-\ell )$-sets. %Since $P_{0}$ can absorb any collection of at most $\alpha n$ $c$-good $(k-\ell )$-sets, 
	Utilizing the absorbing property of $ P_0 $, there exists an $\ell$-path $Q_{0}$ with  $V(Q_{0})= V(P_{0}) \cup R''$ such that  $P_0$ and $Q_{0}$ have the same ordered ends. Thus a Hamilton $\ell$-cycle  $C' = Q_{0}P'_{0}P_{1}P'_{1}\cdots P_{q}P'_{q}$ in $H$ is obtained.
	
	\section{The hypergraph regularity method }
	
	We review the hypergraph regularity method, starting with some notation.
	\subsection{Regular complexes}
	A hypergraph $H$ is a \emph{complex} if whenever $e \in E(H)$ and $e'$ is a non-empty subset of $e$ we have that $e'\in E(H)$. All the complexes considered in this paper have the property that every vertex is contained in an edge. A complex $H$ is a \emph{$k$-complex} if all the edges of $H$ consist of at most $k$ vertices. The edges of size $i$ are called \emph{$i$-edges} of $H$. For convenience, we write $|H| := |V (H)|$ for the order of $H$. 
	Given a $k$-complex $H$, for each $i \in [k]$ we denote by $H_i$ the \emph{underlying $i$-graph} of $H$, where the vertices of $H_i$ are those of $H$ and the edges of $H_i$ are the $i$-edges of $H$.
	
	Note that a $k$-graph $H$ can be naturally turned into a $k$-complex, denoted by $H^{\leq}$, by replacing every edge into a complete $i$-graph $K^{(i)}_k$, for each $1 \leq i \leq k$. 
	%(In a more general $k$-complex we may have $i$-edges which do not lie within an $(i + 1)$-edge.) 
	Given $k \leq s$, a \emph{$(k, s)$-complex} $H$ is an $s$-partite $k$-complex, by which we mean that the vertex set of $H$ can be partitioned into sets $V_1, \dots , V_s$ such that every edge of $H$ meets each $V_i$ in at most one vertex.
	
	Now we motivate the notion of the relative density.
	For $i \geq 2$, suppose $H_i$ is an $i$-partite $i$-graph and $H_{i-1}$ is an $i$-partite $(i-1)$-graph, defined on the same vertex set (also under the same partition into $i$ parts). 
	We denote by $\mathcal{K}_i(H_{i-1})$ as the set of $i$-sets of vertices which form a copy of the complete $(i-1)$-graph $K^{(i-1)}_i$ on $i$ vertices in $H_{i-1}$. Then the \emph{density} of $H_i$ with respect to $H_{i-1}$ is defined to be
	\[ d(H_i|H_{i-1}):=\begin{cases}
		\frac{|\mathcal{K}_i(H_{i-1})\cap E(H_i)|}{|\mathcal{K}_i(H_{i-1})|} & \text{if $|\mathcal{K}_i(H_{i-1})|>0$,} \\
		0 & \text{otherwise.}
	\end{cases}
	\]
	More generally, if $\mathbf{Q} := (Q(1), Q(2),\dots, Q(r))$ is a collection of $r$ subhypergraphs of $H_{i-1}$, then we define $\mathcal{K}_i(\mathbf{Q}):= \bigcup^r_{j=1}\mathcal{K}_i(Q(j))$ and
	\[ d(H_i|\mathbf{Q}):=\begin{cases}
		\frac{|\mathcal{K}_i(\mathbf{Q})\cap E(H_i)|}{|\mathcal{K}_i(\mathbf{Q})|} & \text{if $|\mathcal{K}_i(\mathbf{Q})|>0$,} \\
		0 & \text{otherwise.}
	\end{cases}
	\]
	We say that $H_i$ is \emph{$(d_i, \delta, r)$-regular with respect to $H_{i-1}$} if for all $r$-tuples $\mathbf{Q}$ with $|\mathcal{K}_i(\mathbf{Q})| > \delta|\mathcal{K}_i(H_{i-1})|$ it holds that $d(H_i|\mathbf{Q}) = d_i \pm \delta$. Instead of $(d_i, \delta, 1)$-regularity we simply refer to $(d_i, \delta)$-\emph{regularity}.
	
	Given $3 \leq k \leq s$ and a $(k, s)$-complex $H$, we say that $H$ is $(d_k, \dots , d_2, \delta_k, \delta, r)$-regular if the following conditions hold:\vspace{3mm}
	
	\begin{itemize}
		\item For every $i = 2, \dots , k-1$ and for every $i$-tuple $K$ of vertex classes either $H_i[K]$ is $(d_i, \delta)$-regular with respect to $H_{i-1}[K]$ or $d(H_i[K]|H_{i-1}[K]) = 0$.
		\item For every $k$-tuple $K$ of vertex classes either $H_k[K]$ is $(d_k, \delta_k, r)$-regular with respect to $H_{k-1}[K]$ or $d(H_k[K]|H_{k-1}[K]) = 0$.
	\end{itemize}
	\vspace{3mm}
	Here we write $H_i[K]$ for the restriction of $H_i$ to the union of all vertex classes in $K$. We sometimes denote $(d_k, \dots , d_2)$ by $\mathbf{d}$ and simply refer to $(\mathbf{d}, \delta_k, \delta, r)$-\emph{regularity}.
	
%	We will need the following lemma which states that the restriction of regular complexes to a sufficiently large set of vertices is still regular. 

	\subsection{Statement of the regularity lemma}
	In this section we state the version of the regularity lemma for $k$-graphs due to R\"{o}dl and Schacht, which we will use to prove our connecting lemma. To prepare for this we again have to introduce some notation. Suppose that $V$ is a finite set of vertices and $\mathcal{P}^{(1)}$ is a partition of $V$ into sets $V_1, \dots , V_{a_1}$, called \emph{clusters}. Given $k \geq 3$ and any $j \in [k]$, we denote by $\text{Cross}_j = \text{Cross}_j (\mathcal{P}^{(1)})$, the set of all those $j$-subsets $J$ of $V$ such that $|J\cap V_i| \leq 1$ for all $1\leq i \leq a_1$. For every set $A \subseteq [a_1]$ with $2 \leq |A| \leq k-1$ we write $\text{Cross}_A$ for all those $|A|$-subsets of $V$ that meet each $V_i$ with $i \in A$. Let $\mathcal{P}_A$ be a partition of $\text{Cross}_A$. The partition classes of $\mathcal{P}_A$ is called \emph{cells}. For $i = 2, \dots , k-1$, let $\mathcal{P}^{(i)}$ be the union of all $\mathcal{P}_A$ with $|A| = i$. Thus $\mathcal{P}^{(i)}$ is a partition of $\text{Cross}_i$.
	
	Let $\mathcal{P}(k-1) = \{\mathcal{P}^{(1)}, \dots ,\mathcal{P}^{(k-1)}\}$ be a family of partitions on $V$ which satisfies the following condition. Recall that $a_1$ denotes the number of clusters in $\mathcal{P}^{(1)}$. Consider any $B \subseteq A \subseteq [a_1]$ such that $2 \leq |B| < |A| \leq k-1$ and suppose that $S, T \in \text{Cross}_A $ lie in the same cell of $\mathcal{P}_{A}$. Let $S_B := S \cap \bigcup_{i\in B} V_i$ and define $T_B$ similarly. Then $S_B$ and $T_B$ lie in the same cell of $\mathcal{P}_{B}$.
	
	Given $1 \leq i \leq j \leq k$ with $i < k$, $J \in \text{Cross}_j$ and an $i$-set $Q \subseteq J$, we write $C_Q$ for the set of all those $i$-sets in $\text{Cross}_i$ that lie in the same cell of $\mathcal{P}^{(i)}$ as $Q$. (In particular, if $i = 1$ then $C_Q$ is the cluster containing the unique element in $Q$.) The \emph{polyad} $\hat{P}^{(i)}(J)$ of $J$ is defined by $\hat{P}^{(i)}(J):=\bigcup_Q C_Q$, where the union is over all $i$-subsets $Q$ of $J$. That is, we can view $\hat{P}^{(i)}(J)$ as a $j$-partite $i$-graph (whose vertex classes are the clusters intersecting $J$). We let $\hat{P}^{(j-1)}(J)$ be the set consisting of all the $\hat{P}^{(j-1)}(J)$ for all $J\in \mathrm{Cross}_j$. So for each $K \in \text{Cross}_k$ we can view $\bigcup^{k-1}_{i=1}\hat{P}^{(i)}(K)$ as a $(k-1, k)$-complex.
	We say that the family of partitions $\mathcal{P}=\mathcal{P}(k-1)$ is \emph{$(\eta, \delta, t)$-equitable} if\vspace{3mm}
	
	\begin{itemize}
		\item there exists $\mathbf{d}= (d_{k-1}, \dots , d_2)$ such that $d_i \geq 1/t$ and $1/d_i \in \mathbb{N}$ for all $i = 2, \dots , k-1$,
		\item $\mathcal{P}^{(1)}$ is a partition of $V$ into $a_1$ clusters of equal size, where $1/\eta \leq a_1 \leq t$,
		\item for all $i = 2, \dots , k-1$, $\mathcal{P}^{(i)}$ is a partition of $\mathrm{Cross}_i$ into at most $t$ cells,
		\item for every $K \in \mathrm{Cross}_k$, the $(k-1, k)$-complex $\bigcup^{k-1}_{i=1}\hat{P}^{(i)}(K)$ is $(\mathbf{d}, \delta, \delta, 1)$-regular.
	\end{itemize}
	\vspace{3mm}
	Note that the last condition implies that for $i = 2, \dots , k-1$ the cells of $\mathcal{P}^{(i)}$ have almost equal size.
	
	Let $\delta_k > 0$ and $r \in \mathbb{N}$. Suppose that $H$ is a $k$-graph on $V$ and $\mathcal{P} = \mathcal{P}(k-1)$ is a family of partitions on $V$. Given a polyad $\hat{P}^{(k-1)} \in \hat{\mathcal{P}}^{(k-1)}$, we say that $H$ is \emph{$(\delta_k, r)$-regular with respect to $\hat{P}^{(k-1)}$} if $H$ is $(d, \delta_k, r)$-regular with respect to $\hat{P}^{(k-1)}$ for some $d$. We say that $H$ is \emph{$(\delta_k, r)$-regular with respect to $\mathcal{P}$} if
	\[
	\left|\bigcup\{\mathcal{K}_k(\hat{P}^{(k-1)}): \text{$H$ is not $(\delta_k, r)$-regular with respect} \text{ to } \hat{P}^{(k-1)} \in \hat{\mathcal{P}}^{(k-1)} \}\right| \leq \delta_k|V|^k.\]
	That is, at most a $\delta_k$-fraction of the $k$-subsets of $V$ form a $K_k^{(k-1)}$ that lies within a polyad with respect to which $H$ is not regular.
	
	Now we are ready to state the regularity lemma.
	
	\begin{thm}[Regularity lemma \cite{R2007Regular}, Theorem 17] 
		\label{reglm}
		Let $k\geq 3$ be a fixed integer. For all positive constants $\eta$ and $\delta_k$ and all functions $r:\mathbb{N} \rightarrow \mathbb{N}$ and $\delta:\mathbb{N} \rightarrow (0,1]$, there are integers $t$ and $n_0$ such that the following holds for all $n \geq n_0$ which are divisible by $t!$. Suppose that $H$ is a $k$-graph of order $n$. Then there exists a family of partitions $\mathcal{P}= \mathcal{P}(k-1)$ of the vertex set $V$ of $H$ such that
		\begin{enumerate}
			\item $\mathcal{P}$ is $(\eta, \delta(t), t)$-equitable and
			\item $H$ is $(\delta_k, r(t))$-regular with respect to $\mathcal{P}$.
		\end{enumerate}
	\end{thm}

	\subsection{The reduced $ k $-graph}
	The reduced $k$-graph of a regular partition is an important auxiliary $k$-graph which has been used in many applications of the regularity method\footnote{Note also that there have been a few different notions of such auxiliary hypergraphs, depending on the contexts.}.
	Suppose that we have constants
	$ 1/n\ll1/r,\delta\ll\min\{\delta_{k},1/t\}\le \delta_{k},\eta\ll d'\ll\theta\ll\mu ,1/k $
	and a $ k $-graph $ H $ on $ V  $ of order $ n $. We may apply the regularity
	lemma to $ H $ to obtain a family of partitions $ \mathcal{P}=\{\mathcal{P}^{(1)},\dots,\mathcal{P}^{(k-1)}\} $ of $ V $. Then the \emph{reduced k-graph} $ R = R(H,\mathcal{P}, d') $ is defined as a $ k $-graph whose vertices are the clusters of $ H $, i.e. the parts of $ \mathcal{P}^{(1)}: V_1,V_2,\dots,V_{a_{1}} $.
	A $k$-tuple of clusters forms an edge of $R$ if there is \emph{some} polyad $ \hat{P}^{(k-1)} $ induced on these $ k $ clusters such that $ H $ is $ (d'',\delta_{k}, r) $-regular with respect to $  \hat{P}^{(k-1)}  $ for some $ d''\ge d' $. Suppose that $ e $ is an edge of $R$, then there is some polyad $ \hat{P}^{k-1}(J) $ (where $ J\in \text{Cross}_{k} $) induced by the $ k $ clusters corresponding to $e$ such that $ H $ is $ (d'',\delta_{k}, r) $-regular with respect to $  \hat{P}^{(k-1)}(J)  $ for some $ d''\ge d' $.  Let $H^{*}$ be the $(k,k)$-complex consisting of
	\[\label{eqcomplex}
	\bigcup\limits_{i=1}^{k-1}\hat{P}^{i}(J) \cup (E(H)\cap \mathcal{K}(\hat{P}^{k-1}(J))). \tag{*} 
	\] 
	Then $H^{*}$ is  $(\mathbf{d},\delta_{k},\delta,r)$-regular, where $ \mathbf{d}=(d'',d_{k-1},\dots,d_2) $, and $d_{k-1},\dots,d_2$ are implicitly given in the definition of an $ (\eta,\delta,t) $-equitable family of partitions.
	
	It is well-known that 
	the reduced $k$-graph $R$ almost inherits the minimum degree condition  and the density condition from $H$.
	\begin{lemma}\label{deg lem}
		Suppose that $k\ge3$,  $1\leq d\leq k-1$, and $\gamma,d'\ll \theta \ll \mu$. Let $H$ be a $k$-graph of order $n$ and $ R:=R(H,\mathcal{P}, d') $ be the reduced $k$-graph  of $ H $, where $ \mathcal{P} $ is the family of partitions we obtained by Theorem~\ref{reglm}. If  $\delta_d(H)\ge (\mu+\theta) \binom{n}{k-d}$,
		then all but at most $\gamma \binom{|R|}{d}$ $d$-sets $S \in \binom{V(R)}{d}$ have degree at least $(\mu+\theta/2) \binom{|R|}{k-d}$. Moreover, if  $e(H)\ge (\mu+\theta)  \binom{n}{k}$, then  $e(R)\ge (\mu+\theta/2)  \binom{|R|}{k}$.
	\end{lemma}
	The proof of Lemma~\ref{deg lem} is by now a routine calculation and thus omitted.
	Similar calculations can be found e.g.~\cite{Daniela2010Hamilton}.

	\section{The proof of Lemma~\ref{pat lem}}
	
	%The idea of the proof of the path cover lemma is to apply the regularity lemma to $H$ in order to obtain a reduced $k$-graph $R$. Recall that the minimum degree condition of $H$ is almost inherited by $R$, but we need to use the following lemma (Lemma~\ref{perfectdeg lem}) to obtain the perfect minimum degree condition. By the definition of $t(k,d,\ell)$, we can find an almost perfect $Y_{k,2\ell}$-tiling in $R$. Consider any copy $Y'$ of $Y_{k,2\ell}$ in this packing. We will repeatedly apply Lemma~\ref{embed lem} to the sub-$k$-graph $H[Y']$ of $H$ corresponding to $Y'$ to obtain a bounded number of $\ell$-paths which cover almost all vertices of $H[Y']$. Doing this for all members of the $Y_{k,2\ell}$-tiling of $R$ will give a set of $\ell$-paths as required in Lemma~\ref{pat lem}.
	In this section we prove Lemma~\ref{pat lem}.
	Here is a quick outline of our proof, which also follows the use of the regularity method as in previous works.
	We apply the regularity lemma to $H$ and obtain a regular partition together with a reduced $k$-graph $R$.
	Since $R$ almost inherits the minimum $d$-degree of $H$, by the definition of $t(k,d,\ell)$, we can find an almost perfect $Y_{k,2\ell}$-tiling in $R$.
	Thus, each copy of $Y_{k,2\ell}$ corresponds to $2k-2\ell$ clusters in the partition and the regularity allows us to cover almost all vertices of these clusters by a bounded number of vertex-disjoint $\ell$-paths.
	Altogether we get the desired path cover in Lemma~\ref{pat lem}.
	
	\begin{lemma}[\cite{ferber2020diractype}, Lemma 3.4]
		\label{perfectdeg lem}
		Consider an $n$-vertex $k$-graph $G$, where all but $\delta \binom{n}{d}$ of the $d$-sets have degree at least $(\mu + \eta) \binom{n-d}{k-d}$. Let $S$ be a uniformly random subset of $Q$ vertices of $G$. Then with probability at least $1- \binom{Q}{d}(\delta+\exp^{-\Omega(\eta^{2}Q)})$, the random induced subgraph $G[S]$ has minimum $d$-degree at least $(\mu + \eta/2) \binom{Q-d}{k-d}$.
	\end{lemma}
	
	\begin{fact}\label{f1}
		Let $ P $ be an $\ell$-path on $ n $ vertices with $t := \frac{n-\ell}{k-\ell} $ edges, where $ t $ is odd. Then there is a partition $ V_1, V_2, \dots, V_k $ of $V(P) $  such that $ |V_i|=t$, $i\in[k-2\ell] $  and  $ |V_j|=(t+1)/2$, $j\in[k]\setminus[k-2\ell]  $, and each edge of $ P $ contains exactly one vertex of each $ V_i, i\in[k] $.
		%Let $ P $ be an $\ell$-path on $ n $ vertices with $t := \frac{n-\ell}{k-\ell} $ edges, where $ t $ is odd. Then there is a $ k $-coloring of $P $  with colors $[k]$ such that each color of $[k-2\ell] $ is used $t $ times and every other color is used $ (t+1)/2 $ times.
	\end{fact}
	\begin{proof}
		Denote $P$ as $ x_1,\dots,x_n $.
		Color $ x_i $ with $ i $ for $ i\in [k]$. Let  
		\[C= \{1,\dots,k-2\ell,k-2\ell+1,\dots,k-\ell,1,\dots,k-2\ell,k-\ell+1,\dots, k \}.\] Then we repeatedly use the pattern $ C $ ($(t-1)/2$ times) to color $ x_{k+1},\dots,x_{n} $. Note that every edge contains $ k $ colors where each  color of $ [k-2\ell] $ is used $ t$ times and every other color is used $ (t-1)/2+1$ times. Let $ V_i$ be the color class with color $i$, which forms a partition as desired.
	\end{proof} 
	
	We use the following fact to build our path in a regular $k$-tuple.
	
	\begin{fact}\label{f222}
		Let $k\ge 3$, $1\le \ell<k/2$, $\varepsilon>0$ and $n\in \mathbb N$.
		Suppose $H$ is a $k$-partite $k$-graph with vertex set $V_1, \dots, V_k$ each of size $m$ and $e(H)\ge \varepsilon m^k$.
		Then $H$ contains an $\ell$-path $P$ with  $t\ge \varepsilon m/2$ edges whose vertex set are partitioned as in Fact~\ref{f1}, namely, $P$ contains $t$ vertices in each of $V_1,\dots, V_{k-2\ell}$ and $ (t-1)/2+1$ vertices in every other cluster.
	\end{fact}
	
	\begin{proof}
		We call an $\ell$-set crossing if it contains at most one vertex from each $V_i$.
		%If $H$ contains a $d$-set $S$ such that $\deg(S)<\varepsilon m^{k-\ell}/2$, then we remove all edge containing $S$.
		From $H$ we iteratively do the following: in the current $k$-graph, if there is a crossing $\ell$-set $S$ such that $\deg(S)<\varepsilon m^{k-\ell}/(2\binom{k}{\ell})$, then we remove all edges containing $S$.
		The resulting $k$-partite $k$-graph has at least $\varepsilon m^k - m^\ell\cdot \varepsilon m^{k-\ell}/2>0$ edges, and every crossing $\ell$-set has degree either zero or at least $\varepsilon m^{k-\ell}/(2\binom{k}{\ell})$.
		
		We now find the desired path $P$ greedily, starting from an arbitrary edge.
		We extend the path alternately from either an $\ell$-end in $V_{k-2\ell+1}\times \cdots\times V_{k-\ell}$ or $V_{k-\ell+1}\times \cdots\times V_{k}$, until $P$ has $\varepsilon m/2$ edges.
		All extensions will succeed because the current path will block at most $(\varepsilon m/2)^{k-\ell} < \varepsilon m^{k-\ell}/(2\binom{k}{\ell})$ edges, and an $\ell$-end has degree nonzero (so at least $\varepsilon m^{k-\ell}/(2\binom{k}{\ell})$).
	\end{proof}

	\begin{proof}[Proof of Lemma~\ref{pat lem}]
		Choose constants such that
		\[1/n\ll 1/D  \ll 1/r,\delta ,c  \ll \min\{\delta_{k},1/t \} \le \delta_{k},\eta \ll d' \ll \gamma  \ll \varepsilon \ll \mu,1/k.\]
		We may assume that $t!\mid n$. Apply Theorem~\ref{reglm} to $H$, and let
		$V_{1}, \dots,V_{a_{1}}$ be the clusters of the partition obtained. Let $m = n/a_{1}$ be the size of each of these clusters and $R$ be the reduced $k$-graph on these clusters. The next claim says $R$ has an almost  perfect  $Y_{k,2\ell}$-tiling $ \mathcal{T}$. 
		\begin{claim}
			The $k$-graph $R$ has a $Y_{k,2\ell}$-tiling covering all but at most
			$\varepsilon a_{1}/{2}$ vertices of $R$.
		\end{claim}
		\begin{proof}
			By Lemma~\ref{deg lem}, all but at most $\gamma \binom{a_{1}}{d}$ $d$-sets $S \in \binom{V(R)}{d}$ have degree at least $(t(k,d,\ell)+\mu/2) \binom{a_{1}}{k-d}$. Choose $Q$ to be large enough. 
			In particular, by the definition of $t(k,d,\ell)$ every $k$-graph $G$ of order $Q$ with $\delta_d(G)\ge(t(k, d, \ell)+\mu/4) \binom{Q}{k-d}$ contains a $Y_{k,2\ell}$-tiling covering all but at most $ \varepsilon Q/4 $ vertices of $ G $. Let $\lambda=\binom{Q}{d}(\gamma+\exp^{-\Omega((\mu/2)^{2}Q)})$, and note that we can make $\lambda < \varepsilon/4$ by making $Q$ large.
			
			Now, we randomly partition $V_{1}, \dots,V_{a_{1}}$ into $\lfloor a_1/Q\rfloor$ subsets of size $Q$ (and give up at most $Q-1$ of them). By Lemma~\ref{perfectdeg lem}, the expected number of $Q$-sets that have minimum $d$-degree less than $(t(k,d,\ell)+\mu/4)\binom{Q}{k-d}$ is $ \lambda a_1/Q$. Thus with positive probability all but a $\lambda$-fraction of the subsets have minimum $d$-degree at least $(t(k,d,\ell)+\mu/4)\binom{Q}{k-d}$. 
			As discussed above, each of these $Q$-subsets $S \subseteq V(G)$ has the property that $G[S]$ has a $Y_{k,2\ell}$-tiling that leaves at most $\varepsilon Q/4$ vertices. 
			Combining these $Y_{k,2\ell}$-tilings gives a  $Y_{k,2\ell}$-tiling of $R$ covering all but at most $Q-1+\lambda a_1+(\varepsilon Q/4) ( {a_1}/{Q})\le {\varepsilon a_{1}}/{2}$ vertices.
		\end{proof}
		%	By Lemma~\ref{deg lem}, all but at most $\gamma \binom{a_{1}}{d}$ $d$-sets $S \in \binom{V(R)}{d}$ have degree at least $(t(k,d,\ell)+\mu/2) \binom{a_{1}}{k-d}$. Choose $Q$ to be large enough. In particular, by the definition of $t(k,d,\ell)$ every $k$-graph $G$ of order $Q$ with $\delta_d(G)\ge(t(k, d, \ell)+\mu/4) \binom{Q}{k-d}$ contains a $Y_{k,2\ell}$-tiling covering all but at most $ \varepsilon Q/4 $ vertices of $ G $. 	Note that we can make $\lfloor a_1/Q\rfloor\binom{Q}{d}(\gamma+\exp^{-\Omega((\mu/2)^{2}Q)})<1$ by making $Q$ large.

	%	Now, we randomly partition $V_{1}, \dots,V_{a_{1}}$ into $\lfloor a_1/Q\rfloor$ subsets of size $Q$ (and give up at most $Q-1$ of them). By Lemma~\ref{perfectdeg lem}, 	with positive probability, every subset has minimum $d$-degree at least $(t(k,d,\ell)+\mu/4)\binom{Q}{k-d}$. 
		%As discussed above, each of these $Q$-subsets $S \subseteq V(G)$ has the property that $G[S]$ has a $Y_{k,2\ell}$-tiling that leaves at most $\varepsilon Q/4$ vertices. 	Combining these $Y_{k,2\ell}$-tilings gives a  $Y_{k,2\ell}$-tiling of $R$ covering all but at most $Q-1+(\varepsilon Q/4) ( {a_1}/{Q})\le {\varepsilon a_{1}}/{2}$ vertices.
		
		For each  $ Y_{k,2\ell}\in \mathcal{T}$  with edges denoted by $j_1j_2\dots j_k$ and $j_{k-2\ell+1}j_{k-2\ell+2}\dots j_{2k-2\ell}$, let $\mathcal{V}:=\{V_{j_{k-2\ell+1}},V_{j_{k-2\ell+2}},\dots,V_{j_{k}}\} $. 
		%Let $ H^{*}_1 $ and $ H^{*}_2$ be the $(k,k)$-complexes defined as in \eqref{eqcomplex} corresponding to $ V_{j_1},V_{j_2},\dots,V_{j_k} $ and $ V_{j_{k-2\ell+1}},V_{j_{k-2\ell+2}},\dots,V_{j_{2k-2\ell}} $ separately. 
		Let $ H^{*}_1 $ and $ H^{*}_2$ be the induced subgraph of $H$ on $ V_{j_1},V_{j_2},\dots,V_{j_k} $ and $ V_{j_{k-2\ell+1}},V_{j_{k-2\ell+2}},\dots,V_{j_{2k-2\ell}} $.
		Write $d^*:=(d_k-\delta_k)\prod_{2\le i\le k-1}(d_i-\delta)$ and let $t$ be an odd number between $\lceil d^*\varepsilon m/8\rceil-1$ and $\lceil d^*\varepsilon m/8\rceil$.
		%Let $ P^{\le} $ be a $ (k,k) $-complex corresponding to an $\ell$-path $P $ with $t$ edges. 
		Given any $k$-partite subgraph of $ H^{*}_1 $ and $ H^{*}_2$ with $\varepsilon m/4$ vertices from each part, the regularity implies that there are at least $d^*(\varepsilon m/4)^k$ edges.
		Fact~\ref{f222} implies $ H^{*}_1 $ contains an $\ell$-path $P_1 $ with $t$ edges which intersects each cluster of $ \mathcal{V} $ in $(t+1)/2 $ vertices and every other cluster in $ t $ vertices. 
		Similarly, we can find an $\ell$-path $P_1' $ in $ H^{*}_2$ with $ t$ edges which intersects each cluster of $ \mathcal{V} $ in $ (t+1)/2 $ vertices. 
		Furthermore, we can delete the vertices in $ P_1 $ and $ P_1' $ and find similar $\ell$-paths of length $t$ with the same vertex distribution in the remaining subgraph of $ H^{*}_1 $ and $ H^{*}_2$, until there is one cluster of $ \mathcal{V} $ with less than $\varepsilon m/4$ remaining vertices.
		Note that we obtain at most $ m/t \le 9/(d^*\varepsilon)$ vertex disjoint $\ell$-paths.
		When the process stops, the number of remaining vertices in each cluster is at most $\varepsilon m/4 + 9/(d^*\varepsilon)\le \varepsilon m/2$.
		
		Executing this for all copies of $ Y_{k,2\ell}$ in $ \mathcal{T}$, we obtain a collection of at most $(a_1/(2k-2\ell)) 9/(d^*\varepsilon)\le D $ $ \ell $-paths.
		The total number of uncovered vertices is at most
		$(\varepsilon m/2)a_1+(\varepsilon a_1/2)m \le\varepsilon n$.\end{proof}

	\section{Proof of Theorem~\ref{ineg}}
	
	We collect known results on the Tur\'an number of $Y_{k,b}$.
	%The following result was proved by Frankl and Furedi.
	
	\begin{thm}\cite{Erd1974Intersection,Peter1985Forbidding}  \label{FrFu}
		For $k > b \ge 0$, there exists an integer $n_k$ such that for any $n \ge n_k$, ${\rm ex}(n,Y_{k, b}) \le \binom{n-1}{k-1}$ holds.
	\end{thm}
	For $b=0$, Theorem~\ref{FrFu} is just the Erd\H{o}s--Ko--Rado Theorem \cite{Erd1974Intersection}; for $1\le b\le k-2$, it follows from a result of Frankl and F\"uredi \cite{Peter1985Forbidding} and in fact, a more precise bound $O( n^{\max \{b, k-b-1 \}})$ is known; for $b=k-1$, a simple calculation gives $\ex(n,Y_{k, b}) \le \frac{1}{k}\binom{n}{k-1}$. Next we apply induction on the size of $Y_{k,b}$-tiling and use Theorem~\ref{FrFu} to show Theorem~\ref{ineg}.
	%was obtained but the bound above suffices for our need (when $b=k-1$, it is easy to see that $ex(n,Y_{k, b}) \le \frac{1}{k}\binom{n}{k-1}$).
	%In this paper we shall only need a simpler upper bound as $ex(n,Y_{k, b}) \le d_k n^{k-1}$.

	\begin{proof}[Proof of Theorem~\ref{ineg}]
		Within the proof, we write $Y: = Y_{k, b}$ for short.
		Let $s_0$ be sufficiently  large and $ s\ge s_0$ so that \begin{equation}\label{eqno5.1}
		n-1-(2k-b)(s-1) \ge n_{k-1} \  \text{and} \  n \ge n_k,
		\end{equation} 
		namely, we can apply Theorem \ref{FrFu}.  We introduce a new integer $ t \in [s] $ and
	 prove that if $ n \ge (2(2k-b)^2+1)(k-1)s+t$ and $e(H) \ge \binom{n}{k} - \binom{n-t+1}{k} +  \binom{n-1}{k-1}+ \binom{n-1}{k-2}(2k-b)t$, then $H$ contains a $Y$-tiling of size $t$.
		%We proceed by induction on $t$.
		
		For the base case $t =1$, if the theorem fails, then $H$  contains no copy of $Y$, thus $e(H)\le \binom{n-1}{k-1}$  by Theorem \ref{FrFu}, which is a contradiction.
		Suppose that the theorem holds for $t-1$, but fails for $t \ge 2$. 
		In particular, $H$ contains a $Y$-tiling of size $t-1$.
		If there is a vertex $v$ with $\deg (v) > \binom{n-1}{k-1} - \binom{n-1-(2k-b)(t-1)}{k-1} + \binom{n-1}{k-2}$,
		then we have
		%\[
		%\begin{aligned}
		%&e(H - v) \ge e(H) - \binom{n-1}{k-1} \\
		%& \ge  \binom{n}{k} - \binom{n-t+1}{k} + d_k (n-(2k-b)(k-1))^{k-1} + \deg_{k-1} n^{k-2}(2k-b)(t-1) - \binom{n-1}{k-1}\\
		%& \ge \binom{n-1}{k} - \binom{n-t+1}{k} + d_k ((n-1)-(2k-b)(k-1))^{k-1} + \deg_{k-1} (n-1)^{k-2}(2k-b)(t-1).
		%\end{aligned}
		%\]
		\[
		e(H - v) \ge e(H) - \binom{n-1}{k-1}  >\binom{n-1}{k} - \binom{n-t+1}{k} + \binom{n-2}{k-1}+\binom{n-2}{k-2}(2k-b)(t-2).
		\]
		%\ge  \binom{n}{k} - \binom{n-t+1}{k} + \binom{n}{k-1} - \binom{n-1}{k-1} \ge \binom{n-1}{k} - \binom{n-t+1}{k} + \binom{n-1}{k-1}.
		%\]
		Thus by induction hypothesis, $H-v$ contains a $Y$-tiling $M$ of size $t-1$. Note that the number of members of $ N(v) $ intersecting $ V(M) $ is at most $ \binom{n-1}{k-1} - \binom{n-1-(2k-b)(t-1)}{k-1}  $.
		Since $\deg (v)$ is large enough, by Theorem \ref{FrFu} and equation~\ref{eqno5.1}, $ N(v) $ contains a copy of $ Y_{k-1, b-1} $ which is disjoint from $ V(M) $.
		Therefore we obtain a $Y$-tiling of size $t$, which is a contradiction.
		
		Now we may assume a maximum degree condition of $H$ \[
		\Delta(H) \le \binom{n-1}{k-1} - \binom{n-1-(2k-b)(t-1)}{k-1} + \binom{n-1}{k-2} <  (2k-b)(t-1) \binom{n-2}{k-2} +\binom{n-1}{k-2},
		\]
		where we used 
		$
		\binom{n-1}{k-1} - \binom{n-1-(2k-b)(t-1)}{k-1} = \binom{n-2}{k-2} + \dots + \binom{n-1-(2k-b)(t-1)}{k-2} < (2k-b)(t-1) \binom{n-2}{k-2}.
		$
		Let $M$ be a $Y$-tiling of size $t-1$ in $H$ and note that $V(H)\setminus V(M)$ is $Y$-free.
		Therefore, we have \[
		e(H) \le \binom{n-1}{k-1}+\Delta(H) (2k-b)(t-1) < (2k-b)^2(t-1)^2 \binom{n-2}{k-2} +\binom{n-1}{k-1}+ \binom{n-1}{k-2} (2k-b)(t-1).
		\]
		Combining with the lower bound on $e(H)$, we have
		\[
		(t-1)\frac{n-t+1}{k-1} \binom{n-t}{k-2}  = (t-1) \binom{n-t+1}{k-1} < \binom{n}{k} - \binom{n-t+1}{k}  <  (2k-b)^2(t-1)^2 \binom{n-2}{k-2},
		\]
		where we used $\binom{n}{k} - \binom{n-t+1}{k} = \binom{n-1}{k-1} +\dots + \binom{n-t+1}{k-1} >(t-1) \binom{n-t+1}{k-1}$.
		Rearranging we get
		\[
		\left (\frac{n-t-k+3}{n-k+1} \right )^{k-2} < \frac{\binom{n-t}{k-2}}{ \binom{n-2}{k-2}} < \frac{(2k-b)^2 (k-1)(t-1) }{ n-t+1}.
		\]
		If $t=2$, then we obtain $n \le (2k-b)^2 (k-1)+1$, a contradiction.
		Otherwise $t \ge 3$, letting $ \gamma = \frac{t-2}{n-k+1} $ yields
		\[
		(1-\gamma)^{k-2}< (2k-b)^2(k-1) \frac{t-1}{n-t+1} < (2k-b)^2(k-1) \frac{t-2}{n-t-k+3}\frac{t-1}{t-2} \le 2(2k-b)^2(k-1) \frac{\gamma}{1-\gamma},
		\]
		that is, $(1-\gamma)^{k-1} < 2 (2k-b)^2(k-1)\gamma$.
		Since $
		n \ge (2(2k-b)^2+1)(k-1)s+t $, we derive
		\[
		\gamma \le \frac{t-2}{(2(2k-b)^2+1)(k-1)s+t-k+1} \le \frac{s-2}{(2(2k-b)^2+1)(k-1)(s-2)}.
		\]
		Note that for $0 < \gamma \le  \frac{1}{(2(2k-b)^2+1)(k-1)}$, 
		\[
		(1-\gamma)^{k-1} > 1-(k-1)\gamma \ge 2 (2k-b)^2(k-1)\gamma,
		\] which is a contradiction.
	\end{proof}

	%Combining the above two theorems, we obtain the following corollary.
	
	%\begin{cor} \label{cor}
	%	For $1 \le d < b \le k-1$ such that $2k-b \ge (2(2k-b-d)^2+1)(k-d-1)+1$,
	%\[
	%f_d^{n/(2k-b)}(Y_{k, b}, n) \le \binom{n-d}{k-d} - \binom{(1-\frac1{2k-b})n-d+1}{k-d} + \binom{n-d}{k-d-1}(2k-b-d)\frac{n}{2k-b}.
	%\]
	%	\[
	%	f_d^{n/(2k-b)}(Y_{k, b}, n) \le \binom{n-d}{k-d} - \binom{(1-\frac1{2k-b})n-d+1}{k-d} + o(n^{k-d}).
	%	\]
	%\end{cor}
	
	%For example, for $d = k-3$ this determines $f_d^{n/(2k-b)}(Y_{k, b}, n)$ asymptotically for $b=k-1$ and $k\ge 66$, and for $b=k-2$ and $k\ge 101$ (and thus gives new results on $d$-degree Hamilton $\ell$-cycle problem).
	
	%A test case is when $\ell=(k-1)/2$, and to apply our corollary above one should take $b=2\ell=k-1$.
	%Putting $b=k-1$ in $2k-b \ge (2(2k-b-d)^2+1)(k-d-1)+1$ one can deduce that $d+\sqrt[3]{d/2}\ge k$.
	%So we obtain asymptotically sharp result when $k-d$ is a constant and $k$ large. Unfortunately this does not give, say, $d=(1-\varepsilon)k$ when $k$ goes to infinity.
	
	%In fact, it is easy to see that the inequality holds when $k$ goes to infinity and $k-d$, $k-b$ both stay as constants.

	\section{Large $Y_{3,2}$-tilings: a proof of Theorem~\ref{edge th2}}
	\label{sec:6}
	
	% \textbf{Edge condition of $ Y_{3,2} $ tiling}
	
	Let $H=(V,E)$ be a $3$-graph with $|V|=n$. In this section, we write $ Y:=Y_{3,2} $ for brevity. 
	A $\{Y,E\}$-tiling is a collection of vertex-disjoint copies of $Y$ and edges in $H$. 
	%First, we consider the $\{Y,E\}$-tiling. 
	In the following theorem we study large $\{Y,E\}$-tilings, and then use it and the regularity method to show Theorem~\ref{edge th2}. 
	
	\begin{thm}
		\label{edge th1}
		Suppose that $ 0<\alpha \le  1/7$. Let $H$ be a 3-graph of order $n$ with
		\[
		e(H)\ge\binom{n}{3}-\binom{n-\alpha n}{3}+O(n^2).
		\]
		Then $H$ contains a $\{Y,E\}$-tiling covering more than $ 4\alpha n$ vertices.
	\end{thm}
	
Now we briefly discuss the approach that we use and other possible approaches.
We first remark that it is not clear to us how to use the powerful~\emph{shifting technique} in our context, which has been a crucial tool in studying large matchings.
By the standard regularity method, to build a large $Y$-tiling, it suffices to find a large fractional $Y$-tiling.
However, a suboptimal (pure) fractional $Y$-tiling is hard to analyze so it is not easy to improve in a greedy manner or an algorithmic way.
%Another possible way is to use~\emph{homomorphic tilings} introduced in~\cite{Bu2013Minimum} to track all possible fractional variants that can be efficiently turned into $Y$-tilings, but that needs a large number of applications.
We choose to use the $\{Y, E\}$-tilings, a mixture of $Y$-tilings and matchings.
A large $\{Y, E\}$-tiling can be converted to a large $Y$-tiling by the regularity method, and is much easier to manipulate.

Now we start to set up the proof.
We first show that it suffices to study the case $\alpha=1/7$.
For $\alpha\in (0,1/7)$, let $H'$ be a graph obtained by adding an $ s$-set $S$ such that $n-6s = 7\alpha n$ and adding all edges which contain at least one vertex from $S$. 
	Note that \[e(H') = e(H)+\binom{n+s}{3}-\binom{n}{3} \ge \binom{n+s}{3} - \binom{n-\alpha n}{3} +O(n^2)= \binom{n+s}{3}-\binom{n+s-\frac{n+s}{7}}{3} + O(n^2).\]
	%Note that $ e(H')=e(H)+\binom{s}{3}+n\binom{s}{2}+s\binom{n}{2}$. \[\binom{n+s}{3}-\binom{n+s-\frac{(n+s)}{7}}{3}+o(n^3)=\binom{n}{3}-\binom{n-\alpha n}{3}+(e(H')-e(H))+o(n^3).  \] 
	Applying the theorem with $\alpha=1/7$ to $H'$ gives a $\{Y, E\}$-tiling that covers more than $\frac{4}{7}(n+s)$ vertices. Deleting the copies of $Y$ and $E$ that contain vertices from $S$ gives a $\{Y, E\}$-tiling of $H$ that covers more than $\frac{4}{7}(n+s)-4s = \frac{4}{7}(n-6s) = 4\alpha n$ vertices.
	Therefore, it suffices to show Theorem~\ref{edge th1} for $\alpha=1/7$.
	%it suffices to show that the number of edges in 3-graph $ H $ of order $ n $,  which does not contain a $\{Y,E\}$-tiling covering more than $ 4\alpha n$ vertices, is at most $ \binom{n}{3}-\binom{n-\alpha n}{3}+o(n^3) $ for $\alpha=1/7.$
	
		Suppose  $ \cT=\{Y_1,Y_2,\dots,Y_{m_{1}},E_1,E_2,\dots,E_{m_{2}}\} $ is a maximum $\{Y,E\}$-tiling of $H$, that is, a $\{Y,E\}$-tiling covering the maximum size of vertices. 
	By adding edges to $H$ if necessary, we may assume $4m_1+3m_2=4\alpha n=\frac{4}{7}n$. 
	We denote by $U$ the set of vertices not covered by $\cT$. Obviously, there are no edges in $U$ and $|U|=\frac{3}{7}n.$ Thus we have $|U|\ge 3m_1$ and $|U|\ge \frac{9}{4}m_2$. Based on forbidden structures, we will bound the  number of edges in $H$ from above and derive a contradiction. 
	
	According to the size of intersections with $U$, 	we classify the edges into three categories: $(1,2)$ edges, $(2,1)$ edges and $(3,0)$ edges. Specifically, an edge is called an $(i,j)$ edge if it contains exactly $i$ vertices covered by $\cT$ and $j$ vertices in  $U$ (note that there is no $(0,3)$ edge). Denote by $D_{i}$ the set of all $(i,3-i)$ edges for $i\in[3]$.
	Note that \begin{equation}\label{eqno1_1}
		e(H)=|D_{1}|+|D_{2}|+|D_{3}|.
	\end{equation}
	To get an upper bound of $ e(H) $,  we calculate $|D_{i}|$ separately. 
	
	First, we consider the  (1,2) edges and show
	 \begin{equation}\label{eqno1_2}
		|D_{1}|\le m_1 \binom{|U|}{2}+3m_2|U|/2.
	\end{equation}
	
	\begin{proof}
		A (1,2) edge has exactly one vertex covered by $\cT$ that is in a $Y_{i}$ or an $E_{j}$. Then we fix a $Y_{i}$ or an $E_{j}$, where $i\in [m_1]$ and $j\in [m_2]$, and estimate the number of such edges.
		
		Denote the edges  of $Y_{i}$ by  $xyz,yzw$. Let $ Q' $ be the subgraph of $H[V(Y_i)\cup U]$ consisting of all (1,2) edges. We claim $\deg_{Q'}(x)=\deg_{Q'}(w)=0$. If not, the (1,2) edge containing the vertex $x$, together with $yzw$ forms a matching of size two in $H$, contradicting the maximality of $\cT$. Similarly, two disjoint edges in $Q'$ are also impossible, from which we bound the number of edges containing $ y $ or $ z $ as follows. 
		Indeed in $ Q' $, if one of $y,z$ is contained in more than $2|U|$ edges (and less than $\binom{|U|}{2}$ edges), then the degree of the other one is zero. Otherwise both  $ \deg_{Q'}(y)$ and $\deg_{Q'}(z) $ are at most $2|U|$, then we get 
		\[
		\deg_{Q'}(y)+\deg_{Q'}(z)\le 4|U| \le \binom{|U|}{2},
		\]
		as $|U|=3n/7$ is large. 
		%	Indeed in $ Q' $, if one of $y,z$ is contained in more than $2|U|$ edges, then the degree of the other one is zero. Otherwise both  $ \deg_{Q'}(y)$ and $\deg_{Q'}(z) $ are at most $2|U|$.  
		Therefore, 
		\[
		e(Q')=\deg_{Q'}(x)+\deg_{Q'}(y)+\deg_{Q'}(z)+\deg_{Q'}(w)\le \binom{|U|}{2}.
		\]
		Denote the vertex set of $E_{j}$ by $\{u,g,h\}$. Let $ Q'' $ be the subgraph of $H[\{u,g,h\}\cup U]$ consisting of all (1,2) edges. We observe $\deg_{Q''}(v)\le |U|/2$, $v\in \{u,g,h\}$. Suppose instead, there exist two edges  $vu_{1}u_{2},vu_{1}u_{3}$, where $ u_{1},u_{2},u_{3}\in U $, forming a copy of $Y$, contradicting the maximality. Hence \[e(Q'')=\deg_{Q''}(u)+\deg_{Q''}(g)+\deg_{Q''}(h)\le 3|U|/2.\]
		Summing over all $  Y_i$ and  $E_{j}$ of $\cT$, we have $|D_{1}| \le m_1 \binom{|U|}{2}+3m_2|U|/2$
		%	\[|D_{1}|=m_1e(Q')+ m_2e(Q'')  \le m_1 \binom{|U|}{2}+3m_2|U|/2, \]
		and complete the proof.
	\end{proof}
	
	The following three facts will be useful in our proof.
	\begin{fact}
		[\cite{Markstr2011Perfect}, Fact 1]
		\label{f2}
		For all integers $k\ge1, n\ge2$, and $1\le t \le n-1$, the maximum number of edges in a $k$-partite $k$-graph with $n$ vertices in each class and no matching of size $t+1$ is $tn^{k-1}$.
	\end{fact}	
	\begin{fact} \label{f3}
		Let $a,b$ be integers with $b\ge a\ge2$. Let $H$ be a $3$-partite $3$-graph on $V_1,V_2,V_3$ with $|V_1|=|V_2|=a$, $|V_3|=b$ and no matching of size $a$. Then $e(H)\le(a-1)ab$.
	\end{fact}	
	\begin{proof}
		Let $ V_3=\{v_1,\dots,v_b\}$ and let $V'_{3} := \{v_{i_1},v_{i_2},\dots,v_{i_a}\} \subseteq  V_3$. 
		For $H' : = H[V_1,V_2,V'_3]$, by Fact~\ref{f2}, we get $e(H')=\sum_{j=1}^{a}\deg_{H'}(v_{i_j})\le (a-1)a^{2}$.
		Summing over all subsets of size $a$ of $V_3$ gives
		\[
		e(H)\le \frac{\binom{b}{a}(a-1)a^{2}}{\binom{b-1}{a-1}}=(a-1)ab. \qedhere
		\]
		%	and complete the proof.
		%	Suppose $e(H)>(a-1)ab$, then it suffices to show that there exists $V'_{3}:=\{v_{i_1},v_{i_2},\dots,v_{i_a}\} \subseteq  V_3$ such that $\sum_{j=1}^{a}\deg_{H}(v_{i_j})>(a-1)a^{2}$$-- $by Fact~\ref{f1}, $H[V_1,V_2,V'_3]$ has a matching of size $a$, a contradiction.
		%	
		%	Suppose $\sum_{v\in V''_3}\deg_{H}(v)\le (a-1)a^{2}$ for any $V''_{3}\subseteq  V_3$ of size $a$. Summing over all subsets of size $a$ of $V_3$,	we have \[(a-1)ab<e(H)\le \frac{\binom{b}{a}(a-1)a^{2}}{\binom{b-1}{a-1}}=(a-1)ab,\]a contradiction.
	\end{proof}
		\begin{fact} \label{f5}
		Suppose $ Q=Q(V_1,V_2,V_3) $ is a $3$-partite $3$-graph with $ |V_1|=2$, $|V_2|=t\in \{3, 4\}$, $|V_3|=4 $. If $ Q $ does not have two disjoint copies of $Y$, then	$e(Q)\le 5t$.
	\end{fact}
	\begin{proof}
		
		Let $ V_{1}=\{a_{1},a_{2}\}$, $V_{2}=\{h_{1},\dots,h_t\}$ and $V_{3}=\{g_{1},g_{2},g_{3},g_4\}$. 
		Let $G:=G[V_{2},V_{3}]$ with $ E(G)=E_r(G)\cup E_b(G) $, where $ E_r(G)=\{h_ig_j:a_{1}h_ig_j\in E(Q)\}$, $E_b(G)=\{h_ig_j:a_{2}h_ig_j\in E(Q)\}$, and we color $ E_r(G) $ red, $ E_b(G) $ blue. Note that $ e(Q)=e(G)=e_r(G)+e_b(G)$, where $e_r(G):=|E_r(G)|$ and $e_b(G):=|E_b(G)|$. 
		%We claim that there are no two disjoint copies of  $ P_3 $ (path of length 2) with different colors, otherwise $ \{a_1\}\cup V( P_3')$ and $ \{a_2\}\cup V( P_3'') $ form two disjoint copies of $Y$ in $ Q $, a contradiction. 
		The assumption that $ Q $ does not have two disjoint copies of $Y$ implies that there are no two disjoint  monochromatic copies of  $ P_3 $ with different colors in $G$.
		
		Towards a contradiction, suppose that $e(Q)\ge 5t+1\ge 16$. 	Without loss of generality,  suppose  $e_r(G)\ge e_b(G)$, then $ 4t \ge e_r(G) \ge8$ and $ e_b(G) \ge t+1$.	
		%\red{(Please define $\deg_b(v)$ and $\deg_r(v)$ first.)}
		Denote $\deg_b(v)$ ($\deg_r(v)$) by the number of blue (red) edges containing $ v $ in $ G $. 
		We claim that for every $h_i$, \textbf{(A)} $\deg_{b}(h_i)\le 2$ and \textbf{(B)} $\deg_{r}(h_i)\le3$.
		
		%\noindent
		\textbf{(A)} Suppose  $\deg_{b}(h_1) \ge 3$, 
		then every $g_j$ has at most one neighbor other than $h_1$ in $ E_r(G) $, i.e. $\deg_{r}(g_j)\le2$, because of the absence of two disjoint  monochromatic copies of $ P_3 $ with different colors. 
		Since $ e_r(G) \ge8 $, we have $\deg_{r}(g_j)=2$ for every $j$. 
		Thus, $ e_r(G)=8 $ and  $\deg_{r}(h_1)=4$. 
		Now for every $h_i$, $i>1$, we have $\deg_{b}(h_i)\le1$, %\red{(explain more)}, 
		%as otherwise there exist two disjoint copies of $ P_3 $ with different colors,  that containing $ h_i, i\not=1 $ and $ h_1 $ separately, which is  a contradiction,
		as otherwise $h_i$ gives a blue $P_3$ and $h_1$ gives a vertex-disjoint red $P_3$ (due to $\deg_{r}(h_1)=4$), a contradiction.
		This implies that $ e_b(G) \le t-1+4=t+3 $, and $e(G)=e_b(G)+e_r(G)\le t+3+8<5t+1$, a contradiction.
		
		%\noindent 
		\textbf{(B)} Suppose $\deg_{r}(h_1)=4$, implying that $\deg_b(h_i)\le 1$ for $i\ge 2$, due to the absence of disjoint $P_3$ with different colors.
		Thus, $ e_b(G) \le t-1+2=t+1$ (using $\deg_b(h_1)\le 2$ from \textbf{(A)} and $\deg_b(h_i)\le 1$ for $i\ge 2$). 
		This implies that $e_b(G)=t+1$ and $e_r(G)=4t$ and there are two disjoint $P_3's$ with different colors, a contradiction.
		%Thus $ e_r(G)\ge 5b+1-(b+1)=4b $ and the contradiction is trivial.
		
		%We may assume that $\deg_{r}(h_i)\le3$ and $\deg_{b}(h_i)\le2$ for $i\in [b]$. 
		However, as $e(G)\ge5t+1$, there exists $i$ such that $\deg_{r}(h_i)+\deg_{b}(h_i)\ge6$, a contradiction.
	\end{proof}
	
	The analysis of $(2,1)$ and $(3,0)$ edges are much more involved.
	For that we introduce the following notation.
	Recall that $ \cT=\{Y_1,Y_2,\dots,Y_{m_{1}},E_1,E_2,\dots,E_{m_{2}}\} $.
	Let $ YYU $ represent the set of the edges $\{y_1,y_2,u\}$, $y_1\in V(Y_i)$, $y_2\in V(Y_j)$, $u\in U $, $ i\neq j $. 
	Let $ EYU $ represent the set of the edges $\{x,y,u\}$, $ x\in V(E_i)$, $y\in V(Y_j)$, $u\in U $ and let $ EEU $ be the set of the edges $\{x_1,x_2,u\}$, $ x_1\in V(E_i)$, $x_2\in V(E_j)$, $u\in U $, $ i\neq j $.  
	Note that the number of $(2,1)$ edges that have two vertices covered by a single $E_i$ or $Y_{j}$ of $\cT$ is $ O(n^2) $. 
	Hence 
	\[|D_{2}|=|YYU|+|EYU|+|EEU|+O(n^2).\]
	
	For (3,0) edges, we denote by $ EEY $ the set of the edges $\{x_1,x_2,y\}$, $x_1\in V(E_i)$, $x_2\in V(E_j)$, $y\in V(Y_k) $, $ i\neq j $, and define $EEE$, $EYY$ and $ YYY $ analogously. Other than these four types of edges, the number of remaining $(3,0)$ edges is $ O(n^2) $. Thus   \[|D_{3}|=|EEE|+|EEY|+|EYY|+|YYY|+O(n^2). \]
	
	Similar analysis as in the proof of~\eqref{eqno1_2} would give a good estimate on $|D_2|$, but this will leave $|D_3|$, the number of $(3,0)$ edges, hard to analyze.
	%Owing that (2,1) edges and (3,0) edges have close relation, we consider them in triples of $\cT$ at the same time.  %Triples of $T$ can be classified into four types: $\{E_i,E_j,E_k\}, \{E_i,E_j,Y_k\}, \{E_i,Y_j,Y_k\}$ and $\{ Y_i,Y_j,Y_k \}$. 
	To overcome this difficulty we consider $ |D_{2}|+|D_{3}| $ at the same time, by setting 
	\[\begin{split}
		&y_{1}:= 2|EEE|+|EEU|,\\
		&y_{2}:= 2|EEY|+|EEU|+|EYU|,\\
		& y_{3}:= 2|EYY|+|EYU|,\\
		& y_{4}:= 2|YYY|+2|YYU|.
	\end{split}\]
	Note that \begin{equation}
		2(|D_{2}|+|D_{3}|)=y_{1}+y_{2}+y_{3}+y_{4}+O(n^2)
	\end{equation}
	and we are able to bound each $y_i$ appropriately.
	More precisely, we shall prove the following:
	\begin{align*}
		&y_{1}\le  38\binom{m_2}{3}+3\binom{m_2}{2}|U|+O(n^2)\tag{6.4}\label{eqno1_4},\\
		&y_{2}\le 48\binom{m_2}{2}m_1+3\binom{m_2}{2}|U|+6m_1m_2|U|+O(n^2)\tag{6.5},\\
		&y_{3}\le 60\binom{m_1}{2}m_2+6m_1m_2|U|+O(n^2)\tag{6.6},\\
		&y_{4}\le  74\binom{m_1}{3}+14\binom{m_1}{2}|U|+O(n^2).	\tag{6.7}\label{eqno1_7}
	\end{align*}
	Combining \eqref{eqno1_1}-\eqref{eqno1_7}, we get
	\begin{align*}\label{eq1_9}
		e(H)&=|D_{1}|+|D_{2}|+|D_{3}|\\
		&\le m_1 \binom{|U|}{2}+7\binom{m_1}{2}|U|+3\binom{m_2}{2}|U|+6m_1m_2|U|\\&+37\binom{m_1}{3}+19\binom{m_2}{3}+30\binom{m_1}{2}m_2+24\binom{m_2}{2}m_1+O(n^2). \tag{6.9}
	\end{align*}
	Next we write $M_1=4m_1$, $M_2=3m_2$ and $|U|=3n/7=(3/4)(M_1+M_2)$, we rewrite all terms above with $M_1$ and $M_2$. After collecting terms we obtain
	\[\label{6.10}
	e(H)\le \frac{127}{64\times 6}(M_1+M_2)^3 + O(n^2) = \frac{127}{7^3\times 6}n^3+O(n^2) = \binom{n}{3}-\binom{n- n/7}{3}+O(n^2),
\tag{6.10}	\]
	where we used $M_1+M_2=4n/7$. (See Appendix B for a proof of this inequality).
	%Using $4m_1+3m_2= 4n/7$ and $|U|= 3n/7$, let $g(m_1)= f(m_1,-\frac{4}{3}m_1+\frac{4n}{21}) $ with $m_1\le {n}/{7}$.  
	%Compute the derivative of $g$, that is,
	%\[\begin{split}
	%g'(m_1)&=\binom{|U|}{2}-\frac{11}{3}m_1|U|+\frac{8n|U|}{21}-\frac{8m_1n}{189}-\frac{32n^{2}}{1323}-\frac{m_{1}^{2}}{54}+o(n^2)\\&=-\frac{m_{1}^{2}}{54}-\frac{305m_{1}n}{189}+\frac{611n^{2}}{2646}+o(n^{2}).
	%\end{split} \]
	%It is easy to check that $g'(m_1)\ge 0$ for $0\le m_1\le n/7$, so $g(m_1)$ is monotonic increasing for $m_1\in[0, n/7]$.  By \eqref{eq1_9}, we have
	%\[
	%e(H) \le f(m_1,m_2)= g(m_1) \le g( n/7 )\le \binom{n}{3}-\binom{n- n/7}{3}+o(n^3).
	%\]
	
	Now it remains to prove \eqref{eqno1_4}-\eqref{eqno1_7}.

	\def\EE{\tt EE}
	\def\EY{\tt EY}
	\def\YY{\tt YY}
	
	\subsection{More notation and tools}
	We need some more notation.
	Fix a triple $T=\{ P,S,R\}$ , where $P,S,R\in \cT$. Let $V(P)=V_{1},V(S)=V_{2},V(R)=V_{3} $. Let $Q_{T}$ be the  induced tripartite subgraph of $H$  on $(V_{1},V_{2},V_{3}) $, and let $G_{T}$ be the tripartite graph on $(V_{1},V_{2},V_{3}) $ with $(i, j)\in E(G_T)$ if and only if there are at least eight vertices $u$ in $U$, such that $ uij\in E(H) $, where $i\in V_{p}$, $j\in V_{q}$, $p,q\in[3]$ and $p\neq q$.  Let $EE$, $YY$ and $EY$ represent the edges of $G_T$ that are between two edges in $\cT$, between two copies of $Y$ and between an edge in $\cT$ and a copy of $Y$, respectively.
	% $\{x_1,x_2\}, x_1\in V(E_i), x_2\in V(E_j)$ ,  $\{y_1,y_2\}, y_1\in V(Y_i), y_2\in V(Y_j)$ and $\{x_1,y_2\}, x_1\in V(E_i), y_2\in V(Y_j)$ , respectively. 
	Denote the number of these edges by $|E_{\EE}(G_T)|$, $|E_{\YY}(G_T)|$ and $|E_{\EY}(G_T)|$.
	Thus we have $ e(G_T)=|E_{\EE}(G_T)|+|E_{\YY}(G_T)|+|E_{\EY}(G_T)|.$
	
	We denote by a star a graph all of whose edges share a common vertex that we call a center. We prove the following properties of $G_T$ and $Q_T $.
	
	\begin{fact} \label{f4}
		Fix a triple $T=\{ P,S,R\}$ of $ \cT $. Let $G=G_T$ as defined above. If $	|V_{1}|=3$, then $ G[V_{1},V_{2}] $ is trivial or a star, and $e(G[V_{1},V_{2}])\le|V_{2}|$. 
		If $|V_{1}|=|V_{2}|=4$, then $G[V_{1},V_{2}] $ has no matching of size three and $e(G[V_{1},V_{2}])\le8$. 
	\end{fact}	
	\begin{proof}
		Suppose that $|V_{1}|=3$ and the nontrivial graph $G[V_{1},V_{2}] $ is not a star, then it contains two disjoint edges $e_1$, $e_2$. 
		%By the definition of $G$, we can find two disjoint copies of $Y$ on $V_{1} \cup V_{2} \cup U$, contradicting the maximality of $\cT$. Thus $ G[V_{1},V_{2}] $ is a star or is trivial and thus $e(G[V_{1},V_{2}])\le|V_{2}|$.
		By the definition of $G$, we can find vertices $u_1,\dots, u_4$ such that $e_1u_1, e_1u_2$, $e_2u_3, e_2u_4\in E(H)$, which gives two disjoint copies of $Y$ on $V_{1} \cup V_{2} \cup U$.
		Thus, we can replace $P$ and $S$ by these two copies of $Y$, and since $|V_1|=3$, the resulting $\{Y, E\}$-tiling is larger than $\cT$, a contradiction. 
		Thus $ G[V_{1},V_{2}] $ is a star or is trivial and thus $e(G[V_{1},V_{2}])\le|V_{2}|$.
		
		Suppose $|V_{1}|=|V_{2}|=4$. %Take two leaves $ a\in V_1,b\in V_2$.  If $ab\in E(G[V_{1},V_{2}])$, then there exist $u_1,u_2\in U$ such that $a,b,u_1,u_2$ form a copy of $Y$ and  $V_{1}\setminus\{a\},V_{2}\setminus\{b\} $ form two edges in $H$, a contradiction. 
		We claim that $G[V_{1},V_{2}]$ has no matching of size three. Suppose instead, then similarly we can find three disjoint copies of $Y$ on $V_{1} \cup V_{2} \cup U$.
		Replacing $P$ and $S$ by them gives a larger $\{Y, E\}$-tiling, a contradiction. 
		By K\H onig Theorem\cite{2000Graph}, $G[V_{1},V_{2}]$ has a vertex cover of size at most two, implying that $ e(G[V_{1},V_{2}])\le 8 $.
		%So $ e(G[V_{1},V_{2}])\le 8 $, otherwise the minimum cardinality of a vertex cover of $ E(G[V_{1},V_{2}]) $ is at least three, by K\H onig Theorem, $G[V_{1},V_{2}]$ has a matching of size three, a contradiction.
	\end{proof}
	
	Note that $ G_{T}[V_{1},V_{2}] $ is a star or the union of two stars. We call a center of a star in  $ G_{T}[V_{1},V_{2}] $, a center of $G_{T}[V_{1},V_{2}] $. Then $G_{T}[V_{1},V_{2}] $ has at most two centers.

	%\subsection{}
	\noindent\textbf{Bound} \bm{$y_1$}.
Recall $y_{1}:= |EEU|+2|EEE|$.
	We show
	\[y_{1} \le  38\binom{m_2}{3}+3\binom{m_2}{2}|U|+O(n^2).\]
	
	\begin{proof}
		Without loss of generality, consider a triple $\{E_1,E_2,E_3\}$ in $\cT$.  Let $ V_{1}=V(E_1)=\{a_{1},a_{2},a_{3}\}$, $V_{2}=V(E_2)=\{b_{1},b_{2},b_{3}\}$, $V_{3}=V(E_3)=\{c_{1},c_{2},c_{3}\}$. 
		Let $G=G_T$ and $Q=Q_T$ as defined above.  We consider three cases here.
		
		\textbf{Case 1.} $ e(G)=0 $. Obviously, $ e(Q)\le 27 $.
		
		\textbf{Case 2.} $ e(G)\ge1 $ and $G$ has no matching of size two. 
		
		Note that $G$ is a star, and $ e(G)\le6$. 
		Let $a_{1}b_{1}\in E(G)$ and $Q':=Q-\{a_{1},b_{1}\}$. Since $\{a_{1}b_{1}\}\cup U$ has a copy of $Y$ and $\cT$ is a maximum $ \{Y,E\} $-tiling, $Q'$ has no matching of size two. By Fact~\ref{f3}, we get 	$ e(Q')\le 6 $.
		
		Let $ E_{a_{1},b_{1}} =\{xyz\in E(Q):\{x,y,z\}\cap\{a_{1},b_{1}\}\not= \emptyset\}$.
		Note that $ |E_{a_{1},b_{1}}|\le15. $ Thus we get
		\[e(Q)= e(Q')+|E_{a_{1},b_{1}}|\le 6+15=21.\]
		
		\textbf{Case 3.} $G$ has a matching of size two.
		
		By Fact~\ref{f4}, we know that $ G[V_{i},V_{j}] $ is a star and $e(G[V_{i},V_{j}])\le3$, where $i,j\in[3]$ and $i\not= j$. So $ e(G)\le9$.  Without loss of generality, we may assume that  $a_{1}b_{1},b_{2}c_{2}\in E(G)$. By the maximality of $\cT$, there is no edge in $Q-\{a_{1},b_{1}, b_{2},c_{2}\}$.
		Similar as Case 2 by  Fact~\ref{f3}, we have $e(Q-\{a_{1},b_{1}\})\le6$ and $e(Q-\{b_{2},c_{2}\})\le6$. 
		Let 
		\[E_{pq} := \{pqv\in E(Q): p\in V_i, q\in V_j, v\in V_k, \{i,j,k\}=[3] \}.\]
		%Then the number of the remaining edges in $Q$ are
		Note that 
		$|E_{a_{1}b_{2}}| \le3$, $|E_{a_{1}c_{2}}|\le 3$ and $|E_{b_{1}c_{2}}|\le 3$.	
		%	\[|E_{a_{1}b_{2}}|:=|\{a_{1}b_{2}c_{i}\in E(Q):i\in [3] \}|\le3,\]
		%	\[|E_{a_{1}c_{2}}|:=|\{a_{1}b_{j}c_{2}\in E(Q):j\in [3] \}|\le3,\]
		%	\[|E_{b_{1}c_{2}}|:=|\{a_{k}b_{1}c_{2}\in E(Q):k\in [3] \}|\le3.\]
		Thus	
		\[\begin{split}e(Q)&=|E(Q-\{a_{1},b_{1}\})\cup E(Q-\{b_{2},c_{2}\})|+ |E_{a_{1}b_{2}} \cup E_{a_{1}c_{2}} \cup E_{b_{1}c_{2}}|%- |\{a_{1}b_{2}c_{2},a_{1}b_{1}c_{2}\}|
			\\&\le 6\times2+3\times3-2=19.
		\end{split}\]
		
		Note that each triple $\{E_i, E_j, E_k\}$ in $\cT$ satisfies one of the cases above.
		Suppose the number of triples in three cases above are $x_1$, $x_2$ and $x_3$ respectively, then $x_1+x_2+x_3=\binom{m_2}{3}$. 
		Note that we can bound the number of $EEU$ edges by $\frac{1}{m_2-2}(6x_2+9x_3)|U|+O(n^2)$, because every $EEU$ edge counted by $G$ is counted in exactly $m_2-2$ triples (so computed $m_2-2$ times) and the number of $EEU$ edges not reflected in $G$ is at most $ 7\times 9 \binom{m_2}{2}$.
	Thus we have
		\[\begin{split}y_{1}&\le 	2 (27x_1+21x_2+ 19x_3)+\frac{1}{m_2-2}(6x_2+9x_3)|U|+O(n^2) \\
			&\le 38\binom{m_2}{3}+16x_1+4x_2+3\binom{m_2}{2}|U|-\frac{9x_1+3x_2}{m_2-2}|U|+O(n^2).
		\end{split}\]
		We complete the proof by $\frac{|U|}{m_2-2}\ge \frac{|U|}{m_2}\ge \frac{3n/7}{4n/21}=9/4$, which gives $16x_1+4x_2-\frac{9x_1+3x_2}{m_2-2}|U|<0$.
		%The error term (at most $ 7\times 9 \binom{m_2}{2}$) is  the edges in $ EEU $ that are not reflected in $G$.
	\end{proof}
	
	\noindent\textbf{Bound} \bm{$y_2$}.
 Recall 
	$y_{2}:= 2|EEY|+|EEU|+|EYU|.$ 
	We show \[y_{2}\le 48\binom{m_2}{2}m_1+3\binom{m_2}{2}|U|+6m_1m_2|U|+O(n^2).\]
	\begin{proof}	
		Without loss of generality, consider $E_1,Y_2,E_3\in \cT$. Let $ V_{1}=V(E_1)=\{a_{1},a_{2},a_{3}\}$, $V_{2}=V(Y_2)=\{h_{1},h_{2},h_{3},h_4\}$ and $V_{3}=V(E_3)=\{c_{1},c_{2},c_{3}\}$. Let $G=G_T$ and $Q=Q_T$ as defined before. By Fact~\ref{f4}, we get $|E_{\EY}(G)|\le8$ and $|E_{\EE}(G)|\le3$. 
		In this proof we shall use the fact that there do not exist an edge $e$ of $G$, a $3$-edge and a copy of $Y$ in $Q$, all of which are pairwise disjoint -- since $e$ can be extended to a copy of $Y$ by two vertices in $U$, we may replace $E_1,Y_2,E_3$ by these two copies of $Y$ and one edge, contradicting the maximality of $\cT$.  
		Consider the following three cases.
		
		\textbf{Case 1.} $ |E_{\EY}(G)|=0 $. We use the trivial bound $ e(Q)\le 36 $.
		
		\textbf{Case 2.} $ |E_{\EY}(G)|\ge1 $ and $G[V_{1},V_{2}]\cup G[V_{2},V_{3}]$ has no matching of size two. 
		
		Note that $G[V_{1},V_{2}]\cup G[V_{2},V_{3}]$ is a star and thus $ |E_{\EY}(G)|\le6. $ Let $ a_1h_1\in E(G) $. By the maximality of $\cT$, the graph $Q':=Q-\{ a_1,h_1\}$ does not have an edge and a copy of $Y$ that are disjoint. Then we claim $ e(Q')\le 12$. 
		If not, without loss of generality, we suppose $ \deg_{Q'}(c_{3})\ge \deg_{Q'}(c_{2})\ge \deg_{Q'}(c_{1}) $, then $ 5\le \deg_{Q'}(c_{3})\le6$. 
		We may assume $ a_{2}h_{i}c_{3}\in E(Q') $ for every $i\in\{2,3,4\}$.
		Since 
		\[
		\deg_{Q'}(c_{2})+ \deg_{Q'}(c_{1})\ge7>6\ge |\{a_2h_ic_j\in E(Q'):i\in \{2,3,4 \} ,j\in \{1,2\}\}|,
		\]
		we can pick an edge $e\in Q'$ such that $a_3\in e$ and $c_{3}\notin e$. 
		By our assumption $c_{3}$ together with the remaining vertices in $ V_{1}$ and $V_{2} $ forms a copy of $Y$ in $Q'$ (disjoint with $e$), a contradiction.
		
		Let $ E_{a_{1},h_{1}} =\{xyz\in E(Q):\{x,y,z\}\cap\{a_{1},h_{1}\}\not= \emptyset\}$.
		Note that $ |E_{a_{1},h_{1}}|\le18$. Thus we get
		\[e(Q)= e(Q')+|E_{a_{1},h_{1}}|\le 12+18=30.\]

		\textbf{Case 3.} $G[V_{1},V_{2}]\cup G[V_{2},V_{3}]$ has a matching of size two. 	
		
		Suppose $a_{1}h_{1}, h_{2}c_{2}\in E(G)$ and $ Q':=Q-\{a_{1},h_{1}\} $. Note that $E(Q-\{a_{1},h_{1}, h_{2},c_{2}\})=\emptyset$. Then for any $e \in E(Q')$, $e \cap \{h_{2},c_{2}\}\not=\emptyset$. We claim that $e(Q')\le8$. Indeed, suppose that $ e(Q')\ge 9 $ and $ \deg_{Q'}(h_{2})\ge \deg_{Q'}(c_{2}) $, then  $ 5\le \deg_{Q'}(h_{2})\le 6 $ and $ \deg_{Q-\{a_{1}, h_{1}, h_2\}}(c_{2})\ge 3$. Take a copy of $Y$ containing $ h_{2}, c_1, c_3 $ in $ Q' $. 
		Since $ \deg_{Q-\{a_{1}, h_{1}, h_2\}}(c_{2})\ge 3$, one can take an edge containing $c_2$ disjoint with the copy of $Y$ above, a contradiction. 
		%(Otherwise there exists an edge containing $ c_{2} $ which is disjoint from the copy of $Y$, a contradiction.) 
		Similar arguments give $e(Q-\{h_{2},c_{2}\})\le8$.
		
		Similar as Case 3 in the proof of \eqref{eqno1_4} to estimate the number of the remaining edges in $Q$, we get
		$|E_{a_{1}h_{2}}|\le 3$, $|E_{a_{1}c_{2}}|\le 4$ and $|E_{h_{1}c_{2}}|\le 3$.
		%	\[|E_{a_{1}h_{2}}|:=|\{a_{1}h_{2}c_{i}\in E(Q):i\in \{1,2,3\} \}|\le3, |E_{a_{1}c_{2}}|\le4,|E_{h_{1}c_{2}}|\le3\]
		%$$|E_{a_{1}c_{2}}|:=|\{a_{1}h_{i}c_{2}\in E(Q):i\in \{1,2,3,4\} \}|\le4$$
		%$$|E_{h_{1}c_{2}}|:=|\{a_{i}h_{1}c_{2}\in E(Q):i\in \{1,2,3\} \}|\le3$$ 
		Then 	
		\[\begin{split}
			e(Q)&=|E(Q-\{a_{1}, h_{1}\})\cup E(Q-\{h_{2}, c_{2}\})|+ |E_{a_{1}h_{2}}\cup E_{a_{1}c_{2}} \cup E_{h_{1}c_{2}}|%- |\{a_{1}h_{2}c_{2},a_{1}h_{1}c_{2}\}|
			\\&\le 8\times2+3+4+3-2=24.
		\end{split}\]
		
		Suppose that the number of triples of $ \cT $ in the three cases above are $x_1$, $x_2$ and $x_3$ respectively, then $x_1+x_2+x_3=m_1\binom{m_2}{2}$. Note that by counting edges via triples e.g. $E_1,Y_2,E_3$, all $EYU$ edges reflected by $G$ are counted exactly $m_2-1$ times. Putting everything together we get 
		\[\begin{split}
			y_{2}&\le 	2 (36x_1+30x_2+ 24x_3)+3\binom{m_2}{2}|U|+\frac{1}{m_2-1}(6x_2+ 8x_3)|U|+O(n^2) \\
			&\le 48\binom{m_2}{2}m_1 + 24x_1 + 12 x_2 +3\binom{m_2}{2}|U|+6m_1m_2|U|-\frac{12x_1+6x_2}{m_2}|U|+O(n^2),
		\end{split}\]
		where the error term is the number of the remaining edges in $ EEU $ and $ EYU $ not reflected in $G$.
		We obtain the desired bound on $y_2$ because $24x_1 + 12 x_2 -\frac{12x_1+6x_2}{m_2}|U|<0$ due to  $|U|/m_2\ge 9/4$.
	\end{proof}
	
	\noindent\textbf{Bound} \bm{$y_3$}.
  Recall $y_{3}:= 2|YYE|+|EYU|$.
	We show \[y_{3}\le 60\binom{m_1}{2}m_2+6m_1m_2|U|+O(n^2).\]
	
	\begin{proof}
		Without loss of generality, consider $E_1,Y_2,Y_3\in \cT$.
		Let $ V_{1}=V(E_1)=\{a_{1},a_{2},a_{3}\}$, 
		$V_{2}=V(Y_2)=\{h_{1},h_{2},h_{3},h_4\}$ and $V_{3}=V(Y_3)=\{g_{1},g_{2},g_{3},g_4\}$. 
		Let $G=G_T$ and $Q=Q_T$ as defined above. 
		By Fact~\ref{f4}, we have $|E_{\EY}(G)|\le8$. 
		In this case we use the fact that there do not exist $i$ edges of $G$, and $3-i$ copies of $Y$ in $Q$, all of which are pairwise disjoint -- since each edge of $G$ can be extended to a copy of $Y$ by distinct vertices in $U$, we may replace $E_1,Y_2, Y_3$ by these three copies of $Y$, contradicting the maximality of $\cT$.  
		
		Consider the following three cases here.
		
		\textbf{Case 1.} $ e(G)=0 $.  The trivial bound is $ e(Q)\le 48 $.
		
		\textbf{Case 2.} $ e(G)\ge1$ and $|E_{\EY}(G)|\le6$. 
		Note that the bound below $e(Q)\le 39$ does not use the assumption $|E_{\EY}(G)|\le6$.
		
		First assume $ a_1h_1 \in E(G)$ and $Q'=Q-\{a_{1},h_{1}\}$. By Fact~\ref{f5}, we have $e(Q')\le15$.
		Let $ E_{a_{1}, h_1} :=\{xyz\in E(Q):\{x,y,z\}\cap\{a_{1}, h_1\}\not= \emptyset\}$.
		Note that $ |E_{a_{1},h_{1}}|\le24. $ Thus we get
		\[ e(Q)= e(Q')+|E_{h_{1},h_{1}}|\le 15+24=39.\]
		
		Next we assume $ h_4g_4 \in E(G)$. Since $Q-\{h_{4}, g_4\}$ has no matching of size three, by Fact~\ref{f2}, we have $e(Q-\{h_{4}, g_4\})\le 18$. The number of edges containing $ h_4$ or $g_4 $ is at most 21. Then
		\[e(Q)\le 18+21=39.\]
		
		\textbf{Case 3.} $ e(G)\ge1 , 7\le |E_{\EY}(G)|\le8$. 
		By Fact~\ref{f4}, each of $ E(G[V_{1},V_{2}])$ and $E(G[V_{1},V_{3}])$ is a star.
		Depending on the location of the centers and without loss of generality, we may separate the following cases and show that $e(Q)\le 36$ in each of them.
		%	We consider $|E_{\EY}(G)|=7$.
		
		\medskip
		\noindent\emph{Case 3.1.}
		Suppose that $ E(G[V_{1},V_{2}])=\{a_ph_1,a_ph_2,a_ph_3,a_ph_4\}$ and $E(G[V_{1},V_{3}])\supseteq\{a_qg_1,a_qg_2,a_qg_3\}$. 
		
		If $p\not=q$, then $ Q-\{a_p,a_q\} $ has no copy of $Y$ -- otherwise we obtain two edges of $G$ and a copy of $Y$ in $Q$ such that they are pairwise disjoint, a contradiction. This implies that $ Q-\{a_p,a_q\} $ has at most 4 edges.
		Note that the number of edges containing $ a_p$ or $a_q$ is at most 32. Thus $e(Q)\le 32+4=36.$
		
		If $p=q$, then let $ Q''= Q-\{a_p\} $. We claim that $e(Q'')\le20$. Suppose  instead, by Fact~\ref{f5}, there  exist  $Y',Y''$ that are two disjoint copies of $Y$ in $ Q'' $. 
		Then as $\deg_G(a_p)\ge 7$ there exists $x \notin V(Y')\cup V(Y'')$ such that $a_px\in E(G)$ is disjoint with $Y',Y''$, a contradiction.
		%	Choose $a_px\in E(G)$, $x \notin V(Y')\cup V(Y'')$, there exist $u_1,u_2 \in U$ together with $ a_p,x $ forming another copy of $Y$, a contradiction. 
		Since the number of edges containing $ a_p$ is at most 16, we have $e(Q)\le 20+16=36.$
		
		\medskip
		\noindent\emph{Case 3.2.}
		Suppose $ E(G[V_{1},V_{2}])\supseteq \{a_1h_1,a_1h_2,a_1h_3\}$ and $E(G[V_{1},V_{3}])=\{g_1a_1,g_1a_2,g_1a_3\}$. 
		Let $ Q''':=Q-\{a_1,g_1\} $. Then we have $\deg_{Q'''}(a_i)\le 3$, $i=2,3$ -- otherwise $Q'''$ contains a copy of $Y$, say, containing $a_2$ but not $a_3$, and at most two vertices in $V_2$.
		We can take $g_1a_3$ and $a_1h_i$ for some $h_i$ disjoint with this copy of $Y$, which is a contradiction.
		%$ Q''' $ does not have a copy of $Y$, denoted by $Y^{(1)}$, in which there is exactly one vertex of $V_1$. 
		%If not, we can find $a_ph_i,a_jg_q\in E(G)$, where $ h_i,a_j \notin V(Y^{(1)})$. By the definition of $G$, there exist $u_1,u_2,u_3,u_4  \in U$ such that $ a_ph_iu_1u_2,a_jg_qu_3u_4 $ and $Y^{(1)}$  form three disjoint copies of $Y$, which contradicts the maximality of $\cT$. 
		Thus $e(Q''')\le 6$. 
		Since the number of edges containing $ a_1$ or $g_1$ is at most 24. Thus $e(Q)\le 24+6=30.$
		
		%Hence, in this case, $e(Q)\le 36.$ 	
		
		\medskip	
		Suppose the number of triples e.g. $E_1, Y_2, Y_3$ in $\cT$ in three cases above are $x_1$, $x_2$ and $x_3$ respectively, then $x_1+x_2+x_3=m_2\binom{m_1}{2}$. 	Note that by counting edges via triples e.g. $E_1,Y_2, Y_3$, all $EYU$ edges reflected by $G$ are counted exactly $m_1-1$ times. Putting everything together we get 
		
		\[\begin{split}
			y_{3}&\le 	2 (48x_1+39x_2+36x_3)+\frac{1}{m_1-1}(6x_2+8x_3)|U|+O(n^2) \\
			&\le 60\binom{m_1}{2}m_2 + (36x_1+18x_2+12x_3) + 6m_1m_2|U| - (12x_1+6x_2+4x_3)\frac{|U|}{m_1}+O(n^2),
		\end{split}\]
		where $ O(n^2) $ is the number of the remaining edges in  $ EYU $ not reflected in $G$.
		We obtain the desired bound as $|U|/m_1\ge (3n/7)/(n/7)=3$ and thus $(36x_1+18x_2+12x_3) - (12x_1+6x_2+4x_3)\frac{|U|}{m_1}\le 0$.
	\end{proof}

	\noindent\textbf{Bound} \bm{$y_4$}.
	Recall $y_{4}:= 2|YYY|+2|YYU|$. 
	We show \[y_{4}\le  74\binom{m_1}{3}+14\binom{m_1}{2}|U|+O(n^2).\]
	
	\begin{proof}
		Without loss of generality, consider $Y_1,Y_2,Y_3\in \cT$.	 Let $ V_{1}=V(Y_1)=\{k_{1},k_{2},k_{3},k_4\}$, $V_{2}=V(Y_2)=\{h_{1},h_{2},h_{3},h_4\}$ and $V_{3}=V(Y_3)=\{g_{1},g_{2},g_{3},g_4\}$. Let $G=G_T$ and $Q=Q_T$ as defined before. %We will consider five cases here.
		
		First we give the following claims on the structures of $G$ and $Q$.
		\begin{claim} \label{c1}
			The number of edges in $G$ is at most $21$. When $e(G)=21$, there exist $i_1,i_2,i_3\in [4]$ such that $k_{i_{1}},h_{i_{2}},g_{i_{3}}$  cover all edges of $G$ with $\deg_{G}(v)=7$ for $v\in \{k_{i_{1}},h_{i_{2}},g_{i_{3}}\}$, see Figure~\ref{figure1}.
		\end{claim}
		
		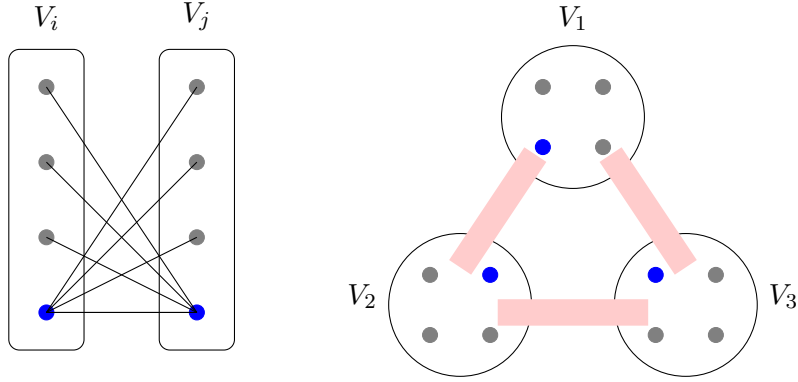
\begin{figure}[h]
			\begin{center}
				\begin{tikzpicture}
					[inner sep=2pt,
					vertex/.style={circle, draw=black!50, fill=black!50},
					rect/.style={rectangle, inner sep=7,minimum size=0.8},
					]

					\draw[rounded corners] (-1.5,-0.5) rectangle (-0.5,3.5);
					\node at (-1,3.6) [label=above:$V_i$] {};
					\node at (-1,3) [vertex] {};
					\node at (-1,2) [vertex] {};
					\node at (-1,1) [vertex] {};
					\node at (-1,0) [vertex,color=blue] {};
					\draw[rounded corners] (0.5,-0.5) rectangle (1.5,3.5);
					\node at (1,3.6) [label=above:$V_j$] {};
					
					\node at (1,3) [vertex] {};
					\node at (1,2) [vertex] {};
					\node at (1,1) [vertex] {};
					\node at (1,0) [vertex,color=blue] {};
					
					\draw (-1,0) -- (1,0);
					\draw (-1,0) -- (1,1);
					\draw (-1,0) -- (1,2);
					\draw (-1,0) -- (1,3);
					\draw (-1,0) -- (1,0);
					\draw (-1,1) -- (1,0);
					\draw (-1,2) -- (1,0);
					\draw (-1,3) -- (1,0);
					\draw  (6,2.6) circle (0.95);
					\node at (6,3.6) [label=above:$V_1$] {};
					\node at (5.6,2.2) [vertex,color=blue] {};
					\node at (5.6,3.0) [vertex] {};
					\node at (6.4,2.2) [vertex] {};
					\node at (6.4,3) [vertex] {};
					\draw  (4.5,0.1) circle (0.95);
					\node at (3.2,-0.1) [label=above:$V_2$] {};
					\node at (4.1,-0.3) [vertex] {};
					\node at (4.1,0.5) [vertex] {};
					\node at (4.9,-0.3) [vertex] {};
					\node at (4.9,0.5) [vertex,color=blue] {};
					
					\draw  (7.5,0.1) circle (0.95);
					\node at (8.8,-0.1) [label=above:$V_3$] {};
					\node at (7.1,-0.3) [vertex] {};
					\node at (7.1,0.5) [vertex,color=blue] {};
					\node at (7.9,-0.3) [vertex] {};
					\node at (7.9,0.5) [vertex] {};
					\draw[line width=10pt, color=red!20] (5.5,2.1)--(4.5,0.6);
					\draw[line width=10pt, color=red!20] (6.5,2.1)--(7.5,0.6);
					\draw[line width=10pt, color=red!20] (5.0,0)--(7,0);
				\end{tikzpicture}
				
				\caption{The structure of $ G $ when $ e(G)=21 $}
				\label{figure1}	
			\end{center}
		\end{figure}

		\begin{proof}
			Without loss of generality, let $e(G[V_{1},V_{2}])\ge \max\{e(G[V_{2},V_{3}]),e(G[V_{1},V_{3}])\}$.
			By Fact~\ref{f4}, we know that $G[V_{1},V_{2}] $ has no matching of size three and $e(G[V_{1},V_{2}])\le8$. 
			Now suppose $e(G)\ge 21$. %and we shall show that $G$ is exactly as described in the claim.
			First assume $e(G[V_{1},V_{2}]) =  8$, and we also have $e(G[V_{1},V_{3}])\ge 5$ and $e(G[V_{2},V_{3}])\ge 5$.
			Note that $ G[V_{1},V_{2}] $ has two centers, and they must be in the same part (otherwise $e(G[V_{1},V_{2}]) \le 7$).
			Without loss of generality, suppose the two centers are in $ V_1 $. 
			Take disjoint $e_1,e_2\in E(G[V_{2},V_{3}])$ and  $e_3,e_4\in E(G [V_{1},V_{2} \setminus(e_1\cup e_2)])$ that form a matching of size four in $G$, which can be extended to four disjoint copies of $Y$ by adding vertices in $U$, contradicting the maximality of $\cT$.
			This contradiction implies that $e(G[V_{1},V_{2}])= e(G[V_{2},V_{3}]) =e(G[V_{1},V_{3}]) =7$.
			%Similarly, there are no two centers of $ G[V_{1},V_{2}] $ in $ V_2 $ as well.
			%Thus, we know that $e(G[V_{1},V_{3}]) \le 4$ or $e(G[V_{2},V_{3}]) \le 4$ or $e(G[V_{1},V_{2}]) \le 7$.
			
			%		Suppose $e(G)>21,$ $e(G[V_{1},V_{2}])\ge e(G[V_{2},V_{3}]),e(G[V_{1},V_{3}])$. By Fact~\ref{f4}, $ e(G[V_{1},V_{2}])\le 8 $. %and two leaves like $ k_1,h_1 $ does not form an edge in $G$.  
			%		Then $ e(G[V_{1},V_{2}])=8 $, $e(G[V_{2},V_{3}])\ge6$.  Suppose two centers of $ G[V_{1},V_{2}] $ are  in $ V_1 $. Take disjoint $ e_1,e_2\in E(G[V_{2},V_{3}]) $ and  $e_3,e_4\in E(G [V_{1},V_{2} \setminus(e_1\cup e_2)])$ that form a matching of size four in $G$, contradicting the maximality of $\cT$.
			
			%By the analysis of the last paragraph, the number of edges in $G$ is at most $21$.
			%When $e(G)=21$, the edge distribution is $ e(G[V_{1},V_{2}])=e(G[V_{2},V_{3}])=e(G[V_{1},V_{3}])=7$.
			%When $e(G)=21.$ Similar as the last paragraph,  $ e(G[V_{1},V_{2}])=8$ is impossible. Thus $ e(G[V_{1},V_{2}])=e(G[V_{2},V_{3}])=e(G[V_{1},V_{3}])=7$.  
			We next claim that $G$ has at most three center vertices, from which the claim easily follows. 
			Towards a contradiction, assume that $G$ has at least four center vertices, and without loss of generality, we may assume that $k_1,k_2 \in V_1$ are the centers, and $k_1$ is a center of $G[V_{1},V_{p}]$, $k_2$ is a center of $G[V_{1},V_{q}]$, $p,q\in \{2, 3\}$ and it is possible that $p=q$.
			%Then we consider two cases: (i) $w_1,w_2$ are two centers of $G[V_1, V_j]$ for $j=2,3$; (ii) $w_1$ is the center of $G[V_1, V_2]$ and $w_2$ is the center of $G[V_1, V_3]$.
			Let $ e_1,e_2\in E(G[V_{2},V_{3}]) $ be two disjoint edges. 
			Since the degree of each center is at least three (as each bipartite graph has seven edges), for $p\neq q$, we can take two disjoint edges one containing $k_1$ and the other containing $k_2$ which together with $e_1, e_2$ form a matching of size four.
			In the case $p=q$, as  $e(G[V_{1},V_{p}])=7$, one of $k_1$ and $k_2$ has degree three and the other has degree four in  $G[V_{1},V_{p}]$ and thus it is also possible to pick two disjoint edges which together with $e_1, e_2$ form a matching of size four. 
			Clearly the existence of the matching of size four in $G$ contradicts the maximality of $\cT$ and we are done.
			%Then there exists a matching of size two in $ G[V_{1},V_{2}]\cup G[V_{1},V_{3}] $ denoted by $ a_1w_1,a_2w_2 $, where $a_1,a_2\notin e_1\cup e_2$,  which contradicts the absence of a matching of size four  in $G$. Thus there are exactly three centers  that cover all edges of $G$.
			%Since the degree of each center in $ G[V_{i},V_{j}] $ is at least three, $ \deg_{G[V_{i_1},V_{i_2}]}(w_1)+\deg_{G[V_{i_1},V_{i_2}]}(w_2)=7 $ or $ \deg_{G[V_{i_1},V_{i_2}]}(w_1)\ge 3, \deg_{G[V_{i_1},V_{i_3}]}(w_2)\ge 3 $. 
		\end{proof}	
		
		\begin{claim} \label{c2}
			If  $e(G)\ge18$ and $\{ v_{1},v_{2},v_{3}\}$ is a vertex  cover  of $G$, %then $ khg\notin E(Q)$ for any $k\in V_1,h\in V_2,g\in V_3$, $k,h,g\notin \{ v_{1},v_{2},v_{3}\}$.  
			then $\{ v_{1},v_{2},v_{3}\}$ is also a vertex  cover  of $Q$, therefore, $e(Q)=|E_{v_{1},v_{2},v_{3}}|$, where $E_{v_{1},v_{2},v_{3}} :=\{xyz\in E(Q):\{x,y,z\}\cap\{v_{1},v_{2},v_{3}\}\not= \emptyset\}.$
		\end{claim}	
		\begin{proof}
			Fix any $k\in V_1$, $h\in V_2$, $g\in V_3$, $k,h,g\notin \{ v_{1},v_{2},v_{3}\}$.  We will show $ khg\notin E(Q) $. 
			Without loss of generality, let $ \deg_{G}(v_{1})\ge \deg_{G}(v_{2}) \ge \deg_{G}(v_{3}) $. We separate three cases below depending on the location of $v_1$, $v_2$ and $v_3$.
			
				Suppose $ v_{1},v_{2},v_3\in V_1$, then $ e(G)\le16 $, which is a contradiction.
				
					Suppose $ v_{1},v_{2}\in V_1,v_{3}\in V_3$. Note that $ e(G[V_1,V_2]\cup G[V_1,V_3])\le 16$. Then $\deg_{G[V_2,V_3]}(v_{3})\ge2>1$,  there exists $v_{3}w'_{3}\in E(G)$ such that $w'_{3}\in V_{2}$, $w'_{3}\not=h$.  Since $\deg_{G}(v_{2})\ge5>4$, $\deg_{G}(v_{1})\ge6>5$. Similar as the previous paragraph, we conclude that $ khg\notin E(Q) $. 
				
			Suppose $ v_{i}\in V_i$ for $i\in [3]$.  Let $G'=G-\{v_{1}v_{2},v_{2}v_{3},v_{1}v_{3}\}$. Note that $ e(G')\ge 18-3=15$. Then $\deg_{G'}(v_{3})\ge3>2$,  there exists $v_{3}w'_{3}\in E(G')$, where $w'_{3}\notin \{ k,h\}$. Similarly,  $\deg_{G'}(v_{2})\ge5>3$, $\deg_{G'}(v_{1})\ge6>5$, there exist $ w'_{1},w'_{2} \notin \{ k,h,g,w'_{3}\}$, such that  $v_{2}w'_{2}\in E(G')$, $v_{1}w'_{1}\in E(G')$. Thus $G$ has three pairwise disjoint edges which are also disjoint with $\{k,h,g\}$. By the maximality of $\cT$, $ khg\notin E(Q) $.

			Therefore, we conclude that $e(Q)= |E_{v_{1},v_{2},v_{3}}|$.
		\end{proof}
		For convenience, by a cross matching, we mean a matching $e_1,e_2,e_3$ in $G_T$ such that  $e_1\subseteq V(Y_1) \times V(Y_2)$, $e_2\subseteq V(Y_2)\times V(Y_3)$, $e_3\subseteq V(Y_3)\times V(Y_1)$. 
		Next we show the existence of a cross matching.
		\begin{claim} \label{c3}
			If $e(G)\ge17,$ then there is a cross matching in $G$.
		\end{claim}
		\begin{proof}
			%Suppose that $G$ has no cross matching. 
			We may assume $e(G[V_{1},V_{2}])\ge e(G[V_{1},V_{3}])\ge e(G[V_{2},V_{3}])$, then $e(G[V_{1},V_{2}])\ge6.$  Suppose the two centers of $G[V_{1},V_{2}]$ are $u$ and $v$, and thus they have degree at least two in $G[V_{1},V_{2}]$. 
			It suffices to find a matching $\{e_1, e_2\}$ of size two such that $ e_1\in E(G[V_{2},V_{3}])$, $e_2\in E(G[V_{1},V_{3}])$ and $|( e_1\cup e_2)\cap\{u,v\}|\le1$.
			Indeed, as $e_1\cup e_2$ contains one vertex in $V_1$ and in $V_2$, respectively, for $u$ or $v$ which is not contained in $e_1\cup e_2$, we can pick an edge $e_3$ in $G[V_{1},V_{2}]$ disjoint from $e_1\cup e_2$.
		Hence we obtain a cross matching.
			%Note that there is no bad matching, otherwise there exists $ e_3\in E(G[V_{1},V_{2}]) $ together with  $ e_1,e_2 $ forming a cross matching, a contradiction.

			We first assume that $u$ and $v$ are in the same part. Note that $e(G[V_{1},V_{3}])\ge5$, $e(G[V_{2},V_{3}])\ge1$. Then we take disjoint $ e_1\in E(G[V_{2},V_{3}])$, $e_2\in E(G[V_{1},V_{3}])$ and clearly $|\{u,v\}\cap( e_1\cup e_2)|\le1$. 
			
			Now suppose $v\in V_1$ and $u\in V_2$, then $e(G[V_{1},V_{2}])\le7,$ $e(G[V_{1},V_{3}])\ge5$ and $e(G[V_{2},V_{3}])\ge3$.  
			%If $G[V_{2},V_{3}]$ has an edge $e$ not containing $v$, then as $e(G[V_{1},V_{3}])\ge5$ we can 
			As $e(G[V_{1},V_{3}])\ge5$, there exists $wv'\in  E(G)$ such that $w\in V_3$, $v'\in V_1$ and $v'\neq v$.
			Then we are done unless all edges in $ E(G[V_{2},V_{3}]) $ contain $w$. 
			But this implies that there exists an edge $e_1=wu'\in  E(G[V_{2},V_{3}])$ with $u'\in V_2\setminus \{u\}$, namely, $u\notin e_1$.
			As $e(G[V_{1},V_{3}])\ge5$, we can pick an edge $e_2$ in $G[V_{1},V_{3}]$ not containing $w$ and we are done.
			%, similarly, thus each edge of $ E(G[V_{1},V_{3}]) $ contains $u'$, which contradicts $ e(G[V_{1},V_{3}])\ge5$.
		\end{proof}	
		
		We consider five cases here.
		
		\textbf{Case 1.} $ e(G)=0 $. The trivial bound is $ e(Q)\le 64 $.
		
		\textbf{Case 2.} $ 1\le e(G)\le16$. 
		
		Suppose $ k_3h_1\in E(G) $, then $Q'=Q-\{k_3,h_1\}$ has no matching of size  three. By Fact~\ref{f3},  $e(Q')\le24.$ Let $ E_{k_3,h_1} =\{xyz\in E(Q):\{x,y,z\}\cap\{k_3,h_1\}\not= \emptyset\}$.
		Note that $ |E_{k_3,h_1}|\le28. $ Hence
		\[
		e(Q)= e(Q')+|E_{k_3,h_1}|\le 24+28=52.
		\]

		\textbf{Case 3.} $  e(G)=17$ .
		
		By Claim~\ref{c3}, let $h_1k_1,k_2g_2,h_3g_3\in E(G)$ and $ V'=\{h_1,k_1,k_2,g_2,h_3,g_3\} $. 
		By the definition of $G$, $ V' $ together with six vertices of $U$ forms three copies of $Y$, so we have $ e(Q-V')=0 $. 
		Let $Q_1=Q-\{k_2,g_2,h_3,g_3\}$.
		% and note that $e(Q_1)=\deg_{Q_1-k_1}(h_1)+\deg_{Q_1-h_1}(k_1) + \deg_{Q_1}(h_1k_1)$. 
		We may assume that $Q_1$ has no matching of size two -- otherwise two disjoint edges together with $k_2g_2$ and $h_3g_3$ would contradict the maximality of $\cT$.
		Now we claim that $e(Q_1)=\deg_{Q_1-k_1}(h_1)+\deg_{Q_1-h_1}(k_1) + \deg_{Q_1}(h_1k_1)\le6$. 
		First, clearly $\deg_{Q_1}(h_1k_1)\le2$.
		Note that if one of $\deg_{Q_1-k_1}(h_1)$ and $\deg_{Q_1-h_1}(k_1)$ is nonzero, then the other is at most two.
		Indeed, if $\deg_{Q_1-k_1}(h_1)$ is nonzero, say, $h_1g_3k_3\in E(Q_1)$, then $\deg_{Q_1-h_1}(k_1)\le 2$ as the only possible edges are $k_1g_3h_3$ and $k_1g_3h_4$.
		This implies that $\deg_{Q_1-k_1}(h_1)+\deg_{Q_1-h_1}(k_1)\le 4$ and our claim is proved.
		%Suppose instead, as $\deg_{Q_1}(h_1k_1)\le2$, we assume $\deg_{Q_1-k_1}(h_1)\ge3$, then $\deg_{Q_1-h_1}(k_1)\ge1$ and $\deg_{Q_1-k_1}(h_1g_i)\ge1,i=1,4$. 
		%Thus there exists a matching $k_1h_{i_1}g_{j_1}, h_1k_{i_2}g_{j_2}\in E(Q_1)$, which contradicts the existence of $k_2g_2$ and $h_3g_3$ of $E(G)$ (corresponding to two disjoint copies of $Y$).%But there are two disjoint copies of $Y$ corresponding with  $k_2g_2,h_3g_3$, which contradicts the 
		
		Since $Q_{2}:=Q-\{h_1,k_1\}$ has no matching of size  three, using Fact~\ref{f3}, we get $e(Q_2)\le24.$ Let $ E_{k_1, h_1} =\{xyz\in E(Q):\{x,y,z\}\cap\{k_1,h_1\}\not= \emptyset\}$. 
		First, $|E_{k_1, h_1}\setminus E(Q_1)|\le 4^3-3\times3\times4-10=18$, where we subtract from all possible edges ($4^3$) the number of edges not containing $k_1$ or $h_1$ ($3\times3\times4$ of them) and the number of edges containing $ h_1$ or $k_1 $ on $ V(Q_1)  $ ($10$ of them). % not containing any of $k_2,g_2,h_3,g_3$ ($10$ of them).
		Thus, by $e(Q_1)\le6$ we get $ |E_{k_1,h_1}|= |E_{k_1, h_1}\setminus E(Q_1)| + e(Q_1)\le 24$. Hence
		\[
		e(Q)= e(Q_2)+|E_{k_1,h_1}|\le 24+24=48.
		\]

		\textbf{Case 4.} $ 18\le e(G)\le20$ . 
		
		The following claim characterizes the structure of $G$. Intuitively, if $G$ has no vertex cover of size three, then it is likely to contain a matching of size four, which can be extended to four disjoint copies of $Y$, contradicting the maximality of $\cT$.
		
		\begin{claim} \label{c4}
			$G$ has a vertex cover of size three.
		\end{claim}
		
		\begin{proof}
			First suppose that there exist $  v_i,v_j \in V(G) $ such that $G-\{v_i,v_j\}$ has no matching of size two.
			Then $G-\{v_i,v_j\}$ is a star or a triangle.
			If $G-\{v_i,v_j\}$ is a star, then the claim holds.
			%\red{(Please add a graph to illustrate the following proof before Fact~\ref{f6}.)}
			So we assume $G-\{v_i,v_j\}$ is a triangle. 
			Let $ v_i\in V_{i_1},v_j\in V_{i_2} $. 
			Since $G[V_i,V_j]$ has no matching of size three, when $ i_1=i_2 $, letting $ i_1=1 $, then $e(G[V_1,V_i])\le 3+2+1=6$, $i=2,3$, and $e(G[V_2,V_3])=1$ (see Figure~\ref{figure2}). Thus $ e(G)\le 6+6+1<18 $. When $ i_1\neq i_2 $, 
			we have $e(G[V_{i_1},V_{i_2}])\le6$, $e(G[V_{i_1},V_k])\le5$ and $e(G[V_{i_2},V_k])\le5$ for $k\not=i_1,i_2$, so $ e(G)\le 5+5+6<18 $, a contradiction. 
			
		Thus in the rest of the proof we may assume the following.
			
			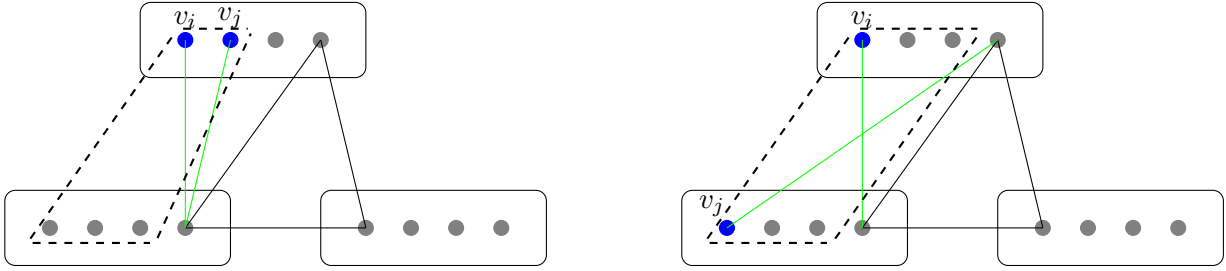
\begin{figure}[h]
				\begin{center}
					\begin{tikzpicture}
						[inner sep=2pt,
						vertex/.style={circle, draw=black!50, fill=black!50},
						rect/.style={rectangle, inner sep=7,minimum size=0.8},
						]

						\draw[rounded corners] (5,5) rectangle (8,4);
						%	\node at (5,5) [label=above:$V_i$] {};
						\node at (5.6,4.5) [vertex,color=blue] {};
						\node at (7.4,4.5) [vertex] {};
						\node at (6.8,4.5) [vertex] {};
						\node at (6.2,4.5) [vertex,color=blue] {};
						%	\draw[rounded corners] (0.5,-0.5) rectangle (1.5,3.5);
						\node at (6.2,4.5) [label=above:$v_j$] {};	
						\node at (5.6,4.5) [label=above:$v_i$] {};	
						
						\draw[rounded corners] (3.2,2.5) rectangle (6.2,1.5);
						\node at (3.8,2) [vertex] {};
						\node at (4.4,2) [vertex] {};
						\node at (5.0,2) [vertex] {};
						\node at (5.6,2) [vertex] {};
						
						\draw[rounded corners] (7.4,2.5) rectangle (10.4,1.5);
						\node at (8.0,2) [vertex] {};
						\node at (8.6,2) [vertex] {};
						\node at (9.2,2) [vertex] {};
						\node at (9.8,2) [vertex] {};
						
						\draw (7.4,4.5) -- (5.6,2);
						\draw (7.4,4.5) -- (8.0,2);
						\draw (5.6,2) -- (8.0,2);
						\draw [color=green](6.2,4.5) -- (5.6,2);
						\draw [color=green](5.6,4.5) -- (5.6,2);
						
						\draw[-, shorten >=2pt, shorten <=2pt, bend right=45, thick, dashed] (3.5,1.8)--(5.5,4.65);
						\draw[-, shorten >=2pt, shorten <=2pt, bend right=45, thick, dashed] (5.2,1.8)--(6.5,4.65);
						\draw[-, shorten >=2pt, shorten <=2pt, bend right=45, thick, dashed] (5.5,4.65)--(6.5,4.65);
						\draw[-, shorten >=2pt, shorten <=2pt, bend right=45, thick, dashed] (3.5,1.8)--(5.2,1.8);
						
						\draw[rounded corners] (14,5) rectangle (17,4);
;p0,p						\node at (14.6,4.5) [vertex,color=blue] {};
						\node at (16.4,4.5) [vertex] {};
						\node at (15.8,4.5) [vertex] {};
						\node at (15.2,4.5) [vertex] {};
						\node at (14.6,4.5) [label=above:$v_i$] {};

						\draw[rounded corners] (12.2,2.5) rectangle (15.2,1.5);
						\node at (12.8,2) [vertex,,color=blue] {};
						\node at (13.4,2) [vertex] {};
						\node at (14.0,2) [vertex] {};
						\node at (14.6,2) [vertex] {};
						\node at (12.6,2) [label=above:$v_j$] {};	
						
						\draw[rounded corners] (16.4,2.5) rectangle (19.4,1.5);
						\node at (17.0,2) [vertex] {};
						\node at (17.6,2) [vertex] {};
						\node at (18.2,2) [vertex] {};
						\node at (18.8,2) [vertex] {};
						
						\draw (16.4,4.5) -- (14.6,2);
						\draw  (16.4,4.5) -- (17.0,2);
						\draw (14.6,2) -- (17.0,2);
						\draw [color=green](14.6,4.5) -- (14.6,2);
						\draw [color=green](16.4,4.5) -- (12.8,2);
						
						\draw[-, shorten >=2pt, shorten <=2pt, bend right=45, thick, dashed] (12.5,1.8)--(14.5,4.65);
						\draw[-, shorten >=2pt, shorten <=2pt, bend right=45, thick, dashed] (14.2,1.8)--(16.2,4.65);
						\draw[-, shorten >=2pt, shorten <=2pt, bend right=45, thick, dashed] (14.5,4.65)--(16.2,4.65);
						\draw[-, shorten >=2pt, shorten <=2pt, bend right=45, thick, dashed] (12.5,1.8)--(14.2,1.8);
					\end{tikzpicture}
					
					\caption{An illustration of the proof of Claim~\ref{c4}, where the dotted region must induce a star and thus has at most three edges.}
					\label{figure2}	
				\end{center}
			\end{figure}
			
			%		\begin{fact} \label{f6}
			\begin{itemize}
				\item[($\dagger$)] For any $  v_i,v_j \in V(G) $, $G-\{v_i,v_j\}$ has a matching of size two.
			\end{itemize}
			%		\end{fact}
			
			Next  we separate the proof into two cases in terms of the maximum degree.
			
			\textbf{Case a.} $\Delta(G) \ge 6 .$
			
			Without loss of generality, suppose $ 6\le \deg_{G}(v_1) \le 8$. Let $G^{(1)} :=G-v_1$. Obviously, 
			\[ \label{eqno1_10} e(G^{(1)})= e(G)-\deg_{G}(v_1)\ge 18-8=10.  \tag{6.11} \]
			We claim that $G^{(1)}$ has no matching of size three. 
			Indeed, if $G^{(1)}$ has a matching $M$ of size three, then since $G[V_i,V_j]$ has no matching of size three, $N(v_1)\setminus V(M)\neq \emptyset$.
			Thus there exists $w\in V(G)$ such that $v_1w\in E(G)$ which together with $M$ forms a matching of size four in $G$, a contradiction.

			\begin{fact}\label{f7}
				%\red{ (Some $G$s' should be replaced by $G^{(1)}$ in the statement and the proof.)}
				For any $v_2\in V(G)$, $ v_2 \ne v_1 $ and $i,j \in [3]$, we have $\deg_{G^{(1)}[V_i,V_j]}(v_2) \le 2$ and $\deg_{G^{(1)}}(v_2) \le 3$.
			\end{fact}			
			\begin{proof}By ($\dagger$),  $G-\{v_1,v_2\}$ has a matching $M'$ of size two.	
				
				Towards a contradiction, suppose $\deg_{G^{(1)}[V_i,V_j]}(v_2) \ge 3$, then there exists $w'\in V(G^{(1)})$ such that $v_2w'\in E(G^{(1)})$ together with $M'$  forming a matching of size three in $G^{(1)}$, a contradiction. We have $\deg_{G^{(1)}[V_i,V_j]}(v_2) \le 2$ and thus $\deg_{G^{(1)}}(v_2) \le 4$.
				
				Next suppose $\deg_{G^{(1)}}(v_2) = 4$, then for any choice of $M'$, we must have $ M'\subseteq E(G[N_{G^{(1)}}(v_2)])$. If not, suppose   $M'$ covers $v'\notin N_{G^{(1)}}(v_2)$. Then  there exists $w'$ such that $v_2w'\in E(G^{(1)})$ together with $M'$  forming a matching of size three in $G^{(1)}$ (see Figure~\ref{figur3}), which is a contradiction.  Thus $ E(G^{(1)}-\{v_2\})\subseteq E(G[N_{G}(v_2)])$. Note that $e(G[N_{G^{(1)}}(v_2)])\le4$. Hence, $ e(G^{(1)})\le 4+4<10 $,  which contradicts~\eqref{eqno1_10}. \end{proof}
			
			\begin{figure}[h]
				\begin{center}
					\begin{tikzpicture}
						[inner sep=2pt,
						vertex/.style={circle, draw=black!50, fill=black!50},
						rect/.style={rectangle, inner sep=7,minimum size=0.8},
						]

						\draw[rounded corners] (5,5) rectangle (8,4);
						%	\node at (5,5) [label=above:$V_i$] {};
						\node at (5.6,4.5) [vertex] {};
						\node at (7.4,4.5) [vertex] {};
						\node at (6.8,4.5) [vertex] {};
						\node at (6.2,4.5) [vertex,color=blue] {};
						%	\draw[rounded corners] (0.5,-0.5) rectangle (1.5,3.5);
						\node at (6.2,4.5) [label=above:$v_2$] {};

						\draw  (4.7,2) circle (0.95);
						\node at (4.3,1.6) [vertex] {};
						\node at (4.3,2.4) [vertex] {};
						\node at (5.1,1.6) [vertex] {};
						\node at (5.1,2.4) [vertex] {};
						
						\draw  (8.3,2) circle (0.95);
						\node at (7.9,1.6) [vertex] {};
						\node at (7.9,2.4) [vertex] {};
						\node at (8.7,1.6) [vertex] {};
						\node at (8.7,2.4) [vertex] {};
						
						\draw (6.2,4.5) -- (7.9,1.6);
						\draw (6.2,4.5) -- (7.9,2.4);
						\draw (6.2,4.5) -- (5.1,1.6);
						\draw (6.2,4.5) -- (5.1,2.4);
						
						\draw[-, shorten >=2pt, shorten <=2pt, bend right=45, thick, dashed] (7.9,1.6)--(5.1,1.6);
						\draw[-, shorten >=2pt, shorten <=2pt, bend right=45, thick, dashed] (7.9,1.6)--(5.1,2.4);
						\draw[-, shorten >=2pt, shorten <=2pt, bend right=45, thick, dashed] (7.9,2.4)--(5.1,2.4);
						\draw[-, shorten >=2pt, shorten <=2pt, bend right=45, thick, dashed] (7.9,2.4)--(5.1,1.6);
					\end{tikzpicture}
					
					\caption{Any edge except the four dotted ones would result a matching of size two not completely in $  N_{G^{(1)}}(v_2) $.}
					\label{figur3}	
				\end{center}
			\end{figure}
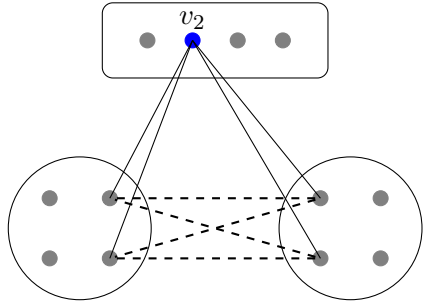

			%Since  $    G^{(1)}$ has a matching of size two but has no matching of size three. 
			Now we resume the proof of \textbf{Case a}.
			Suppose $\{w_1w_2,w_3w_4\}$ is a matching of size two in $ G^{(1)} $. Since  $G^{(1)}$ has no matching of size three, $e(G[V(G^{(1)})\setminus \{w_1,w_2,w_3,w_4\}])=0$. 
			Note that $ G^{(1)} $ has no vertex cover of size three, otherwise by Fact~\ref{f7}, $ e(G^{(1)})\le3\times 3=9 $, which contradicts~\eqref{eqno1_10}. 
			Thus every $w_i$, $i\in [4]$, is incident an edge outside $\{w_1,w_2,w_3,w_4\}$, implying that 
			 there exist $ u_1,u_2,u_3,u_4 $ such that $w_iu_i\in E(G^{(1)})$. 
			By the constraint of the matching number, $ u_1=u_2$, $u_3=u_4 $ and $ \deg_{G^{(1)}}(w_i)=2$, $i\in[4] $, i.e.  $    G^{(1)}$ is the union of two  triangles $w_1w_2u_1$ and $w_3w_4u_2$.
			Hence, $ e(G^{(1)})\le 3+3<10 $,  contradicting~\eqref{eqno1_10}.

			\medskip		
			\textbf{Case b.}  $\Delta(G) \le 5 .$
			
			By Claim~\ref{c3}, let $v_{i}w_{i}\in E(G)$, $i\in [3] $ be a cross matching, $V'=\{ v_{1},v_{2},v_{3},w_{1},w_{2},w_{3}\}$ and $V''=(V_1\cup V_{2} \cup V_3) \setminus V'$. Since there is no matching of size four in $G$, $e(G[V''])=0$ and for $\deg_{G[V''\cup\{v_{i}\}]}(v_{i})$ and  $\deg_{G[V''\cup\{w_{i}\}]}(w_{i})$, either both of them are one with the same neighborhood in $V''$, or one of them is zero.  
			\begin{claim} \label{c5}
				$ e(G[V''\cup\{a_{i}\}])=0$, $i\in [3] $, where $a_i\in \{v_i,w_i\}$.
			\end{claim}
			\begin{proof} Note that there exists $a_i\in \{v_i,w_i\}$, $a_j\in \{v_j,w_j\}$ such that $ e(G[V''\cup\{a_i\}])=e(G[V''\cup\{a_j\}])=0$. If not, suppose $G$ has two triangles, $w_1v_1a$ and $w_2v_2b$, where $a,b \in V'',a\not= b$. Then $G[V_i,V_j]$ which contains $v_{3}w_{3}$ has a matching of size three, a contradiction.
				
				Suppose  $e(G[V''\cup\{w_1\}])=e(G[V''\cup\{w_2\}])=0$.  Then there exists $a_3\in \{v_3,w_3\}$ such that $ e(G[V''\cup\{a_3\}])=0$. Suppose instead, i.e. 	$e(G[V''\cup\{v_{3}\}])=e(G[V''\cup\{w_{3}\}])=1 $.
				Obviously, $e(G[v_3,w_1,w_2,w_3])\le 5$. Since $\deg_G(v_i) \le 5$ for $i\in[2]$, we get $ e(G)\le 5+5\times2+2=17<18,$ a contradiction.\end{proof}

			By Claim~\ref{c5}, suppose $ e(G[V''\cup\{w_{i}\}])=0$, $i\in [3] $. Note that $ e(G-\{v_1,v_2,v_3\})=e(G[\{w_1,w_2,w_3\}])$ $\le3 $.
			Since  $\deg_{G}(v_i)\le 5$, $i\in [3] $, $e(G)\ge 18$, $ w_1,w_2,w_3 $ form a triangle and  $\deg_{G}(v_i)= 5$, $i\in [3] $.  This further implies that $ \{v_1,v_2,v_3\} $ is an independent set. Thus $\deg_{G[V']}(v_i)\le 2$ and $ e(G[V''\cup\{v_{i}\}]) \ge3$, $ i\in[3]$,  implying that there exists a matching $M:=\{w_1w_2,v_1x,v_2y,v_3z	\}$ in $G$,  where $x\in V''$, $y\in  V''\setminus \{x\}$, $z\in  V'' \setminus \{x,y\}$,  a contradiction.  Consequently,  $\Delta(G) \le 5$ is impossible.
			
			Claim~\ref{c4} is proved.
		\end{proof}

		By Claim~\ref{c4}, $G$ has a vertex cover  $\{ v_{1},v_{2},v_{3}\}$.
		Note that $v_{1},v_{2},v_{3}$ can not be in the same $V_i$ as otherwise $e(G)\le 8+8=16<18$.
		By Claim~\ref{c2}, $e(Q)\le |E_{v_{1},v_{2},v_{3}}|$.
		If $ v_{i}\in V_i$, $i\in[3]$, then we have  $  e(Q)\le 64-27=37 .$
		If $ v_{1},v_{2}\in V_1,v_{3}\in V_3$, then we have $  e(Q)\le 64-24=40 .$

		\textbf{Case 5.} $ e(G)=21$ .
		
		By Claim~\ref{c1}, $G$ has a vertex cover $\{k_{i_{1}},h_{i_{2}},g_{i_{3}}\}$. By Claim~\ref{c2}, 
		\[e(Q)=|E_{k_{i_{1}},h_{i_{2}},g_{i_{3}}}|\le37.\]

		\medskip
		Suppose the number of triples in each of cases above are $x_1,x_2,x_3,x_4,x_5$ respectively,  then $x_1+x_2+x_3+x_4+x_5=\binom{m_1}{3}$. 
		Note that by counting edges via triples e.g. $Y_1,Y_2, Y_3$, all $YYU$ edges reflected by $G$ are counted exactly $m_1-2$ times. Putting everything together we get 
		\[\begin{split}y_{4}&\le 	2 (64x_1+52x_2+48x_3+40x_4+ 37x_5)+\frac{1}{m_1-2}(16x_2+17x_3+20x_4+21x_5)2|U|+O(n^2) \\
			&\le 74\binom{m_1}{3} + 54x_1+30x_2+22x_3+6x_4 +14\binom{m_1}{2}|U| - (21x_1+5x_2+4x_3+x_4)\frac{2|U|}{m_1} + O(n^2),
		\end{split}\]
		where the error term comes from the remaining edges in  $ YYU $ which are not reflected in $G$.
		By $|U|/m_1\ge 3$, we have $54x_1+30x_2+22x_3+6x_4 - (21x_1+5x_2+4x_3+x_4)\frac{2|U|}{m_1}\le 0$ and we are done.
	\end{proof}

	We conclude this section by deriving Theorem~\ref{edge th2} from Theorem~\ref{edge th1} and the regularity method.
	
	\begin{proof}[Proof of Theorem~\ref{edge th2}]
		Define constants such that $0<1/n\ll \varepsilon'\ll d, \varepsilon$. By Theorem~\ref{ineg}, $ H $ contains  a $Y$-tiling $M'$ of size $ \varepsilon' n $. Apply the regularity lemma to $ H':=H[V(H)\setminus V(M')] $, let $ V_{1},\dots,V_{K} $ be the clusters of the partition obtained. Note that $ e(H')\ge \binom{n-4\varepsilon' n}{3}-\binom{(n-4\varepsilon' n)(1-\alpha)}{3}+o(n^3) $. By Lemma ~\ref{deg lem}, the reduced hypergraph $R$ satisfies $|E(R)|\ge\binom{K}{3}-\binom{K-\alpha K}{3}+o(K^3).$	Applying Theorem~\ref{edge th1}, we get a $\{Y,E\}$-tiling $\cT=\{Y_1,Y_2,\dots,Y_{m_{1}},E_1,E_2,\dots,E_{m_{2}}\} $ in  $ R$, such that $4m_1+3m_2\ge 4\alpha K$.

		For each  $ Y_i\in  \cT$  with edges $i_1i_2i_3$, $i_2i_3i_4$, $E_j=\{j_1,j_2,j_3\}\in  \cT $, let $F_{Y}:=\lceil (1-2\varepsilon')m\rceil$, $F_{E}:=\frac{3}{4}\lceil (1-2\varepsilon')m\rceil$, where $m$ is the size of each cluster. Now construct a $Y$-tiling $M''$  in $H'$, by greedily adding copies of $Y$ to $M''$ using the regularity.
		
		Recall that the Tur\'an density of every $k$-partite $k$-graph is zero. Let $H'_{1}:=H'[V_{i_1},V_{i_2},V_{i_3}]$, $H'_{2}:=H'[V_{i_2},V_{i_3},V_{i_4}]$ and $H'_{3}:=H'[V_{j_1},V_{j_2},V_{j_3}]$. %be the $(3,3)$-complexes defined as $ (*) $. %$ \bigcup\limits_{i=1}^{k-1}\hat{P}^{i}(J) \cup (E(H)\cap \mathcal{K}(\hat{P}^{k-1}(J))) $. Note that $H^{*}_{i}, i=1,2,3$ is  $(\mathbf{d},\delta_{k},\delta,r)$-regular.  %Let $ Y^{\le} $ be the $ (3,3) $-complex corresponding to $ Y_{3,2} $. 
		In $ H'_{1} $ and $ H'_{2} $, %applying the embedding lemma,  
		we repeatedly find disjoint copies of $Y$ each of which intersects each of $V_{i_2}$ and $V_{i_3}$ in one vertex and $V_{i_1}$ (or $V_{i_4}$) in two vertices, until $M$ contains precisely $F_Y$ disjoint copies of $Y$ in $ H'[V_{i_1},V_{i_2},V_{i_3},V_{i_4}]$ (see Figure~\ref{figure:main_thm}). In $ H'_{3} $, we find a copy of $ K_{4,4,4} $ and decompose it into three copies of $ Y $.
		%applying the embedding lemma, we find a copy of $Y$ that intersects $ V_{j_1} $ in two vertices, a copy of $Y$ that intersects $ V_{j_2} $ in two vertices, a copy of $Y$ that intersects $ V_{j_3} $ in two vertices alternately,  until $M$ contains precisely $F_{E}$ disjoint copies of $Y$ in $ H[V_{j_1},V_{j_2},V_{j_3}]$ (see the right picture in Figure~\ref{figure:main_thm}). 
		
		Note that at each stage of this process, the number of vertices in each $V_i\in V(R)$ that would be covered by $M''$ is at most
		$(1-2\varepsilon')m+2\le (1-\varepsilon')m $.	Obtained by deleting the vertices which has been covered by $M''$,  the subgraphs of $H'_{i}$ for $i\in [3]$ are still regular. 
		%$(\mathbf{d},\sqrt{\delta_{k}},\sqrt{\delta},r)$-regular. 
		Thus %as to find the first copy of $Y$, 
		it is possible to successively add copies of $ Y$ to $M''$ in order to obtain a $Y$-tiling $M''$ as desired.
		Note that the size of $M''$ is 
		\[\begin{split}
			{m_1} F_{Y} + {m_2} F_{E}&\ge (1-2\varepsilon')m(m_1+\frac{3}{4}m_2)\ge (1-2\varepsilon')m\alpha K\\&\ge (1-2\varepsilon')\alpha (n-4\varepsilon' n)\ge(1-6\eps')\alpha n\ge \alpha n-\varepsilon' n,\end{split}\] where we used $ \alpha \le 1/7 $.

Consider the union  of $ M' $ and $ M'' $.	Thus $H$ has a $Y$-tiling of size at least $ \alpha n$, i.e. covering  at least  $ 4\alpha n$ vertices.
	\end{proof}

	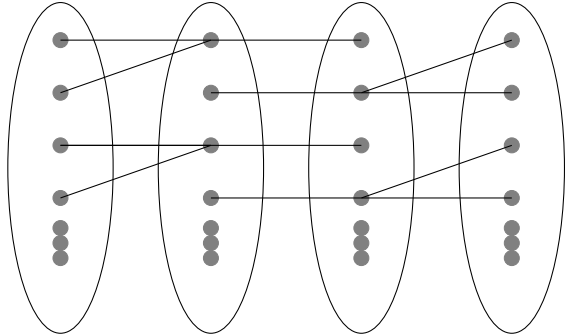
\begin{figure}[h]
		
		\begin{tikzpicture}
			[inner sep=2pt,
			vertex/.style={circle, draw=black!50, fill=black!50},
			%   rect/.style={rectangle, inner sep=7,minimum size=0.8},
			]
			
			\draw (-1,0) ellipse (0.7 and 2.2);
			\draw (1,0) ellipse (0.7 and 2.2);
			\draw (3,0) ellipse (0.7 and 2.2);
			\draw (5,0) ellipse (0.7 and 2.2);
			\node at (-1,1.7) [vertex] {};
			\node at (-1,1) [vertex] {};
			\node at (-1,0.3) [vertex] {};
			\node at (-1,-0.4) [vertex] {};
			\node at (-1,-0.8) [vertex] {};
			\node at (-1,-1) [vertex] {};
			\node at (-1,-1.2) [vertex] {};
			\node at (1,1.7) [vertex] {};
			\node at (1,1) [vertex] {};
			\node at (1,0.3) [vertex] {};
			\node at (1,-0.4) [vertex] {};
			\node at (1,-0.8) [vertex] {};
			\node at (1,-1) [vertex] {};
			\node at (1,-1.2) [vertex] {};
			\node at (3,1.7) [vertex] {};
			\node at (3,1) [vertex] {};
			\node at (3,0.3) [vertex] {};
			\node at (3,-0.4) [vertex] {};
			\node at (3,-0.8) [vertex] {};
			\node at (3,-1) [vertex] {};
			\node at (3,-1.2) [vertex] {};
			\node at (5,1.7) [vertex] {};
			\node at (5,1) [vertex] {};
			\node at (5,0.3) [vertex] {};
			\node at (5,-0.4) [vertex] {};
			\node at (5,-0.8) [vertex] {};
			\node at (5,-1) [vertex] {};
			\node at (5,-1.2) [vertex] {};
			\draw (-1,1.7) -- (1,1.7);
			\draw (-1,1) -- (1,1.7);
			\draw (1,1.7) -- (3,1.7);
			\draw (1,1) -- (3,1);
			\draw (3,1) -- (5,1.7);
			\draw (3,1) -- (5,1);
			
			\draw (-1,0.3) -- (3,0.3);
			\draw (-1,0.3) -- (1,0.3);
			\draw (-1,-0.4) -- (1,0.3);
			\draw (1,-0.4) -- (3,-0.4);
			\draw (3,-0.4) -- (5,0.3);
			\draw (3,-0.4) -- (5,-0.4);
			
		\end{tikzpicture}
		\caption{A $ Y $-tiling in the clusters corresponding to a copy of $ Y $ in the reduced graph}	
		\label{figure:main_thm}
	\end{figure}

	\section{The minimum degree condition of $Y_{k,b}$-tiling}
	%\subsection{add definition of fractional tiling}
	For $k>b\ge 0$, recall that  $Y_{k,b}$ is the $k$-graph consisting of two edges that intersect in exactly $b$ vertices. 
For $ p>0 $,	fix two $k$-graphs $F$ of order $p$ and $H$, let $\mathcal F_{F,H} \subseteq \binom {V(H)}{p}$ be the family of $p$-sets in $V(H)$ that span a copy of $F$.
	A fractional $F$-tiling in $H$ is a function $\omega: \mathcal F_{F,H} \to [0,1]$ such that for each $v\in V(H)$ we have $\sum_{v\in e\in  \mathcal F_{F,H}}\omega(e) \le 1$. Then $\sum_{e\in  \mathcal F_{F,H}}\omega(e)$ is the size of $\omega$. 
	Such a fractional $F$-tiling is called perfect if it has size $n/p$.
	
	When $F$ is a single edge, a fractional $F$-tiling is also called a fractional matching.
	The size of the largest fractional matching in a $k$-graph $H$ is denoted by $\nu^*(H)$. Its dual problem is to find a minimum fractional vertex cover $ \tau^*(H)=\sum_{v\in V}\omega(v)$ over all functions $ \omega:V\to [0,1] $ such that for each $e\in E$ we have $\sum_{v\in e}\omega(v) \ge 1$. Then by the Duality Theorem, $ \nu^*(H)=\tau^*(H)$.
	Let $f_d^s(F,n)$ be the minimum integer $m$ so that every $n$-vertex $k$-graph $H$ with $\delta_d(H)\ge m$ has a fractional $F$-tiling of size $s$. In particular,  $\delta_0(H)=e(H)$.
	
	Alon et al.~\cite{AFHRRS} proved the following result for fractional matchings.
	We extend it to fractional $Y$-tilings by similar arguments.
	Given a $k$-graph $H$ and a set $L$ with $|L|<k$, the link graph $H(L)$ is defined as a $(k-|L|)$-graph with vertex set $V(H)\setminus L$ and edge set $\{e\setminus L: L\subseteq e\in E(H)\}$.
	\begin{thm}\label{tran}
		For all $k \ge 3$, $1 \le d < b \le k-1$, and $n \ge k$,
		\[
		f_d^{n/(2k-b)}(Y_{k, b}, n) \le f_0^{n/(2k-b)}(Y_{k-d, b-d}, n-d).
		\]
	\end{thm}
	
	\begin{proof}
		Let $H$ be an $n$-vertex $k$-graph with $\delta_d(H) \ge f_0^{n/(2k-b)}(Y_{k-d, b-d}, n-d)$ and
		let $H'$ be a $(2k-b)$-graph on $V(H)$, where $E(H') = \mathcal F_{Y_{k, b},H}$.
		By the minimum $d$-degree assumption, for every $d$-set $L$ in $H$, $e(H(L)) \ge \delta_d(H)\ge f_0^{n/(2k-b)}(Y_{k-d, b-d}, n-d)$, 
		then there is a fractional $Y_{k-d, b-d}$-tiling $\omega_1: \mathcal F_{Y_{k-d, b-d}, H(L)} \to [0,1]$ such that for each $v\in V$ we have $\sum_{v\in e\in \mathcal F} \omega_1(Y) \le 1$ and $\sum_{e\in \mathcal F} \omega_1(Y) \ge n/(2k-b)$. Let $H'(L)$ be a $(2k-b-d)$-graph on $V(H)\setminus L$, where $E(H'(L)) = \mathcal F_{Y_{k-d, b-d}, H(L)}$.
		We claim that $\omega_1$ is also a fractional matching of $H'(L)$.
		Indeed, every copy of $Y_{k-d, b-d}$ in $H(L)$ together with $L$ spans a copy of $Y_{k,b}$ in $H$, that is, $\mathcal F_{Y_{k-d, b-d}, H(L)} \subseteq H'(L)$ and thus the claim follows.
		This implies that for every $d$-set $L$, $\nu^*(H'(L))\ge n/(2k-b)$.
		
		We will assume that there is no perfect fractional $Y_{k, b}$-tiling in $H$.
		Then by the definition of $H'$, there is no perfect fractional matching in $H'$.
		We will show that for a particular choice of $L$, $\nu^*(H'(L))<n/(2k-b)$, contradicting the previous paragraph.
		
		%Let an $n$-vertex $(2k-b)$-graph $H'$ satisfy $\nu^*(H') < n/(2k-b)$, that is, have no perfect fractional matching. 
		As $\tau^*(H') = \nu^*(H')$, there is a function $\omega: V \to [0,1]$ such that $\sum_{v\in V} \omega(v) < n/(2k-b)$ and, for every $e \in E(H')$, we have $\sum_{v\in e} \omega(v) \ge 1$. 
		%We replace $H'$ with the $(2k-b)$-graph whose edge set consists of every $(2k-b)$-tuple of vertices on which $\omega$ totals to at least one.
		%Formally, for every weight function $\omega : V \to [0, 1]$ define
		Next, we define a $(2k-b)$-graph
		\[
		H'_{\omega} := \left \{e\in \binom{V}{2k-b}: \sum_{v\in e}\omega(v)\ge1 \right \}.
		\]
		Suppose $L$ is the set of $d$ vertices with the smallest weights in $\omega$. 
		Without loss of generality, we may assume that the $d$ lowest values of $\omega (x)$ are all equal to each other, since otherwise we could replace them by their average. (Obviously, this would not change $\sum_{v\in V} \omega(v)$ nor the set $L$.) 
		Note that the minimum $d$-degree $\delta_d(H'_{\omega})$ is achieved by the $d$-set $L$. 
		Let $H'_{\omega}(L)$ be the neighborhood of $L$ in $H'_{\omega}$, that is, a $(2k-b-d)$-graph on the vertex set $V \setminus L$ and with the edge set
		\[
		\left \{S \in \binom{V-L}{2k-b-d}: S\cup L\in E(H'_{\omega}) \right \}.
		\]
		%Then $|H'(L)| = \delta_d(H'_{\omega})$ and it remains to prove that $\tau^*(H'(L)) < n/(2k-b)$.
		Let $x = \min_{v\in V} \omega(v)$ and observe that $x < 1/(2k-b)$. 
		If $x > 0$, then apply the linear map
		\[
		\omega' = \frac{\omega-x}{1-(2k-b)x}. 
		\]
		Then, still $\sum_{v\in V} \omega'(v) < n/(2k-b)$ and $H'_w = H'_{w'}$.
		Moreover, for every $v\in L$, we have $\omega'(v) = 0$.
		It follows that the function $\omega'$ restricted to the set $V\setminus L$ is a fractional vertex cover of $H'(L)$, so $\nu^*(H'(L))=\tau^*(H'(L))<n/(2k-b)$, which completes the proof.
	\end{proof}
	
	The following lemma allows us to convert an almost perfect fractional $Y_{k,b}$-tiling into an almost perfect integer $Y_{k,b}$-tiling. Its proof is an application of the regularity method and similar to that of Lemma~\ref{pat lem} and Theorem~\ref{edge th2}.
	Here we omit the proof.
	\begin{lemma}\label{ftot}
		Suppose $k \ge 3$ and $1 \le d < b \le k-1$, and $ 0<1/n\ll\varepsilon'<\varepsilon\ll \eta $. Let $ H $ be a $ k $-graph on $ n $ vertices with
		\[
		\delta_{d}(H)\ge	f_d^{n/(2k-b)-\varepsilon'n}(Y_{k, b}, n) +\eta\binom{n}{k-d}.
		\] Then $ H $ contains a $Y_{k,b}$-tiling covering all but at most $\varepsilon n$ vertices.
	\end{lemma}
	
	%Applying  Lemma~\ref{ftot}, we show a minimum $ d $-degree condition for the existence of an almost perfect $Y_{k,2\ell}$-tiling in some cases, which gives new results on $ d $-degree Hamilton $ \ell $-cycle problem.
	The following result is a consequence of Theorem~\ref{tran}, Lemma~\ref{ftot}, Theorem~\ref{ineg} and Theorem~\ref{edge th2}.
	\begin{thm} \label{yth}
		Let $ 0<1/n\ll\varepsilon\ll \eta $.	Suppose $k\ge3$, $1\le d < 2\ell \le k-1$ such that $2k-2\ell \ge (2(2k-2\ell-d)^2+1)(k-d-1)+1$ or suppose $ k $ is odd, $ k\ge7,\ell=(k-1)/2 $  and  $d=k-3$. Let $H$ be a $ k $-graph of order $n$ with \[\delta_{d}(H)\ge \binom{n}{k-d} - \binom{n-\frac{n}{2k-2\ell}}{k-d} + o(n^{k-d})=(1-(1-\frac{1}{2(k-\ell)})^{k-d}+\eta) \binom{n}{k-d}.\] Then $H$ contains a $Y_{k,2\ell}$-tiling covering all but at most $\varepsilon n$ vertices.
		In particular, $ t(k,d,\ell)\le 1-(1-\frac{1}{2(k-\ell)})^{k-d}$.
	\end{thm}
	\begin{proof}
		
	Suppose that we have constants such that  $0<1/n\ll\varepsilon'<  \varepsilon \ll \eta $. 
		Since 
		\[
		\delta_{d}(H)\ge \binom{n}{k-d} - \binom{n-\frac{n}{2k-2\ell}}{k-d} + o(n^{k-d}),\] by Lemma~\ref{ftot}, it suffices to show \[f_{d}^{\frac{n}{2k-2\ell}-\varepsilon' n}(Y_{k,2\ell},n)\le  \binom{n}{k-d}-\binom{n-\frac{n}{2k-2\ell}}{k-d}+o(n^{k-d}).  \tag{6.11}
		\]

		If $ k $ is odd, $ k\ge7,\ell=(k-1)/2 $  and  $d=k-3$, then
		by Theorem~\ref{tran} and Theorem~\ref{edge th2}, we have 
		\[
		f_{k-3}^{\frac{n}{k+1}-\varepsilon' n}(Y_{k,2\ell},n)\le f_{0}^{\frac{n}{k+1}-\varepsilon' n}(Y_{3,2},n-k+3)\le \binom{n}{3}-\binom{n-\frac{n}{k+1}}{3}+o(n^3).
		\] 
		Similarly, if $k\ge3$, $1\le d < 2\ell \le k-1$ such that $2k-2\ell \ge (2(2k-2\ell-d)^2+1)(k-d-1)+1$, then we apply  Theorem~\ref{tran} and Theorem~\ref{ineg}. 
		Thus	
		\[
		f_{d}^{\frac{n}{2k-2\ell}-\varepsilon' n}(Y_{k,2\ell},n) \le f_d^{n/(2k-2\ell)}(Y_{k, 2\ell}, n) \le \binom{n}{k-d} - \binom{n-\frac{n}{2k-2\ell}}{k-d} + o(n^{k-d}).
		\]
		Then we are done in both cases.
		Note that the ``in particular'' part of the theorem follows from the definition of $t(k,d,\ell)$.
	\end{proof}
	
	%\textbf{Remark}: For those choices of $ k,d $ and $ \ell $,	Theorem~\ref{yth} implies $ t(k,d,\ell)\le 1-(1-\frac{1}{2(k-\ell)})^{k-d}  $. Recall that $ 1-(1-\frac{1}{2(k-\ell)})^{k-d}\le h_{d}^{\ell}(k)\le t(k,d,\ell)$, so we get  $ h_{d}^{\ell}(k)=t(k,d,\ell)=1-(1-\frac{1}{2(k-\ell)})^{k-d}$, completing the proof of Theorem~\ref{thm1}. What is more, Theorem~\ref{yth} is best possible up to the error term $ \eta\binom{n}{k-d} $.
	
	Now Theorem~\ref{thm1} follows immediately from  Theorem~\ref{thm}, Proposition~\ref{pr} and Theorem~\ref{yth} .
	
	\begin{figure}[ht]	
		%  \centerline{\small\begin{tikzpicture}[transform canvas={scale=0.60}]
		\centerline{\small\begin{tikzpicture}%[scale=0.6, transform shape]
				\tikzstyle{every rectangle}=[draw=blue!50,fill=blue!20,thick]
				\node[rectangle,draw,minimum height=6mm] (thm64) at (0,8.5)   {Theorem~\ref{edge th1}};
				\node[rectangle,draw,minimum height=6mm] (thm16) at (3.5,8.5)   {Theorem~\ref{edge th2}};
				\node[rectangle,draw,minimum height=6mm] (thm610) at (3.5,9.5)   {Lemma~\ref{ftot}};
				\node[rectangle,draw,minimum height=6mm] (thm63) at (3.5,7.5)   {Theorem~\ref{ineg}};
				\node[rectangle,draw,minimum height=6mm] (thm61) at (3.5,6.5)   {Theorem~\ref{tran}};
				\node[rectangle,draw,minimum height=6mm] (thm611) at (7,8)   {Theorem~\ref{yth}};
				\node[rectangle,draw,minimum height=6mm] (thm15) at (10.5,8)   {Theorem~\ref{thm1}};
				\node[rectangle,draw,minimum height=6mm] (thm13) at (7,9)   {Theorem~\ref{thm}};
				\node[rectangle,draw,minimum height=6mm] (thm19) at (7,7)   {Proposition~\ref{pr}};
				\draw[->] (thm64) -- (thm16);
				\draw[->] (thm610) -- (thm611);
				\draw[->] (thm16) -- (thm611);
				\draw[->] (thm63) -- (thm611);
				\draw[->] (thm61) -- (thm611);
				\draw[->] (thm611) -- (thm15);
				\draw[->] (thm13) -- (thm15);
					\draw[->] (thm19) -- (thm15);
		\end{tikzpicture}}
		\caption{\label{fig:1} A diagram summarizing the flow of the proofs
			of our main results}
	\end{figure}
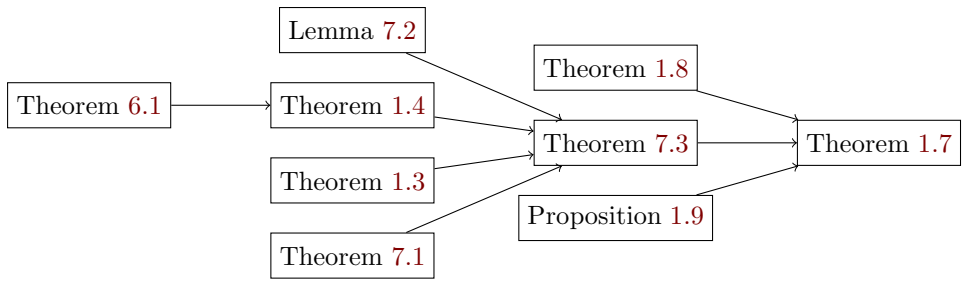
	
	%\section{Concluding remarks}
	%Unfortunately our result on $\ell$-cycle with $\ell=(k-1)/2$ in Theorem~\ref{thm1} does not cover its smallest open case $k=5$, $d=2$ and $\ell=2$, as both our connecting lemma and the absorbing lemma require a minimum 3-degree. 
	%Similar as the proof of Theorem~\ref{thm}, we get the following result for $k=5, \ell=2$.
	%Nevertheless, our proof gives the following result and it is reasonable to conjecture that the minimum 3-degree condition can be removed.
%	\begin{thm}
		%For $\ell\in\{1,2\}$ and any $ \eta>0 $, there exists $n_{0}$ such that $n \ge n_{0}$ and $(5-\ell)\mid n$. Then any $5$-graph $H$ of order $n$ with $\delta_2(H)\ge(t(5,2,\ell)+\eta) \binom{n}{3}$ and $\delta_{\max\{k-\ell,\ell+1\}}(H)\ge\eta \binom{n}{5-\max\{k-\ell,\ell+1\}}$  contains a Hamilton $\ell$-cycle.
		%For any $ \eta>0 $, there exists $n_{0}$ such that $n \ge n_{0}$ and $3\mid n$, then any $5$-graph $H$ of order $n$ with $\delta_2(H)\ge(\frac{91}{216}+\eta) \binom{n}{3}$ and $\delta_{3}(H)\ge\eta \binom{n}{2}$  contains a Hamilton $2$-cycle.
%	\end{thm}
	
	%%%%%%%%%%%%%%%%%%%%%%%%%%%%%%%%%%%%%%%%%%%
	%%%%%%%%%%%%%%% interval case %%%%%%%%%%%%%%%%%%

	%%%%%%%%%%%%%%%%%%%%%%%%%%%%%%%%%%%%%%%%%%%%%%%%%%%%%%%%%%%%%%%%%%%%%%%%%%%%%%%%%%%%%%
	
	\bibliographystyle{abbrv}
	\bibliography{HamiltonCycle-final1}
	
	\begin{appendix}
		\section{Proof of the absorbing path lemma}
		%\subsection{The embedding and extension lemmas}
		%In our proof we will also use the embedding lemma, which guarantees the existence of a copy of a complex $G$ of bounded maximum degree insides a suitable regular complex $H$, where the order of $G$ is allowed to be linear in the order of $H$. In order to state this lemma, we need one more definition.
		We first present an extension lemma, for which we need the following notation.
		Suppose that $H$ is a $(k, s)$-complex with vertex classes $V_1,\dots,V_s$ of size $m$, and also suppose that $G$ is a $(k, s)$-complex with vertex classes $X_1,\dots, X_s$ of size at most $m$. We say that $H$ respects the partition of $G$ if whenever $G$ contains an $i$-edge with vertices in $X_{j_1}, \dots , X_{j_i}$, then there is an $i$-edge of $H$ with vertices in $V_{j_1}, \dots , V_{j_i}$. On the other hand, we say that a labelled copy of $G$ in $H$ is partition-respecting if for $i \in [s]$ the vertices corresponding to those in $X_i$ lie within $V_i$.
		
		%\begin{lemma}(Embedding lemma, \cite{Cooley2009Embeddings}, Theorem 3).\label{embed lem}
		%	Let $\Delta,s,k,r,m_{0}$ be positive integers and let $c,\beta,\deg_{2} ,\dots,\deg_{k},\delta,\delta_{k}$ be positive constants such that $1/\deg_{i} \in N$ for all $i < k$ and
		%	\[1/m_{0} \ll 1/r, \delta \ll \min\{\delta_{k}, \deg_{2},\dots,\deg_{k-1}\} \le \delta_{k} \ll \deg_{k}, 1/s, 1/\Delta .\]
		%	Then the following holds for all integers $m\ge m_{0}$. Suppose that $G'$
		%	is a $(k,s)$-complex of maximum vertex degree at most $ \Delta $ with vertex classes $X_{1}, \dots, X_{s}$ such that $ |X_i|\le cm $ for all $ i=1,\dots,s $. Suppose also that $H$ is a  $(d,\delta_{k},\delta,r)$-regular  $(k,s)$-complex with vertex classes $V_{1},\dots,V_{s}$, all of size $m$, which respects the partition of $G'$. Then $H$ contains a labelled partition-respecting copies	of $G$.	
		%\end{lemma}
		
		Roughly speaking, the extension lemma says that if $G'$ is an induced subcomplex of $G$, and $H$ is suitably regular, then almost all copies of $G'$ in $H$ can be extended to a large number of copies of $G$ in $H$. We write $|G'|_{H}$ for the number of labelled partition-respecting copies of $G'$ in $H$.

		\begin{lemma}[Extension lemma, \cite{Cooley2009Embeddings}, Lemma 5]\label{exten lem}
			Let $k,s,r,b',b'',m_{0}$ be positive integers, where	$b'<b''$, and let $c,\beta,d_{2} ,\dots, d_{k},\delta,\delta_{k}$ be positive constants such that $1/d_{i} \in N$ for all $i < k$ and
			\[1/m_{0} \ll 1/r, \delta \ll c \ll \min\{\delta_{k}, d_{2} ,\dots, d_{k-1}\} \le \delta_{k} \ll \beta, d_{k}, 1/s, 1/b''.\]
			Then the following holds for all integers $m\ge m_{0}$. Suppose that $G'$
			is a $(k,s)$-complex on $b''$ vertices with vertex classes $X_{1}, \dots, X_{s}$ and let $G$ be an induced subcomplex of 
			$G'$ on $b'$ vertices.	Suppose also that $H$ is a  $(d,\delta_{k},\delta,r)$-regular  $(k,s)$-complex with vertex classes $V_{1},\dots,V_{s}$, all of size $m$, which respects the partition of $G'$. Then all but at most $\beta |G|_{H}$ labelled partition-respecting copies of $G$ in $H$ are extendible to at least $cn^{b''-b'}$ labelled partition-respecting copies	of $G'$ in $H$.	
		\end{lemma}
		
		Let $H$ be a $k$-graph, and let $S$ be a $( k-\ell) $-set of $V(H)$. 
		The following absorbing path was constructed and used in \cite{Daniela2010Hamilton}, which we will use as well in our proof.
		%The next result describes the absorber that we will use.
		
		\begin{prop}[\cite{Daniela2010Hamilton}, Proposition 6.1]\label{Prop1} 
			Suppose that $k\ge3$, and that $1\leq \ell \leq k-1$ is such that $(k-\ell)\nmid k$. Then there is a $k$-partite $k$-graph $\mathcal{AP}(k,\ell)$ with the following properties.
			\begin{enumerate}
				\item $|\mathcal{AP}(k,\ell)|\le k^{4}$.
				\item The vertex set of $\mathcal{AP}(k,\ell)$ consists of two disjoint sets S and X with $|S| = k-\ell $.
				\item $\mathcal{AP}(k,\ell)$ contains an $ \ell $-path $P$ with vertex set 
				$X$ and ordered ends $ P^{beg} $ and $ P^{end} $.
				\item \label{prop1_4}$\mathcal{AP}(k,\ell)$ contains an $ \ell $-path $Q$ with vertex set $S\cup X$ and ordered ends $ P^{beg} $ and $ P^{end} $.
				\item \label{prop1_5}No edge of $\mathcal{AP}(k,\ell)$ contains more than one vertex of S.
				\item No vertex class of $\mathcal{AP}(k,\ell)$ contains more than one vertex of S.
			\end{enumerate}
		\end{prop}
		%We fix a $k$-partite $k$-graph $\mathcal{AP}(k,\ell)$ satisfying all properties in Proposition~\ref{Prop1} for the subsequent discussion (as there might be many choices).
		Note that the $k$-graph $\mathcal{AP}(k,\ell)$ may not be unique. 
		For our proof we just fix an arbitrary such $k$-graph satisfying all properties in Proposition~\ref{Prop1}.
		Let $b(k,\ell)$= $|\mathcal{AP}(k,\ell)|-k+\ell$. Then $b(k,\ell)$ is the order of the $ \ell $-path $P$.
		
		Let $c>0$. Recall that a $( k-\ell) $-set $S$ of $V(H)$ is \emph{$c$-good} (otherwise \emph{$c$-bad}) if $H$ contains at least $cn^{b(k,\ell)}$ absorbing paths for $S$, each on  $b(k,\ell)$ vertices. The following lemma shows that for the values of $k$ and $\ell$ that we are interested in, and any small $c$, if $H$ is sufficiently large and has large minimum degree, then almost all $( k-\ell) $-sets of $ V(H) $ are $c$-good. For convenience, let $b:=b(k,\ell)$.
		
		\begin{lemma} \label{lem cbad}
			Suppose that $k\ge3$, $1\leq \ell \leq k-1$ and that $k-\ell\leq d \leq k-1$ is such that $(k-\ell)\nmid k$ and $\frac{1}{n} \ll c \ll \gamma  \ll \mu,k-1$. Let $H$ be a $k$-graph on $n$ vertices such that  $\delta_d(H)\ge \mu \binom{n}{k-d}$. Then at most $\gamma n^{k-\ell}$ sets $S$ of $k-\ell$ vertices of $H$ are $c$-bad.	
		\end{lemma}
		\begin{proof} We choose more constants that satisfy the following hierarchy
			\[\frac{1}{n} \ll \frac{1}{r},\delta \ll c \ll \min \{\delta_{k},\frac{1}{t}\} \ll \delta_{k},\eta \ll d' \ll \gamma\ll \mu'\ll \mu,k-1\]
			and assume that $t!\mid n$. Let $H$ be a $k$-graph on $n$ vertices such that  $\delta_d(H)\ge \mu \binom{n}{k-d}$. Apply Lemma~\ref{reglm} to $H$, and let $V_{1},\dots,V_{a_{1}}$ be the clusters of the partition obtained where $ a_1\le t $. Let $m = \frac{n}{a_{1}}$ be the size of each cluster. Define $R$ as the reduced $ k$-graph on these clusters.
			
			The following claim shows that for almost all $(k-\ell)$-sets $ T $  of $V(H)$, $ T $ is contained in clusters lying  in some  edge of $R$.
			\begin{claim} \label{claim2}
				For all but at most $\gamma n^{k-\ell}/2$  $(k-\ell)$-sets $T=\{v_{1},v_{2},\dots,v_{k-\ell} \} \in [V(H)]^{k-\ell}$, there are $V_{i_{1}},V_{i_{2}},\dots,V_{i_{k}}$ which form an edge in $R$, where $v_{j}\in V_{i_{j}}, 1\le j \le k-\ell$.
			\end{claim}
			\begin{proof}
				First, we show that $\deg_{R}(S)\ge 1 $ holds for all but at most $\gamma a_{1}^{k-\ell}/3$  $(k-\ell)$-sets 
				$S\in (V(R))^{k-\ell}$. Indeed, by Lemma~\ref{deg lem}, all but at most $\theta a_{1}^{d}$ sets $X\in \binom{V(R)}{d}$ satisfy $\deg_{R}(X)\ge \mu \binom{a_{1}}{k-d} $. Suppose that there are at least $O(a_{1}^{k-\ell})$ $(k-\ell)$-sets $S$ of $V(R)$ which are not in any edge of $R$, then the number of $d$-sets $X$ satisfying $\deg_{R}(X)< \mu  \binom{a_{1}}{k-d}$ is more than $O(a_{1}^{d}) $, which is a contradiction.
				
				By the definition of equitability, at most $\eta n^{k-\ell}\ll \gamma n^{k-\ell}/6$ sets $T$ of $k-\ell$ vertices of $H$ do not lie in $\text{Cross}_{k-\ell}$. For those $T$ lying in $\text{Cross}_{k-\ell}$,  there are at most $\gamma {n}^{k-\ell}/3$ such $T$ whose corresponding clusters do not lie in an edge of $ R $.
			\end{proof}
			
			Now we continue the proof.
			Suppose that $V_{i_{1}},V_{i_{2}},\dots,V_{i_{k}}$ are clusters forming an edge of $R$. Let $I$ be the index set $\{i_{1},\dots,i_{k}\}$. By the definition of $ R $, there exists $J\in \text{Cross}_{I}$, such that $H$ is $(\delta_{k},r)$-regular w.r.t. $\hat{P}^{k-1}(J)$. Let $H^{*}$ be the $(k,k)$-complex %\[\bigcup\limits_{i=1}^{k-1}\hat{P}^{i}(J) \cup (E(H)\cap \mathcal{K}(\hat{P}^{k-1}(J)))\] 
			as defined in \eqref{eqcomplex}. Thus $H^{*}$ is  $(\mathbf{d},\delta_{k},\delta,r)$-regular and satisfies the conditions of Lemma~\ref{exten lem}. Moreover, recall that $\mathcal{AP}(k,\ell)^{\le}$ is defined as a $(k,k)$-complex. Recall that $V(\mathcal{AP}(k,\ell))=X\cup S$, then the vertex set $S$ can be viewed as a labelled $(k,k)$-complex without edge of size at least two,  an induced subcomplex of  $\mathcal{AP}(k,\ell)^{\le}$ by \eqref{prop1_4} and \eqref{prop1_5} of Proposition~\ref{Prop1}. Now we can apply Lemma~\ref{exten lem} and conclude that all but at most  $\gamma m^{k-\ell}/2$ ordered sets $S'=\{v_{1},v_{2},\dots,v_{k-\ell} \}$ with  $v_{j}\in V(i_{j})$ can be  extended to at least $cb!n^{b}$ labelled partition-respecting copies of $\mathcal{AP}(k,\ell)$ in $H^{*}$.
			Considering each copy $C$ of $\mathcal{AP}(k,\ell)$, by the structure of  $\mathcal{AP}(k,\ell)$, $C-S'$ is an absorbing path for $S'$. Thus,  $H^{*}$ contains at least $cn^{b}$ absorbing paths on $b$ vertices for $S'$. Thus at most $\gamma m^{k-\ell}/2$ such sets $S'$ are $c$-bad. 
			
			Summing over all sets of $k-\ell$ clusters in $\mathcal P$, there are at most  $\gamma n^{k-\ell}/2$ $c$-bad $(k-\ell)$-sets with no two vertices in the same cluster and the corresponding $(k-\ell)$ clusters are in some edge of $R$.	Combining with Claim~\ref{claim2}, we get that the number of $c$-bad $(k-\ell)$-sets of  $V(H)$ is at most  $\gamma n^{k-\ell}$. 
		\end{proof}

Given Lemma~\ref{lem cbad}, the proof of Lemma~\ref{abs lem} follows verbatim as the proof of \cite[Lemma 6.3]{Daniela2010Hamilton}, after replacing \cite[Corollary 5.4]{Daniela2010Hamilton} with Lemma~\ref{con lem} and \cite[Lemma 6.2]{Daniela2010Hamilton} with  Lemma~\ref{lem cbad}.
 %Following the proof in \cite{Daniela2010Hamilton} verbatim after replacing Lemma 6.3 with Lemma~\ref{abs lem} and Lemma 6.2 with  Lemma~\ref{lem cbad}, we get Lemma~\ref{abs lem}.

	%	We are now in a position to prove Lemma~\ref{abs lem}. It states that for any	positive $c$, if $H$ is a sufficiently large $k$-graph with large minimum $d$ degree, then we can find an $\ell$-path in $H$ which contains a small proportion of the vertices of $H$, and includes all vertices of $H$ which lie in many $c$-bad $(k-\ell)$-sets and can absorb any small collection of $c$-good $(k-\ell)$-sets of vertices of $H$.

	\section{Proof of \eqref{6.10}}
Rewriting all terms above with $M_1$ and $M_2$ in \eqref{eq1_9} with  $M_1=4m_1$, $M_2=3m_2$ and $|U|=3n/7=(3/4)(M_1+M_2)$, we get
	\begin{align*}	
	e(H)&\le \frac{M_1}{4}\binom{\frac{3}{4}(M_1+M_2)}{2}+\frac{21}{4}\binom{\frac{1}{4}M_1}{2}(M_1+M_2)+\frac{9}{4}\binom{\frac{1}{3}M_2}{2}(M_1+M_2)+\frac{3}{8}M_1M_2(M_1+M_2)\\&+37\binom{\frac{1}{4}M_1}{3}+19\binom{\frac{1}{3}M_2}{3}+10\binom{\frac{1}{4}M_1}{2}M_2+6\binom{\frac{1}{3}M_2}{2}M_1+O(n^2)\\&\le\frac{9}{64}M_1\binom{M_1+M_2}{2}+\left[\frac{21}{64}\binom{M_1}{2}(M_1+M_2)+\frac{1}{4}\binom{M_2}{2}(M_1+M_2)+\frac{3}{8}M_1M_2(M_1+M_2)\right]\\&+\left[\frac{37}{64} \binom{M_1}{3}+\frac{19}{27}\binom{M_2}{3}+\frac{5}{8}\binom{M_1}{2}M_2+\frac{2}{3}\binom{M_2}{2}M_1\right]+O(n^2)\\&\le\frac{27}{64}\binom{M_1+M_2}{3}+\frac{21}{64\times 2}(M_1+M_2)^3+\frac{37}{64}\binom{M_1+M_2}{3}+O(n^2)\\&\le \frac{127}{64\times 6}(M_1+M_2)^3 + O(n^2)= \frac{127}{7^3\times 6}n^3+O(n^2) \\&= \binom{n}{3}-\binom{n- n/7}{3}+O(n^2), 
\end{align*}
where we used $M_1+M_2=4n/7$, $(M_1+M_2)^2=M_{1}^{2}+M_{2}^{2}+2M_{1}M_{2}$, and $ \binom{M_1+M_2}{3}=\binom{M_1}{3}+\binom{M_2}{3}+\binom{M_1}{2}M_2+\binom{M_2}{2}M_1 $.
\end{appendix}

\end{document}